\pdfoutput=1

\documentclass[12pt]{article}
\usepackage{modnatbib}
\usepackage{amssymb,amsmath,amsthm,epsfig,a4,verbatim,pstricks,url}
\usepackage[ruled]{algorithm2e}
\usepackage{ifpdf}  
\newtheorem{thm}{\sc Theorem.}[section]
\newtheorem{lem}[thm]{\sc Lemma.}
\newtheorem{rem}[thm]{\sc Remark.}

\newenvironment{AMS}%
{{\upshape\bfseries AMS subject classifications. }\ignorespaces}{}
\newenvironment{keywords}{{\upshape\bfseries Key words. }\ignorespaces}{}

\newcommand{\Rplus}{{\mathbb R}_{>0}}
\newcommand{\Rgeq}{{\mathbb R}_{\geq 0}}

\newcommand{\R}{{\mathbb R}}
\newcommand{\Ds}{\mathcal{D}_\Gamma}

\newcommand{\tr}{\operatorname{tr}}
\newcommand{\spa}{\operatorname{span}}

\newcommand{\diam}{\operatorname{diam}}
\newcommand{\vol}{\mathcal{L}^d} 

\newcommand{\GT}{{\mathcal{G}_T}}
\newcommand{\GhT}{{\mathcal{G}^h_T}}

\newcommand{\dH}[1]{\;{\rm d}{\cal H}^{#1}} 
\newcommand{\dL}[1]{\;{\rm d}{\cal L}^{#1}} 

\newcommand{\bigchi}{\ensuremath{\mathrm{\mathcal{X}}}}
\newcommand{\charfcn}[1]{\bigchi_{#1}} 
\newcommand{\Domain}{\Omega}

\newcommand{\Vhz}{\underline{V}(\Gamma^0)}

\newcommand{\Whp}{W(\Gamma^{m+1})}
\newcommand{\Vh}{\underline{V}(\Gamma^m)}
\newcommand{\Vhm}{\underline{V}(\Gamma^{m-1})}

\newcommand{\Wh}{W(\Gamma^m)}
\newcommand{\Vht}{\underline{V}(\Gamma^h(t))}
\newcommand{\Wht}{W(\Gamma^h(t))}
\newcommand{\Whtz}{W(\Gamma^h(0))}
\newcommand{\uspace}{\mathbb{U}}
\newcommand{\utimespace}{\mathbb{V}}
\newcommand{\psispace}{\mathbb{S}}
\newcommand{\pspace}{\mathbb{P}}
\newcommand{\sigmaO}{o}
\newcommand{\nabs}{\nabla_{\!s}}
\newcommand{\Id}{\rm Id}
\newcommand{\id}{\rm id}
\newcommand{\ddt}{\frac{\rm d}{{\rm d}t}}

\newcommand{\matpartu}{\partial_t^\bullet}

\newcommand{\matpartx}{\partial_t^\circ}
\newcommand{\matpartxh}{\partial_t^{\circ,h}}

\DeclareMathOperator{\esssup}{ess\,sup}

\newcommand{\unitn}{\vec{\rm n}}
\newcommand{\unitt}{\vec{\rm t}}
\newcommand{\ek}{e}

\newcommand{\errorXx}{\|\vec X - \vec{x}\|_{L^\infty}}

\newcommand{\errorUu}[1]{\|\vec U - \vec I^h_{#1}\,\vec u\|_{L^\infty}}

\newcommand{\LerrorPpc}{\|P_c - p_c\|_{L^2}}

\newcommand{\LerrorLl}{\|\theta^h - \theta\|_{L^2}}
\newcommand{\XFEMGAMMA}{XFEM$_\Gamma$}
\def\epsilon{\varepsilon} 
\newcommand{\mat}[1]{\underline{\underline{#1}}\rule{0pt}{0pt}}

\def\hat{\widehat}

\def\vL{L\kern-0.08cm\char39}
\begin{document}
\title{
Stable Numerical Approximation of Two-Phase Flow with
a Boussinesq--Scriven Surface Fluid}
\author{John W. Barrett\footnotemark[2] \and 
        Harald Garcke\footnotemark[3]\ \and 
        Robert N\"urnberg\footnotemark[2]}

\renewcommand{\thefootnote}{\fnsymbol{footnote}}
\footnotetext[2]{Department of Mathematics, 
Imperial College London, London, SW7 2AZ, UK}
\footnotetext[3]{Fakult{\"a}t f{\"u}r Mathematik, Universit{\"a}t Regensburg, 
93040 Regensburg, Germany}

\date{}

\maketitle
\begin{abstract}
We 
consider two-phase
Navier--Stokes flow with a Boussinesq--Scriven surface fluid. In such a
fluid the rheological behaviour at the interface includes surface viscosity effects, 
in addition to the classical surface tension effects. 
We introduce and analyze parametric finite element approximations, and show, in
particular, stability results for semi-discrete versions of the methods,
by demonstrating that a free energy inequality also holds on the
discrete level. We perform several numerical simulations for various
scenarios in two and three dimensions, which illustrate the effects
of the surface viscosity. 
\end{abstract} 



\renewcommand{\thefootnote}{\arabic{footnote}}

\section{Introduction}\label{sec:intro}

Fluid interfaces typically have their own dynamic properties and, in
particular, a surface stress tensor, involving interfacial shear and
dilatational viscosities, can have a significant effect on the
dynamics. Surface tension effects on a fluid interface are well-known,
and in this case the stresses acting on the interface are balanced by
the surface tension and the curvature of the interface. However, 
in systems with high surface area to volume ratios, such as
micro bubbles, blood cells, dispersions of vesicles and emulsions, 
the dynamics of the system are
also highly influenced by the dynamics on
the interface. Hence one can argue, see e.g.\ \cite{Sagis11}, 
that a more detailed study of the stress-deformation behaviour of interfaces is
highly relevant for many disciplines, e.g.\ interface science,
biophysics, pharmaceutical science, polymer physics, food science
and engineering. 

If only surface tension effects are taken into account in the surface
stress tensor $\mat\sigma_\Gamma$, one obtains the form
\begin{equation}\label{surften1}
\mat\sigma_\Gamma = \gamma\,\mat{\mathcal{P}}_\Gamma,
\end{equation}
where $\mat{\mathcal{P}}_\Gamma$ is the projection to the tangent
space of the interfacial surface $\Gamma$, 
and $\gamma$ is the surface tension,
which in the simplest case is constant. In this case the stress
balance on the interface is given as 
\begin{equation}\label{surften2}
-\nabs\,.\, \mat\sigma_\Gamma = [\mat\sigma\, \vec\nu]^+_-
\qquad \Leftrightarrow\qquad  -\gamma\,\varkappa\,\vec\nu = [\mat\sigma\,\vec\nu]^+_-\,.
\end{equation}
Here $\nabs\,.$ is the surface divergence, $[\,\cdot\, ]^+_-$ denotes the
jump of a quantity across the interface, $\mat \sigma$ denotes the bulk fluid stress tensor, 
$\vec\nu$ is the unit normal
to the interface, and $\varkappa$ is the mean curvature, we refer to
Section~\ref{sec:1} for the precise definitions. Equation (\ref{surften1})
expresses the momentum balance at a dividing surface, see
e.g.\ \cite{SlatterySO07}. When the surface tension
coefficient in (\ref{surften1}) is not constant, which is the case
when a surface active agent has an effect on the surface tension, the
stress balance (\ref{surften2}) becomes
\begin{equation*}
-\gamma\,\varkappa\,\vec\nu - \nabs\, \gamma = [\mat\sigma \,\vec\nu]^+_-\,,
\end{equation*}
which in turn gives rise to discontinuities in the tangential
components of the bulk fluid stresses at the surface. However, in general
other interfacial properties, such as the resistance of an interface
to deformation, have to be taken into account. This is
particularly relevant in cases, where the interface is not clean. 
For systems with species that adsorb at the interface,
like emulsions or foams stabilized by surfactants and proteins, it is
expected that the surface stresses have a pronounced effect on the
dynamics. Therefore
the interest in surface rheology has increased significantly in the last
twenty years, see e.g.\ \cite{SlatterySO07}.
One key difference between bulk and surface rheology is that in
the bulk phase one usually assumes incompressibility, whereas this
assumption often does not hold for interfaces -- biomembranes are
a notable exception, see e.g.\ \cite{ArroyoS07}. The general
momentum balance, which generalizes (\ref{surften2}), now, in addition,
has to take the surface momentum and a generalized stress tensor,
involving surface shear and dilatational viscosities, into account. The
overall momentum balance on the surface then reads as 
\begin{equation}\label{NSS1}
\matpartu\,(\rho_\Gamma\,\vec u) + (\nabs\,.\,\vec u)\,\rho_\Gamma\,\vec u
- \nabs\,.\,\mat\sigma_\Gamma =
[\mat\sigma\,\vec\nu]_-^+ \,,
\end{equation}
with the surface stress tensor now given by
\begin{equation*}
\mat\sigma_\Gamma = 2\,\mu_\Gamma\,\mat D_s (\vec u) 
+ ( \lambda_\Gamma \,\nabs\,.\,\vec u + \gamma)
\,\mat{\mathcal{P}}_\Gamma \,.
\end{equation*}
Here $\rho_\Gamma$ is the surface material density, $\vec u$ is the fluid velocity,
$\matpartu$ is the
material derivative on the interface, $\mu_\Gamma$ is the surface
shear viscosity, $\mu_\Gamma+\lambda_\Gamma$ is the surface
dilatational viscosity and $\mat D_s (\vec u) =
\tfrac12\,\mat{\mathcal{P}}_\Gamma\,(\nabs\,\vec u+(\nabs\,\vec u)^T)\,
\mat{\mathcal{P}}_\Gamma$ is the interfacial rate-of-deformation tensor. This
tensor describes how the lengths of curves on the surface change, and
how the angles between intersecting curves change with the flow. 

Although, at first glance, the surface momentum
equation looks very similar to the bulk momentum equation, it turns out that
new geometric quantities appear. For example, we note that in
$\nabs\,.\,\vec u$ we take the divergence of a non-tangential vector
field, which for alternative formulations of (\ref{NSS1}) would
lead to a curvature term. In particular, 
splitting $\vec u$ into its normal part $(\vec u \,.\,\vec \nu)\,\vec\nu$
and its tangential part 
$\vec u_{\rm tan} = \vec u - (\vec u \,.\,\vec \nu)\,\vec\nu$ gives
$\nabs\,.\,\vec u = \nabs\,.\,\vec u_{\rm tan} - 
\vec u \,.\,\vec \nu\,\varkappa$; see e.g.\ \cite{ArroyoS07}.

For the surface material density $\rho_\Gamma$ the mass balance law 
\begin{equation}\label{NSS2}
\matpartu\,\rho_\Gamma + (\nabs\,.\,\vec u)\,\rho_\Gamma = 0
\end{equation}
holds on the surface. Hence (\ref{NSS1}) and (\ref{NSS2}) is a compressible
Navier--Stokes system on an evolving surface
with a forcing $[\mat\sigma \,\vec\nu]^+_-$ arising
from bulk stresses. In this paper we also allow for an insoluble
surface active agent (surfactant), whose concentration we denote by $\psi$. We then
require that the advection-diffusion equation
\begin{equation}\label{NSS3}
\matpartu\,\psi + (\nabs\,.\,\vec u)\,\psi -
\nabs\,.\,(\mathcal{D}_\Gamma\,\nabs\,\psi)=0\,,
\end{equation}
with a diffusion coefficient $\mathcal{D}_\Gamma$, has to hold on the 
interface. In this case, 
the surface viscosities $\lambda_\Gamma$, $\mu_\Gamma$ and the
surface tension $\gamma$ may depend on $\psi$. The system
(\ref{NSS1})--(\ref{NSS3}) then has to be coupled to the classical incompressible Navier--Stokes
system in the bulk, and we refer to Section~\ref{sec:1} for the details. 

The first ideas, which later lead to the surface fluid model discussed
above, are due to \cite{Boussinesq13}, and the approach of Boussinesq
was later generalized to
arbitrary moving and deforming surfaces by \cite{Scriven60}. Hence,
one speaks of a Boussinesq--Scriven surface fluid, and we refer to the
book \cite{SlatterySO07} for more details on the physics of the
model and for experiments on Boussinesq--Scriven surface fluids.

The mathematical literature on models involving Boussinesq--Scriven surface 
fluids
is very sparse. We refer to \cite{BotheP10}, who initiated 
the rigorous mathematical study of
two-phase flows with surface viscosity in the case $\rho_\Gamma = 0$,
i.e.\ when 
no separate mass balance is considered. To the best knowledge of the
authors, only the paper by \cite{ReuskenZ13} contains numerical
simulations of a two-phase flow including a Boussinesq--Scriven surface
fluid. Also in that paper the surface material density was set to be zero
and no surfactants were considered. It is the goal of this paper to
introduce a stable finite element method for two-phase flow with a
Boussinesq--Scriven interface stress tensor, which allows for a surface
material density and an insoluble surfactant. 
Besides showing stability results, we
also present numerical simulations in two and three dimensions, which
show different phenomena arising from the surface viscosity
effects. 

Let us state the main features of the topics studied in this paper. 

\begin{itemize}
\item Our approach is based on a parametric finite element method for the 
numerical approximation of the interface. Such an approach, in the context
of a purely geometric evolution of the interface, 
was introduced by \cite{Dziuk91}, see also the review article 
\cite{DeckelnickDE05}. 
We also use the techniques of \cite{DziukE13} for the approximation
of partial differential equations on surfaces.   
  
\item For one variant of our introduced approximations,
based on the present authors' work, see 
\cite{triplej,gflows3d,spurious,fluidfbp},
the parameterization of the evolving interface has good mesh properties
  and, in contrast to other parametric approaches, no remeshing is
  needed in practice. 

\item A suitable variational formulation of the complex conditions at
  the free boundary is introduced, which allows one to show stability of
  semi-discrete (discrete in space, continuous in time) versions of
  the schemes. This extends the present authors' work on the 
  stable numerical approximation of two-phase 
  flow with insoluble surfactant, see \cite{tpfs}, by including surface 
  viscosity effects and a surface material density. 

\item Fully discrete finite element approximations are introduced,
  which lead to linear systems of equation at each time step. In particular,
existence and uniqueness of the discrete solutions 
can be shown. 
If no surface material density is present, then stability can be shown also 
  for these fully discrete variants. 

\item Conservation properties and non-negativity properties of the surface material density 
and the surfactant can be
  shown for the discretized systems.

\item We present several numerical simulations in two and three space
  dimensions, which demonstrate the convergence of the scheme and illustrate
  several effects of surface viscosity. 
For example, in a shearing experiment one 
  observes that bubbles with higher surface viscosities are
  much less elongated. 
\end{itemize}

The study of numerical methods for two-phase flows is a very active
area, and the available numerical approaches can be broadly grouped
into three different categories: parametric front tracking methods, such as the
approximations presented in this paper, level set methods and phase field
methods; see the introduction in \cite{fluidfbp}. 
For more details, and for further background information on the 
various approaches, we refer, for example, to
\cite{HirtN81,Bansch01,Tryggvason_etal01,
LaiTH08,SussmanO09,GanesanT09a,GrossR11,ChengF12,JemsionLSSMMW13}.
We remark that only \cite{ReuskenZ13} have 
considered the case of a Boussinesq--Scriven surface fluid
numerically. 

The outline of the paper is as follows. In Section~\ref{sec:1} we give
a mathematical formulation of the Navier--Stokes two-phase problem for
a Boussinesq--Scriven surface fluid. Section~\ref{sec:2} states
two semi-discrete approximations of the problem together with several
analytical results such as stability, and conservation 
and non-negativity properties of the approximations to the surface material density and 
the surfactant concentration. In
Section~\ref{sec:3} the corresponding fully discrete approximations are introduced. 
Section~\ref{sec:5} discusses some issues
concerning the practical implementation of the method,
in particular, the assembly of the bulk-interface cross terms. 
Finally, in Section~\ref{sec:6} several numerical
computations are presented. 

\setcounter{equation}{0}
\section{Mathematical setting} \label{sec:1}

Let $\Domain\subset\mathbb{R}^d$ be a given domain,  where $d=2$ or $d=3$. 
We now seek a time dependent interface $(\Gamma(t))_{t\in[0,T]}$,
$\Gamma(t)\subset\Domain$, 
which for all $t\in[0,T]$ separates
$\Domain$ into a domain $\Omega_+(t)$, occupied by one phase,
and a domain $\Omega_{-}(t):=\Domain\setminus\overline\Omega_+(t)$, 
which is occupied by the other phase. Here the phases could represent
two different liquids, or a liquid and a gas. Common examples are oil/water
or water/air interfaces.
See Figure~\ref{fig:sketch} for an illustration.
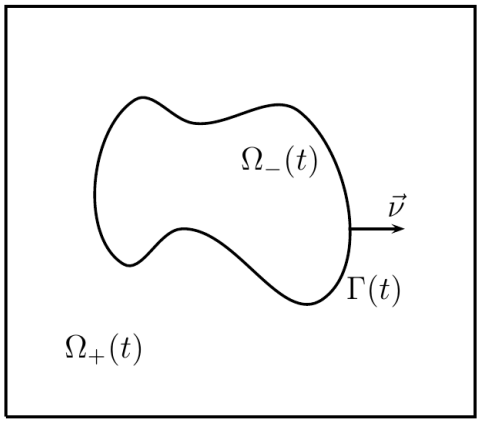
\begin{figure}
\begin{center}
\ifpdf
\includegraphics[angle=0,width=0.4\textwidth]{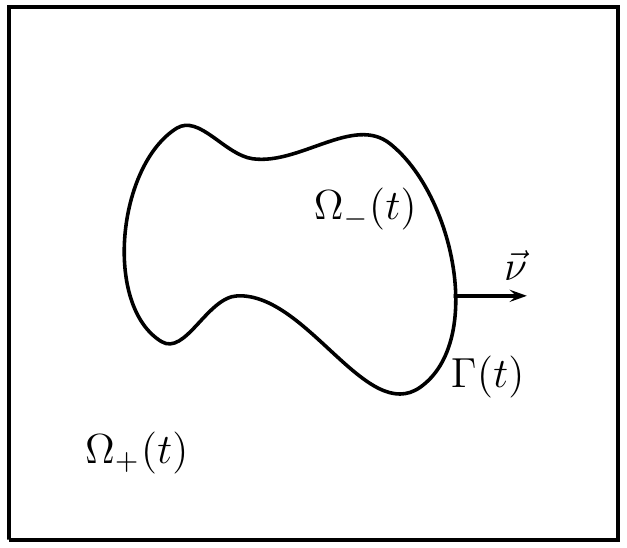}
\else
\unitlength15mm
\psset{unit=\unitlength,linewidth=1pt}
\begin{picture}(4,3.5)(0,0)
\psline[linestyle=solid]{-}(0,0)(4,0)(4,3.5)(0,3.5)(0,0)
\psccurve[showpoints=false,linestyle=solid] 
 (1,1.3)(1.5,1.6)(2.7,1.0)(2.5,2.6)(1.6,2.5)(1.1,2.7)
\psline[linestyle=solid]{->}(2.92,1.6)(3.4,1.6)
\put(3.25,1.7){{\black $\vec\nu$}}
\put(2.9,1.0){{$\Gamma(t)$}}
\put(2,2.1){{$\Omega_-(t)$}}
\put(0.5,0.5){{$\Omega_+(t)$}}
\end{picture}
\fi
\end{center}
\caption{The domain $\Omega$ in the case $d=2$.}
\label{fig:sketch}
\end{figure}%
For later use, we assume that
$(\Gamma(t))_{t\in [0,T]}$ 
is a sufficiently smooth evolving hypersurface without boundary that is
parameterized by $\vec x(\cdot,t):\Upsilon\to\R^d$,
where $\Upsilon\subset \R^d$ is a given reference manifold, i.e.\
$\Gamma(t) = \vec x(\Upsilon,t)$. Then
\begin{equation} \label{eq:V}
\vec{\mathcal{V}}(\vec z, t) := \vec x_t(\vec q, t)
\qquad \forall\ \vec z = \vec x(\vec q,t) \in \Gamma(t)
\end{equation}
defines the velocity of $\Gamma(t)$, and
$\vec{\mathcal{V}} \,.\,\vec{\nu}$ is
the normal velocity of the evolving hypersurface $\Gamma(t)$,
where $\vec\nu(t)$ is the unit normal on $\Gamma(t)$ pointing into 
$\Omega_+(t)$.
Moreover, we define the space-time surface
\begin{equation} \label{eq:GT}
\GT := \bigcup_{t \in [0,T]} \Gamma(t) \times \{t\}\,.
\end{equation}

Let $\rho(t) = \rho_+\,\charfcn{\Omega_+(t)} + \rho_-\,\charfcn{\Omega_-(t)}$,
with $\rho_\pm \in \Rplus$, denote the fluid densities, where here and
throughout $\charfcn{\mathcal{A}}$ defines the characteristic function for a
set $\mathcal{A}$.
Denoting by $\vec u : \Omega \times [0, T] \to \R^d$ the fluid velocity,
by $\mat\sigma : \Omega \times [0,T] \to \R^{d \times d}$ the stress tensor,
and by $\vec f : \Omega \times [0, T] \to \R^d$ a possible forcing,
the incompressible Navier--Stokes equations in the two phases are given by
\begin{subequations}
\begin{alignat}{2}
\rho\,(\vec u_t + (\vec u \,.\,\nabla)\,\vec u)
- \nabla\,.\,\mat\sigma & = \vec f := \rho\,\vec f_1 + \vec f_2
\qquad &&\mbox{in } 
\Omega_\pm(t)\,, \label{eq:NSa} \\
\nabla\,.\,\vec u & = 0 \qquad &&\mbox{in } \Omega_\pm(t)\,, \label{eq:NSb} \\
[\vec u]_-^+ & = \vec 0 \qquad &&\mbox{on } \Gamma(t)\,, \label{eq:1a} \\
\vec u & = \vec 0 \qquad &&\mbox{on } \partial_1\Omega\,, \label{eq:NSc} \\
\vec u \,.\,\unitn = 0\,,\quad
\mat\sigma\,\unitn \,.\,\unitt& = 0 
\quad\forall\ \unitt \in \{\unitn\}^\perp 
\qquad &&\mbox{on } \partial_2\Omega\,, 
\label{eq:NSd} 
\end{alignat}
\end{subequations}
where $\partial\Domain = \partial_1\Omega \cup
\partial_2\Omega$, with $\partial_1\Omega \cap \partial_2\Omega
=\emptyset$, denotes the boundary of $\Domain$ with outer unit normal $\unitn$
and $\{\unitn\}^\perp := \{ \unitt \in \R^d : \unitt \,.\,\unitn = 0\}$.
Hence (\ref{eq:NSc}) prescribes a no-slip condition on 
$\partial_1\Omega$, while (\ref{eq:NSd}) prescribes a free-slip condition on 
$\partial_2\Omega$. 
As usual, $[\vec u]_-^+ := \vec u_+ - \vec u_-$ 
denotes the jump in velocity across the interface
$\Gamma(t)$, where here and throughout we employ the shorthand notation
$\vec g_\pm := \vec g\!\mid_{\Omega_\pm(t)}$ for a function 
$\vec g : \Omega \times [0,T] \to \R^d$; and similarly for scalar and
matrix-valued functions.
In addition, the stress tensor in (\ref{eq:NSa}) is defined by
\begin{equation} \label{eq:sigma}
\mat\sigma = \mu \,(\nabla\,\vec u + (\nabla\,\vec u)^T) - p\,\mat\Id
= 2\,\mu\, \mat D(\vec u)-p\,\mat\Id\,,
\end{equation}
where $\mat\Id \in \R^{d \times d}$ denotes the identity matrix,
$\mat D(\vec u):=\frac12\, (\nabla\,\vec u+(\nabla\,\vec u)^T)$ is the
rate-of-deformation tensor, with $\nabla\,\vec u = \left( \partial_{x_j}\,u_i
\right)_{i,j=1}^d$. 
As usual, $\nabla\,.\, \mat{A} \in {\mathbb R}^d$ 
with $[\nabla \,.\,\mat{A}]_l = \nabla\,.\,\vec{A}_l$, $l=1\to d$,  
for $\mat{A}^T=[ \vec{A}_1 \ldots \vec{A}_d] \in \mathbb R^{d \times d}$.
Moreover,
$p : \Omega \times [0, T] \to \R$ is the pressure and
$\mu(t) = \mu_+\,\charfcn{\Omega_+(t)} + \mu_-\,\charfcn{\Omega_-(t)}$,
with $\mu_\pm \in \R_{>0}$, denotes the dynamic viscosities in the two 
phases.

Let $\rho_\Gamma(\cdot,t) : \Gamma(t) \to \Rgeq$ denote the surface material 
density.
Then on the free surface 
the following conditions need to hold:
\begin{subequations}
\begin{alignat}{2}
\matpartu\,\rho_\Gamma + (\nabs\,.\,\vec u)\,\rho_\Gamma & = 0 
\qquad &&\mbox{on } \Gamma(t)\,,\label{eq:1c}\\
\matpartu\,(\rho_\Gamma\,\vec u) + (\nabs\,.\,\vec u)\,\rho_\Gamma\,\vec u
- \nabs\,.\,\mat\sigma_\Gamma & =
[\mat\sigma\,\vec\nu]_-^+ 
\qquad &&\mbox{on } \Gamma(t)\,, \label{eq:1b} \\ 
\vec{\mathcal{V}}\,.\,\vec\nu &= 
\vec u\,.\,\vec \nu \qquad &&\mbox{on } \Gamma(t)\,, 
\label{eq:1d} 
\end{alignat}
\end{subequations}
see e.g.\ \cite{BotheP10} and \citet[p.\ 18--19]{GrossR11}.
Here $[\mat\sigma\,\vec\nu]_-^+ := \mat\sigma_+\,\vec\nu - \mat\sigma_-\,\vec\nu$
denotes the jump in normal stress across 
$\Gamma(t)$, $\nabs\,.\,$ denotes the surface divergence on $\Gamma(t)$, 
$\mat\sigma_\Gamma$ is the surface stress tensor and
\begin{equation} \label{eq:matpartu}
\matpartu\, \zeta = \zeta_t + (\vec u \,.\,\nabla)\,\zeta 
= \zeta_t + \vec u \,.\,\nabla\,\zeta 
\qquad \mbox{on } \Gamma(t)
\end{equation}
denotes the material time derivative of $\zeta : \GT \to \R$,
and similarly for $\vec\zeta : \GT \to \R^d$
, i.e.\ $\matpartu\, \vec\zeta = \vec\zeta_t + (\vec u \,.\,\nabla)\,\vec\zeta 
= \vec\zeta_t + (\nabla\,\vec\zeta)\,\vec u$. 
We set $H^1(\GT) := \{ \zeta \in L^2(\GT) : 
\nabs\,\zeta \in L^2(\GT), \matpartu\,\zeta \in L^2(\GT) \}$.
We stress
that the derivative in (\ref{eq:matpartu}) is well-defined, and depends only on
the values of $\zeta$ on $\GT$, even though $\zeta_t$ and
$\nabla\,\zeta$ do not make sense separately; see e.g.\
\citet[p.\ 324]{DziukE13}.
The surface stress tensor is defined by
\begin{equation} \label{eq:sigmaG}
\mat\sigma_\Gamma = 2\,\mu_\Gamma(\psi)\,\mat D_s (\vec u) 
+ ( \lambda_\Gamma(\psi) \,\nabs\,.\,\vec u + \gamma(\psi))
\,\mat{\mathcal{P}}_\Gamma \,,
\end{equation}
where $\mu_\Gamma \in C(\R, \Rgeq)$
is the interfacial shear viscosity, and
$\lambda_\Gamma \in C(\R)$ 
is the second interfacial viscosity coefficient satisfying
\begin{equation} \label{eq:mulambda}
\lambda_\Gamma(r) + \tfrac2{d-1}\, \mu_\Gamma(r) 
\geq 0 
\qquad \forall\ r \in \R\,.
\end{equation} 
In the special case that 
\begin{equation} \label{eq:muconst}
\mu_\Gamma(r) = \overline\mu_\Gamma \in \Rgeq
\quad\text{and}\quad
\lambda_\Gamma(r) = \overline\lambda_\Gamma \in \R
\qquad \forall\ r \in \R\,,
\end{equation}
the constants $\overline\lambda_\Gamma$ and $\overline\mu_\Gamma$ are also 
called the first and second surface Lam\'e constants, respectively.
In addition, $\gamma \in C^{1}([0,\psi_\infty))$, 
with $\psi_\infty > 0$ and
\begin{equation} \label{eq:gammaprime}
\gamma'(r) \leq 0 \qquad \forall\ r \in [0,\psi_\infty)\,,
\end{equation}
denotes the surface tension.
The interfacial viscosities and the surface tension depend on 
the surfactant concentration $\psi : \GT \to [0,\psi_\infty)$, 
recall (\ref{eq:GT}).
In addition,
\begin{subequations}
\begin{equation} \label{eq:Ps}
\mat{\mathcal{P}}_\Gamma = \mat\Id - \vec \nu \otimes \vec \nu
\end{equation}
is the tangential projection at $\Gamma(t)$, 
and
\begin{equation} \label{eq:Ds}
\mat D_s(\vec u) = \tfrac12\,\mat{\mathcal{P}}_\Gamma\,(\nabs\,\vec u + (\nabs\,\vec u)^T)\,
\mat{\mathcal{P}}_\Gamma
\end{equation}
\end{subequations}
is the interfacial rate-of-deformation tensor,
where $\nabs = \mat{\mathcal{P}}_\Gamma \,\nabla = (\partial_{s_1}, \ldots, \partial_{s_d})$ 
denotes the surface gradient on $\Gamma(t)$, and
$\nabs\,\vec u = \left( \partial_{s_j}\, u_i \right)_{i,j=1}^d$.

The surfactant transport (with diffusion) on
$\Gamma(t)$ is then given by
\begin{equation} \label{eq:1surf}
\matpartu\,\psi + (\nabs\,.\,\vec u)\,\psi - \nabs\,.\,(\Ds\,\nabs\,\psi) = 0
\qquad \mbox{on } \Gamma(t)\,,
\end{equation}
where $\Ds \geq 0$ is a diffusion coefficient. 
The system (\ref{eq:NSa}--e), (\ref{eq:sigma}), (\ref{eq:1c}--c),
(\ref{eq:sigmaG}), (\ref{eq:1surf}) is closed with the initial conditions
\begin{equation} \label{eq:1init}
\Gamma(0) = \Gamma_0 \,, \quad 
\rho_\Gamma(\cdot, 0) = \rho_{\Gamma,0} \quad \mbox{on } \Gamma_0\,,\quad
\psi(\cdot, 0) = \psi_0 \quad \mbox{on } \Gamma_0\,,\quad
\vec u(\cdot,0) = \vec u_0 \qquad \mbox{in } \Omega\,,
\end{equation}
where $\Gamma_0 \subset \Omega$, 
$\rho_{\Gamma,0} : \Gamma_0 \to \Rgeq$,
$\psi_0 : \Gamma_0 \to [0,\psi_\infty)$ and $\vec u_0 : \Omega \to \R^d$
are given initial data. 

With a view towards substituting (\ref{eq:sigmaG}) into (\ref{eq:1b}), we
observe that
\begin{align}
\nabs\,.\,\mat\sigma_\Gamma & =
2\,\mu_\Gamma(\psi)\,\nabs\,.\,\mat D_s (\vec u) + 
\nabs\,.\left[ ( \lambda_\Gamma(\psi) 
\,\nabs\,.\,\vec u + \gamma(\psi))\,\mat{\mathcal{P}}_\Gamma \right] \nonumber \\ &
= 2\,\mu_\Gamma(\psi)\,\nabs\,.\,\mat D_s (\vec u) 
+ \nabs\,.\left[
 \lambda_\Gamma(\psi) \,(\nabs\,.\,\vec u)\,\mat{\mathcal{P}}_\Gamma\right]
+ \gamma(\psi)\,\varkappa\,\vec\nu + \nabs\,\gamma(\psi) \,,
\label{eq:nabssigma}
\end{align}
where
$\nabs\,.\, \mat{A} \in {\mathbb R}^d$ 
with $[\nabs \,.\,\mat{A}]_l = \nabs\,.\,\vec{A}_l$, $l=1\to d$,  
for $\mat{A}^T=[ \vec{A}_1 \ldots \vec{A}_d] \in \mathbb R^{d \times d}$,
and where we have noted that
$\nabs\,.\,\vec \nu= \nabla\,.\,\vec \nu = - \varkappa$ implies that
\begin{equation*}
\nabs\,.\,\mat{\mathcal{P}}_\Gamma = \varkappa\, \vec\nu\,.
\end{equation*}
Here $\varkappa$ denotes the mean curvature of $\Gamma(t)$, i.e.\ the sum of
the principal curvatures of $\Gamma(t)$, where we have adopted the sign
convention that $\varkappa$ is negative where $\Omega_-(t)$ is locally convex.
In particular, it holds that
\begin{equation} \label{eq:LBop}
\Delta_s\, \vec\id = \varkappa\, \vec\nu =: \vec\varkappa
\qquad \mbox{on $\Gamma(t)$}\,,
\end{equation}
where $\Delta_s = \nabs\,.\,\nabs$ is the Laplace--Beltrami operator on 
$\Gamma(t)$. 

In the case that the interface is non-material, i.e.\ when
$\rho_\Gamma=0$, then the interface conditions (\ref{eq:1c}--c) simplify dramatically. 
In this case, on
recalling (\ref{eq:nabssigma}), we are left with the 
following conditions to hold on $\Gamma(t)$:
\begin{subequations}
\begin{align}
 [\mat\sigma\,\vec\nu]_-^+ & =
- 2\,\mu_\Gamma(\psi)\,\nabs\,.\,\mat D_s (\vec u) 
- \nabs\,.\left[
 \lambda_\Gamma(\psi) \,(\nabs\,.\,\vec u)\,\mat{\mathcal{P}}_\Gamma\right]
- \gamma(\psi)\,\varkappa\,\vec\nu - \nabs\,\gamma(\psi) 
\,, \label{eq:2b} \\ 
\vec{\mathcal{V}}\,.\,\vec\nu & = \vec u\,.\,\vec \nu 
\,. \label{eq:2c} 
\end{align}
\end{subequations}
If, in addition, $\lambda_\Gamma(\psi) = \mu_\Gamma(\psi) = 0$, 
then (\ref{eq:2b},b)
reduce to the interface conditions studied by the authors in \cite{tpfs}, where
a two-phase flow problem with insoluble surfactant is considered.

For later purposes, we introduce the 
surface energy function $F$, which satisfies
\begin{subequations}
\begin{equation} \label{eq:F}
\gamma(r) = F(r) - r\,F'(r) \qquad \forall\ r \in (0,\psi_\infty) \,,
\end{equation}
and
\begin{equation} \label{eq:F0}
\lim_{r\to0} r\,F'(r) = F(0) - \gamma(0) = 0\,.
\end{equation}
\end{subequations}
This means in particular that 
\begin{equation} \label{eq:Fdd}
\gamma'(r) = - r\,F''(r) \qquad \forall\ r \in (0,\psi_\infty)\,. 
\end{equation}
It immediately follows from (\ref{eq:Fdd}) and (\ref{eq:gammaprime}) that
$F \in C([0,\psi_\infty)) \cap C^2(0,\psi_\infty)$ is convex.
Typical examples for $\gamma$ and $F$ are given by
\begin{subequations}
\begin{equation} \label{eq:gamma1}
\gamma(r) = \overline\gamma \,( 1 - \beta\,r) \,,\quad
F(r) = \overline\gamma\left[1+\beta\,r\left(\ln r 
- 1 \right) \right],\ \psi_\infty = \infty\,,
\end{equation}
which represents a linear equation of state, and by
\begin{equation} \label{eq:gamma2}
\gamma(r) = \overline\gamma \left[ 1 + \beta\,\psi_\infty\,\ln \left( 
1 - \tfrac{r}{\psi_\infty} \right) \right] , \quad
F(r) = \overline\gamma\left[ 1 + 
\beta\left( r\,\ln \tfrac{r}{\psi_\infty-r} + \psi_\infty\,\ln
\tfrac{\psi_\infty-r}{\psi_\infty} \right) \right],
\end{equation}
\end{subequations}
the so-called Langmuir equation of state,
where $\overline\gamma \in \R_{>0}$ and $\beta \in \Rgeq$ are further given 
parameters, where we note that the special case $\beta = 0$ means that
(\ref{eq:gamma1},b) reduce to
\begin{equation} \label{eq:Fconst}
F(r) = \gamma(r) = \overline\gamma \in \R_{>0}\qquad \forall\ r \in \R\,.
\end{equation}
In the case (\ref{eq:Fconst}) the surface tension no longer depends on the
surfactant concentration $\psi$.

Before introducing our finite element approximation, 
we will state an appropriate weak formulation. With this in mind, 
we introduce the function spaces
\begin{align*}
\uspace & := \{ \vec\varphi \in [H^1(\Omega)]^d 
: \vec\varphi = \vec0 \ \mbox{ on } 
\partial_1\Omega\,,\ \vec\varphi \,.\,\unitn = 0 \ \mbox{ on } \partial_2\Omega
 \} \,,
\quad \pspace := L^2(\Omega)\,, \\ 
\widehat\pspace & := \{\eta \in \pspace : \int_\Omega\eta \dL{d}=0 \}\,,\quad
\utimespace :=  L^2(0,T; \uspace) \cap H^1(0,T;[L^2(\Omega)]^d)\,,\quad
\psispace := H^1(\GT)\,, \nonumber \\
\utimespace_\Gamma & := \{ \vec\varphi \in \utimespace : \vec\varphi \!\mid_\GT
\in [\psispace]^d \}\,.
\end{align*}
Let $(\cdot,\cdot)$ and $\langle \cdot, \cdot \rangle_{\Gamma(t)}$
denote the $L^2$--inner products on $\Omega$ and $\Gamma(t)$, respectively.
For later use we recall from \citet[Def.\ 2.11]{DziukE13} that
\begin{equation} \label{eq:DEdef2.11}
\left\langle \zeta, \nabs\,.\, \vec\eta \right\rangle_{\Gamma(t)}
+ \left\langle \nabs\,\zeta , \vec\eta \right\rangle_{\Gamma(t)}
= - \left\langle \zeta\,\vec\eta , \vec\varkappa \right\rangle_{\Gamma(t)} 
\qquad \forall\ \zeta \in H^1(\Gamma(t)),\, \vec \eta \in [H^1(\Gamma(t))]^d\,.
\end{equation}
We remark that it follows from (\ref{eq:DEdef2.11}) that
\begin{equation} \label{eq:newGD}
\left\langle \gamma(\psi)\,\vec\varkappa + \nabs\,\gamma(\psi), 
\vec \xi \right\rangle_{\Gamma(t)} 
=
\left\langle \gamma(\psi)\,\varkappa\,\vec \nu + \nabs\,\gamma(\psi), 
\vec \xi \right\rangle_{\Gamma(t)}
=
- \left\langle \gamma(\psi), \nabs\,.\,\vec \xi \right\rangle_{\Gamma(t)}
\quad \forall\ \vec\xi \in \uspace\,.
\end{equation}

We recall from \cite{fluidfbp} that it follows from (\ref{eq:NSb}--e) and
(\ref{eq:1d}) that
\begin{align}
( \rho\,(\vec u \,.\,\nabla)\,\vec u, \vec \xi )
&= \tfrac12 \left[ (\rho\,(\vec u\,.\,\nabla)\,\vec u, \vec \xi) -
(\rho\,(\vec u\,.\,\nabla)\,\vec \xi,\vec u)
-\langle [\rho]_-^+\,\vec u\,.\,\vec \nu, 
  \vec u\,.\,\vec \xi \rangle_{\Gamma(t)} \right]
\nonumber \\ 
& \hspace{3.2in} \qquad \forall\ \vec \xi \in [H^1(\Omega)]^d
\label{eq:advect}
\end{align}
and
\begin{align}
\ddt(\rho \,\vec u, \vec \xi) 
&= (\rho\,\vec u_t, \vec \xi)
+ ( \rho\,\vec u, \vec \xi_t)
- \left\langle [\rho]_-^+\,\vec u\,.\,\vec \nu, 
\vec u \,.\,\vec \xi \right\rangle_{\Gamma(t)}
\qquad \forall \ 
\vec \xi \in \utimespace\,,
\label{eq:rhot1}
\end{align}
respectively.
Therefore, it holds that
\begin{equation*} 
(\rho\,\vec u_t, \vec \xi) = 
\tfrac{1}{2} \left[
\ddt (\rho\,\vec u,\vec \xi) + (\rho\,\vec u_t, \vec \xi)
- ( \rho\,\vec u, \vec \xi_t)
+ \left\langle [\rho]_-^+\,\vec u\,.\,\vec \nu, 
\vec u \,.\,\vec \xi \right\rangle_{\Gamma(t)} 
\right]
\qquad \forall\ \vec \xi \in \utimespace\,,
\end{equation*}
which on combining with (\ref{eq:advect}) yields that
\begin{align} \label{eq:rhot3}
& (\rho\,[\vec u_t + (\vec u \,.\,\nabla)\,\vec u], \vec \xi)
\nonumber \\ & \quad
= \tfrac12\left[ \ddt (\rho\,\vec u, \vec \xi) + (\rho\,\vec u_t, \vec \xi)
- ( \rho\,\vec u, \vec \xi_t)
+ (\rho, [(\vec u\,.\,\nabla)\,\vec u]\,.\,\vec \xi
- [(\vec u\,.\,\nabla)\,\vec \xi]\,.\,\vec u) \right]
\qquad \forall\ \vec \xi \in \utimespace\,.
\end{align}
Moreover, it holds, 
on noting (\ref{eq:NSd}) and (\ref{eq:sigma}),
that for all $\vec \xi \in \uspace$ 
\begin{align}
& \int_{\Omega_+(t)\cup\Omega_-(t)} (\nabla\,.\,\mat\sigma)\,.\, \vec \xi 
\dL{d} 
= - 2\,(\mu\,\mat D(\vec u), \mat D(\vec \xi)) + (p, \nabla\,.\,\vec \xi)
- \left \langle [\mat\sigma\,\vec\nu]_-^+, \vec \xi \right \rangle_{\Gamma(t)}
\,,
\label{eq:sigmaibp}
\end{align}
where we have also noted for symmetric matrices $\mat A \in \R^{d \times d}$ that
$\mat A : \mat B = \mat A : \tfrac12\,(\mat B + \mat B^T)$ 
for all $\mat B \in \R^{d \times d}$.

Similarly to (\ref{eq:matpartu}) we define the following time derivative that
follows the parameterization $\vec x(\cdot, t)$ of $\Gamma(t)$, rather than
$\vec u$. In particular, we let
\begin{equation} \label{eq:matpartx}
\matpartx\, \zeta = \zeta_t + \vec{\mathcal{V}} \,.\,\nabla\,\zeta
\qquad \forall\ \zeta \in H^1(\GT)\,;
\end{equation}
where we stress once again that this definition is well-defined, even though
$\zeta_t$ and $\nabla\,\zeta$ do not make sense separately for a
function $\zeta \in H^1(\GT)$.
On recalling (\ref{eq:matpartu}) we obtain that
\begin{equation} \label{eq:matpartux}
\matpartx = \matpartu  \qquad\text{if}\qquad 
\vec{\mathcal{V}} = \vec u \quad \text {on } \Gamma(t)\,.
\end{equation}
We note that the definition (\ref{eq:matpartx}) 
differs from the definition of $\partial^\circ$ in 
\citet[p.\ 327]{DziukE13}, where
$\partial^\circ\,\zeta = \zeta_t + (\vec{\mathcal{V}}\,.\,\vec \nu)\, 
\vec\nu\,.\,\nabla\,\zeta$
for the ``normal time derivative''. 
It holds that
\begin{equation} \label{eq:DElem5.2}
\ddt \left\langle \chi, \zeta \right\rangle_{\Gamma(t)}
 = \left\langle \matpartx\,\chi, \zeta \right\rangle_{\Gamma(t)}
 + \left\langle \chi, \matpartx\,\zeta \right\rangle_{\Gamma(t)}
+ \left\langle \chi\,\zeta, \nabs\,.\,\vec{\mathcal{V}} 
\right\rangle_{\Gamma(t)}
\qquad \forall\ \chi,\zeta \in H^1(\GT)\,,
\end{equation}
see \citet[Lem.\ 5.2]{DziukE13}.

If $\vec {\cal V} = \vec u$ on $\Gamma(t)$, then 
it follows from (\ref{eq:1b}), (\ref{eq:nabssigma}),
(\ref{eq:matpartux}), (\ref{eq:DElem5.2}) and (\ref{eq:DEdef2.11}) that
\begin{align}
&\ddt\left\langle \rho_\Gamma\,\vec u, \vec\xi \right\rangle_{\Gamma(t)}
+ 2 \left\langle \mu_\Gamma(\psi)\,\mat D_s (\vec u) , \mat D_s ( \vec \xi )
\right\rangle_{\Gamma(t)}
+ \left\langle \lambda_\Gamma(\psi) \, 
 \nabs\,.\, \vec u , \nabs\,.\, \vec \xi \right\rangle_{\Gamma(t)}
\nonumber \\ & \hspace{1cm}
-\left\langle \gamma(\psi)\,\vec\varkappa + \nabs\,\gamma(\psi) , 
\vec \xi \right\rangle_{\Gamma(t)}
= 
\left\langle \rho_\Gamma\,\vec u, \matpartx\, \vec\xi 
\right\rangle_{\Gamma(t)}
+
\left \langle [\mat\sigma\,\vec\nu]_-^+, \vec \xi \right \rangle_{\Gamma(t)}
\qquad \forall\ \vec \xi \in \psispace\,,
\label{eq:weakGam}
\end{align}
where we have noted for symmetric matrices $\mat A \in \R^{d \times d}$
that $\mat{\mathcal{P}}_\Gamma\,\mat A\,\mat{\mathcal{P}}_\Gamma : \mat B = 
\mat{\mathcal{P}}_\Gamma\,\mat A \, \mat{\mathcal{P}}_\Gamma : \tfrac12\,\mat{\mathcal{P}}_\Gamma\,(\mat B + \mat B^T)\,
\mat{\mathcal{P}}_\Gamma$ for all $\mat B \in \R^{d \times d}$.

We are now in a position to state weak formulations of the 
Navier--Stokes two-phase flow problem for a Boussinesq--Scriven surface fluid
that we consider in this paper. 
The natural weak formulation of
the system (\ref{eq:NSa}--e), (\ref{eq:sigma}), (\ref{eq:1c}--c), 
(\ref{eq:sigmaG}) and (\ref{eq:1surf}) is given as follows.
Find $\Gamma(t) = \vec x(\Upsilon, t)$ for $t\in[0,T]$ 
with $\vec{\mathcal{V}} \in [L^2(\GT)]^d$, and functions 
$\rho_\Gamma \in \psispace$, 
$\vec u \in \utimespace_\Gamma$, 
$p \in L^2(0,T; \widehat\pspace)$, 
$\vec\varkappa \in [L^2(\GT)]^d$ and
$\psi \in \psispace$ 
such that for almost all $t \in (0,T)$ it holds that
\begin{subequations}
\begin{align}
& \ddt\left\langle \rho_\Gamma, \zeta \right\rangle_{\Gamma(t)}
= \left\langle \rho_\Gamma, \matpartx\, \zeta \right\rangle_{\Gamma(t)}
\quad\forall\ \zeta \in \psispace\,,
\label{eq:weakGDa0}\\
& \tfrac12\left[ \ddt (\rho\,\vec u, \vec \xi) + (\rho\,\vec u_t, \vec \xi)
- (\rho\,\vec u, \vec \xi_t)
+ (\rho, [(\vec u\,.\,\nabla)\,\vec u]\,.\,\vec \xi
- [(\vec u\,.\,\nabla)\,\vec \xi]\,.\,\vec u) \right] 
+ 2\,(\mu\,\mat D(\vec u), \mat D(\vec \xi)) 
\nonumber \\ & \quad
- (p, \nabla\,.\,\vec \xi)
+ \ddt\left\langle \rho_\Gamma\,\vec u, \vec\xi \right\rangle_{\Gamma(t)}
+ 2 \left\langle \mu_\Gamma(\psi)\, \mat D_s (\vec u) , \mat D_s ( \vec \xi )
\right\rangle_{\Gamma(t)}
\nonumber \\ & \quad
+  \left\langle\lambda_\Gamma(\psi) \, 
 \nabs\,.\, \vec u , \nabs\,.\, \vec \xi \right\rangle_{\Gamma(t)}
- \left\langle \gamma(\psi)\,\vec\varkappa + \nabs\,\gamma(\psi) , 
\vec \xi \right\rangle_{\Gamma(t)}
\nonumber \\ & \quad
 = (\vec f, \vec \xi) 
+ \left\langle \rho_\Gamma\,\vec u, \matpartx\, \vec\xi 
\right\rangle_{\Gamma(t)}
\qquad \forall\ \vec \xi \in \utimespace_\Gamma\,,
\label{eq:weakGDa} \\
& (\nabla\,.\,\vec u, \varphi) = 0  \qquad \forall\ \varphi \in 
\widehat\pspace\,, \label{eq:weakGDb} \\
&  \left\langle \vec{\mathcal{V}} - \vec u, \vec\chi \right\rangle_{\Gamma(t)} = 0
 \qquad\forall\ \vec\chi \in [L^2(\Gamma(t))]^d\,,\label{eq:weakGDc} \\
& \left\langle \vec\varkappa, \vec\eta \right\rangle_{\Gamma(t)}
+ \left\langle \nabs\,\vec \id, \nabs\,\vec \eta \right\rangle_{\Gamma(t)}
 = 0  \qquad\forall\ \vec\eta \in [H^1(\Gamma(t))]^d\,, \label{eq:weakGDd} \\
& \ddt\left\langle \psi, \zeta \right\rangle_{\Gamma(t)}
+ \Ds \left\langle \nabs\,\psi, \nabs\,\zeta \right\rangle_{\Gamma(t)}
= \left\langle \psi, \matpartx\, \zeta \right\rangle_{\Gamma(t)}
\quad\forall\ \zeta \in \psispace\,,
\label{eq:weakGDe}
\end{align}
\end{subequations}
as well as the initial conditions (\ref{eq:1init}), where in (\ref{eq:weakGDc}) 
we have recalled (\ref{eq:V}). 
Here (\ref{eq:weakGDa}) is derived from (\ref{eq:NSa}) and (\ref{eq:1b})
by combining (\ref{eq:rhot3}), 
(\ref{eq:sigmaibp}) and (\ref{eq:weakGam}), on noting (\ref{eq:weakGDc}).
The equations (\ref{eq:weakGDa0},f) 
are derived,
similarly
to (\ref{eq:weakGam}),
 from (\ref{eq:1c}) and (\ref{eq:1surf}), respectively, 
on noting (\ref{eq:DElem5.2}) and (\ref{eq:weakGDc}). 
Of course, it follows from
(\ref{eq:weakGDc}) and (\ref{eq:matpartux}) that $\matpartx$ in
(\ref{eq:weakGDa0},b,f) can be replaced by $\matpartu$.

In what follows we would like to derive an energy bound for a solution of
(\ref{eq:weakGDa0}--f). All of the following considerations are formal, in the
sense that we make the appropriate assumptions about the existence, 
boundedness and regularity of a solution to \mbox{(\ref{eq:weakGDa0}--f)}. 
In particular, we assume that $\psi \in [0,\psi_\infty)$.
Choosing $\vec\xi = \vec u$ in (\ref{eq:weakGDa}),
$\varphi = p(\cdot, t)$ in (\ref{eq:weakGDb}) and 
$\zeta = -\frac12\,|\vec u\mid_{{\cal G}_T}\!\!|^2$
in (\ref{eq:weakGDa0}), and combining 
yields 
that
\begin{align} \label{eq:d1}
& \tfrac12\,\ddt \,\left [\|\rho^\frac12\,\vec u \|_0^2 
+ \left \langle \rho_\Gamma \,\vec u, \vec u \right \rangle_{\Gamma(t)} \right]
+ 2\,\| \mu^\frac12\,\mat D(\vec u)\|_0^2 
+ 2\left\langle\mu_\Gamma(\psi)\, \mat D_s(\vec u), \mat D_s(\vec u) 
\right\rangle_{\Gamma(t)} 
\nonumber \\
& \hspace{1cm}
+  \left\langle\lambda_\Gamma(\psi) \,
 \nabs\,.\, \vec u , \nabs\,.\, \vec u \right\rangle_{\Gamma(t)} 
= (\vec f, \vec u) 
+ \left\langle \gamma(\psi)\,\vec\varkappa + \nabs\,\gamma(\psi) , \vec u 
\right\rangle_{\Gamma(t)}.
\end{align}
If $\gamma$ is constant, recall (\ref{eq:Fconst}), then 
the second term on the right hand side of (\ref{eq:d1})
collapses, on noting (\ref{eq:weakGDc},e) and (\ref{eq:DElem5.2}), to
\begin{align}
\overline\gamma \left\langle \vec \varkappa, \vec u \right \rangle_{\Gamma(t)}
&= \overline\gamma \left\langle \vec \varkappa, \vec {\cal V} \right \rangle_{\Gamma(t)}
= - \overline\gamma \left\langle \nabs \,\vec \id , \nabs \,\vec {\cal V} \right \rangle_{\Gamma(t)}
=- \overline\gamma \left\langle \mat\Id , \nabs \,\vec {\cal V} \right \rangle_{\Gamma(t)}
\nonumber \\
&=- \overline\gamma \left\langle 1, \nabs \,.\,\vec {\cal V} \right \rangle_{\Gamma(t)}
= - \overline\gamma\, \ddt\, {\cal H}^{d-1}(\Gamma(t))\,.
\label{eq:constgam}
\end{align}  
Combining (\ref{eq:d1}) and (\ref{eq:constgam})
yields the energy identity \citet[(3.2)]{BotheP10} if $\vec f = \vec 0$ in 
the absence of surfactant, i.e.\ if (\ref{eq:muconst}) and (\ref{eq:Fconst}) 
hold. Here we note that the authors in \cite{BotheP10} use a
slightly different notation and assume that 
$\overline\lambda_\Gamma \geq \overline\mu_\Gamma$, which is a stronger
assumption than (\ref{eq:mulambda}). In particular, we note that 
\begin{align}
& 2\left\langle\mu_\Gamma(\psi)\, \mat D_s(\vec \eta), \mat D_s(\vec \eta) 
\right\rangle_{\Gamma(t)} 
+  \left\langle\lambda_\Gamma(\psi) \,
 \nabs\,.\, \vec \eta , \nabs\,.\, \vec \eta \right\rangle_{\Gamma(t)} 
\nonumber \\ & \hspace{1cm}
= 2\left\langle\mu_\Gamma(\psi)\, \mat {\hat D}_s(\vec \eta), 
\mat {\hat D}_s(\vec \eta) \right\rangle_{\Gamma(t)} 
+  \left\langle (\lambda_\Gamma(\psi) + \tfrac2{d-1}\,\mu_\Gamma(\psi))\,
 \nabs\,.\, \vec \eta , \nabs\,.\, \vec \eta \right\rangle_{\Gamma(t)} 
\nonumber \\ & \hspace{10cm}
\qquad \forall\ \vec\eta \in [H^1(\Gamma(t))]^d \,,
\label{eq:HG}
\end{align}
where
\begin{equation} \label{eq:hatDs}
\mat {\hat D}_s(\vec \eta) =
\mat D_s(\vec \eta) - \tfrac1{d-1}\left(\tr{\mat D_s(\vec\eta)}\right)
\mat{\mathcal{P}}_\Gamma
= \mat D_s(\vec \eta) - \tfrac1{d-1}\left( \nabs\,.\,\vec \eta \right)
\mat{\mathcal{P}}_\Gamma
\end{equation}
denotes the deviatoric part of $\mat D_s(\vec \eta)$.
Hence (\ref{eq:d1}) can be reformulated as
\begin{align} \label{eq:d11}
& \tfrac12\,\ddt \,\left [\|\rho^\frac12\,\vec u \|_0^2 
+ \left \langle \rho_\Gamma \,\vec u, \vec u \right \rangle_{\Gamma(t)} \right]
+ 2\,\| \mu^\frac12\,\mat D(\vec u)\|_0^2 
+ 2\left\langle\mu_\Gamma(\psi)\, \mat {\hat D}_s(\vec u), 
\mat {\hat D}_s(\vec u) \right\rangle_{\Gamma(t)} 
\nonumber \\
& \hspace{5mm}
+  \left\langle(\lambda_\Gamma(\psi) + \tfrac2{d-1}\,\mu_\Gamma(\psi)) \,
 \nabs\,.\, \vec u , \nabs\,.\, \vec u \right\rangle_{\Gamma(t)} 
= (\vec f, \vec u) 
+ \left\langle \gamma(\psi)\,\vec\varkappa + \nabs\,\gamma(\psi) , \vec u 
\right\rangle_{\Gamma(t)}.
\end{align}
In order to formally derive an energy bound for the solution of
(\ref{eq:weakGDa0}--f), we need to control the last term on the right hand side
of (\ref{eq:d11}). This can be achieved as in \cite{tpfs}, and we
repeat these formal considerations here for the benefit of the reader.
On assuming that $\gamma$ is not constant, recall (\ref{eq:Fconst}), 
we would like to choose $\zeta = F'(\psi)$ in 
(\ref{eq:weakGDe}). As $F'$ in general is singular at the origin, 
recall (\ref{eq:Fdd}), we instead
choose $\zeta = F'(\psi + \alpha)$ for some $\alpha \in \R_{>0}$ with
$\psi + \alpha < \psi_\infty$.
Then we obtain, on recalling (\ref{eq:F}) and (\ref{eq:DElem5.2}), that
\begin{align} \label{eq:d4}
& \ddt \left\langle F(\psi + \alpha) - \gamma(\psi + \alpha), 1 
\right\rangle_{\Gamma(t)}
+ \Ds \left\langle \nabs\,(\psi + \alpha), 
\nabs\,F'(\psi + \alpha) \right\rangle_{\Gamma(t)} \nonumber \\ &
\qquad\qquad
= \left\langle \psi + \alpha, \matpartx\,F'(\psi + \alpha) \right\rangle_{\Gamma(t)}
+ \alpha  \left\langle F'(\psi + \alpha) , \nabs\,.\,\vec{\mathcal{V}} \right\rangle_{\Gamma(t)}
\,.
\end{align}
Moreover, choosing $\chi = \gamma(\psi+\alpha)$, $\zeta = 1$ in
(\ref{eq:DElem5.2}), 
and then choosing $\vec\eta = \vec{\mathcal{V}}$, 
$\zeta = \gamma(\psi+\alpha)$ in
(\ref{eq:DEdef2.11}) gives that
\begin{align} \label{eq:d2}
\ddt \left\langle \gamma(\psi+\alpha), 1 \right\rangle_{\Gamma(t)}
& = \left\langle \matpartx\,\gamma(\psi+\alpha), 1 \right\rangle_{\Gamma(t)}
+ \left\langle \gamma(\psi+\alpha), \nabs\,.\,\vec{\mathcal{V}} \right\rangle_{\Gamma(t)} 
\nonumber \\ &
= \left\langle \matpartx\,\gamma(\psi+\alpha), 1 \right\rangle_{\Gamma(t)}
- \left\langle \gamma(\psi+\alpha) \,\vec\varkappa + 
 \nabs\,\gamma (\psi+\alpha) , \vec{\mathcal{V}} \right\rangle_{\Gamma(t)}.
\end{align}
In addition, it follows from (\ref{eq:Fdd}) that
\begin{equation} \label{eq:d3}
\matpartx\,\gamma(\psi+\alpha) = \gamma'(\psi+\alpha) \, \matpartx\,\psi
= - (\psi+\alpha)\,F''(\psi+\alpha)\,\matpartx\,\psi 
= - (\psi+\alpha)\,\matpartx\,F'(\psi+\alpha)\,.
\end{equation}
Combining (\ref{eq:d4}), (\ref{eq:d2}) and (\ref{eq:d3}) yields that
\begin{align} \label{eq:d123}
& \ddt \left\langle F(\psi + \alpha) , 1 \right\rangle_{\Gamma(t)}
+ \Ds \left\langle \nabs\,\mathcal{F}(\psi + \alpha), 
\nabs\,\mathcal{F}(\psi + \alpha) \right\rangle_{\Gamma(t)} \nonumber \\ &
\qquad\qquad
= - \left\langle \gamma(\psi+\alpha) \,\vec\varkappa + 
 \nabs\,\gamma (\psi+\alpha) , \vec{\mathcal{V}} \right\rangle_{\Gamma(t)}
+ \alpha  \left\langle F'(\psi + \alpha) , \nabs\,.\,\vec{\mathcal{V}} \right\rangle_{\Gamma(t)}
\,,
\end{align}
where, on recalling (\ref{eq:Fdd}) and (\ref{eq:gammaprime}),
\begin{equation*}  
\mathcal{F}(r) = \int_0^r [F''(y)]^{\frac12} \;{\rm d}y\,.
\end{equation*}
Letting $\alpha \to 0$ in (\ref{eq:d123}) yields, on recalling (\ref{eq:F0}),
that
\begin{equation}  \label{eq:d45}
 \ddt \left\langle F(\psi) , 1 \right\rangle_{\Gamma(t)}
+ \Ds \left\langle \nabs\,\mathcal{F}(\psi), 
\nabs\,\mathcal{F}(\psi) \right\rangle_{\Gamma(t)} 
= - \left\langle \gamma(\psi) \,\vec\varkappa + 
 \nabs\,\gamma (\psi) , \vec{\mathcal{V}} \right\rangle_{\Gamma(t)} .
\end{equation}
We note that (\ref{eq:d45}) is still valid,
 on recalling 
(\ref{eq:constgam}),
in the case (\ref{eq:Fconst}).
Combining (\ref{eq:d45}) with (\ref{eq:d11}) implies the a priori energy 
equation
\begin{align} \label{eq:d5}
& \ddt \left( \tfrac12\left[\|\rho^\frac12\,\vec u \|_0^2 
+ \left \langle \rho_\Gamma \,\vec u, \vec u \right \rangle_{\Gamma(t)} \right]
+ \left\langle F(\psi), 1 \right\rangle_{\Gamma(t)} \right)
+ 2\,\| \mu^\frac12\,\mat D(\vec u)\|_0^2 
\nonumber \\ & \qquad
+ 2\left\langle \mu_\Gamma(\psi)\, \mat {\hat D}_s(\vec u), 
\mat {\hat D}_s(\vec u)\right\rangle_{\Gamma(t)}
+  \left\langle(\lambda_\Gamma(\psi)  + \tfrac2{d-1}\,\mu_\Gamma(\psi))\,
 \nabs\,.\, \vec u , \nabs\,.\, \vec u \right\rangle_{\Gamma(t)}
\nonumber \\ & \qquad
+ \Ds \left\langle \nabs\,\mathcal{F}(\psi), \nabs\,\mathcal{F}(\psi) 
\right\rangle_{\Gamma(t)}
= (\vec f, \vec u )\,,
\end{align}
where we recall the assumption (\ref{eq:mulambda}). 

Apart from the energy law (\ref{eq:d5}), certain conservation properties can
also be shown for a solution of (\ref{eq:weakGDa0}--f). 
For example, 
the volume
of $\Omega_-(t)$ is preserved in time, i.e.\ the mass of each phase is 
conserved. To see this, choose $\vec\chi = \vec\nu$ in
(\ref{eq:weakGDc}) and $\varphi = \charfcn{\Omega_-(t)}$ in (\ref{eq:weakGDb})
to obtain
\begin{equation}
\frac{\rm d}{{\rm d}t} \vol(\Omega_-(t)) = 
\left\langle \vec{\mathcal{V}}, \vec\nu \right\rangle_{\Gamma(t)}
= \left\langle \vec u, \vec\nu \right\rangle_{\Gamma(t)}
= \int_{\Omega_-(t)} \nabla\,.\,\vec u \dL{d} 
=0\,. \label{eq:conserved}
\end{equation}
In addition, we note that it immediately follows from 
choosing $\zeta = 1$ in (\ref{eq:weakGDa0},f) that the 
total surface mass and the total amount of surfactant are preserved, i.e.\
\begin{equation} \label{eq:totalpsi}
\ddt \int_{\Gamma(t)} \rho_\Gamma \dH{d-1}= 0 \qquad \mbox{and}\qquad
\ddt \int_{\Gamma(t)} \psi \dH{d-1} = 0 \,.
\end{equation}

We note that, in contrast to (\ref{eq:matpartux}), 
if we relax $\vec{\mathcal{V}} = \vec u \!\mid_{\Gamma(t)}$ to
\begin{equation*} 
\vec{\mathcal{V}} \,.\,\vec \nu = \vec u\,.\,\vec \nu
\quad \text{on } \Gamma(t)\,,
\end{equation*}
then it holds that
\begin{equation} \label{eq:matpartxu}
\matpartx\, \zeta = \matpartu\,\zeta 
+ ((\vec{\mathcal{V}} - \vec u)\,.\,\nabla)\,\zeta 
= \matpartu\,\zeta 
+ ((\vec{\mathcal{V}} - \vec u)\,.\,\nabs)\,\zeta 
\qquad \forall\ \zeta \in H^1(\GT)\,,
\end{equation}
and similarly for $\vec\zeta \in [H^1(\GT)]^d$.

Our preferred finite element approximation will be based on the following weak
formulation.
Find $\Gamma(t) = \vec x(\Upsilon, t)$ for $t\in[0,T]$ 
with $\vec{\mathcal{V}} \in [L^2(\GT)]^d$, and functions 
$\rho_\Gamma \in \psispace$, 
$\vec u \in \utimespace_\Gamma$, 
$p \in L^2(0,T; \widehat\pspace)$, 
$\varkappa \in L^2(\GT)$ and $\psi \in \psispace$ 
such that for almost all $t \in (0,T)$ it holds that
\begin{subequations}
\begin{align}
& \ddt\left\langle \rho_\Gamma, \zeta \right\rangle_{\Gamma(t)}
+ \left\langle \rho_\Gamma,(\vec {\mathcal{V}} - \vec u)\,.\, 
 \nabs\,\zeta\right\rangle_{\Gamma(t)}
= \left\langle \rho_\Gamma, \matpartx\, \zeta \right\rangle_{\Gamma(t)}
\quad\forall\ \zeta \in \psispace\,,
\label{eq:weaka0}\\
& \tfrac12\left[ \ddt (\rho\,\vec u, \vec \xi) + (\rho\,\vec u_t, \vec \xi)
- (\rho\,\vec u, \vec \xi_t)
+ (\rho, [(\vec u\,.\,\nabla)\,\vec u]\,.\,\vec \xi
- [(\vec u\,.\,\nabla)\,\vec \xi]\,.\,\vec u) \right] 
+ 2\,(\mu\,\mat D(\vec u), \mat D(\vec \xi)) 
\nonumber \\ & \quad
- (p, \nabla\,.\,\vec \xi)
+ \ddt\left\langle \rho_\Gamma\,\vec u, \vec\xi \right\rangle_{\Gamma(t)}
+ 2 \left\langle \mu_\Gamma(\psi)\, \mat D_s (\vec u) , \mat D_s ( \vec \xi )
\right\rangle_{\Gamma(t)}
\nonumber \\ & \quad
+ \left\langle\lambda_\Gamma(\psi) \,
 \nabs\,.\, \vec u , \nabs\,.\, \vec \xi \right\rangle_{\Gamma(t)}
- \left\langle \gamma(\psi)\,\varkappa\,\vec\nu + \nabs\,\gamma(\psi) 
,\vec \xi \right\rangle_{\Gamma(t)}
\nonumber \\ & \quad
+ \left\langle \rho_\Gamma\,\vec u, [(\vec {\mathcal{V}} - \vec u)\,.\,\nabs]\,\vec\xi
 \right\rangle_{\Gamma(t)}
 = (\vec f, \vec \xi) 
+ \left\langle \rho_\Gamma\,\vec u, \matpartx\, \vec\xi 
\right\rangle_{\Gamma(t)}
\qquad \forall\ \vec \xi \in \utimespace_\Gamma\,,
\label{eq:weaka} \\
& (\nabla\,.\,\vec u, \varphi) = 0  \qquad \forall\ \varphi \in 
\widehat\pspace\,, \label{eq:weakb} \\
&  \left\langle \vec{\mathcal{V}} - \vec u, \chi\,\vec\nu \right\rangle_{\Gamma(t)} = 0
 \qquad\forall\ \chi \in L^2(\Gamma(t))\,,
\label{eq:weakc} \\
& \left\langle \varkappa\,\vec\nu, \vec\eta \right\rangle_{\Gamma(t)}
+ \left\langle \nabs\,\vec\id, \nabs\,\vec \eta \right\rangle_{\Gamma(t)}
 = 0  \qquad\forall\ \vec\eta \in [H^1(\Gamma(t))]^d\,, \label{eq:weakd}\\
& \ddt\left\langle \psi, \zeta \right\rangle_{\Gamma(t)}
+ \Ds \left\langle \nabs\,\psi, \nabs\,\zeta \right\rangle_{\Gamma(t)}
+ \left\langle \psi,(\vec {\mathcal{V}} - \vec u)\,.\, \nabs\,\zeta\right\rangle_{\Gamma(t)}
= \left\langle \psi, \matpartx\, \zeta \right\rangle_{\Gamma(t)}
\qquad\forall\ \zeta \in \psispace\,,
\label{eq:weake}
\end{align}
\end{subequations}
as well as the initial conditions (\ref{eq:1init}), where in (\ref{eq:weaka0},b,d,f)
we have recalled (\ref{eq:V}). 

Similarly to (\ref{eq:d1}), choosing $\vec\xi = \vec u$ in (\ref{eq:weaka}), 
$\varphi = p(\cdot, t)$ in (\ref{eq:weakb})
and $\zeta = -\frac12\,|\vec u \mid_{{\cal G}_T}\!\!|^2$
in (\ref{eq:weaka0}) yields, on noting $\vec\varkappa=\varkappa\,\vec\nu$
and
\begin{equation} \label{eq:nbgn}
\tfrac12 \left\langle \rho_\Gamma,(\vec {\mathcal{V}} - \vec u)\,.\, 
 \nabs\,|\vec u|^2\right\rangle_{\Gamma(t)}
= \left\langle \rho_\Gamma\,\vec u, [(\vec {\mathcal{V}} - \vec u)\,.\,\nabs]\,
 \vec u \right\rangle_{\Gamma(t)} ,
\end{equation}
that the formal equation (\ref{eq:d1}) holds for a solution of the
weak formulation (\ref{eq:weaka0}--f). 
Moreover, 
similarly to (\ref{eq:d1})--(\ref{eq:d5}), we can formally show that a solution
to (\ref{eq:weaka0}--f) satisfies the a priori energy bound (\ref{eq:d5}). 
We observe that the
analogue of (\ref{eq:d45}) has as right hand side
\begin{align} \label{eq:dbgn}
& - \left\langle \gamma(\psi)\,\vec\varkappa + \nabs\,\gamma(\psi), 
  \vec{\mathcal{V}} \right\rangle_{\Gamma(t)}
- \left\langle \psi\,(\vec {\mathcal{V}} - \vec u), \nabs\,F'(\psi)
\right\rangle_{\Gamma(t)} \nonumber \\ & \qquad
= - \left\langle \gamma(\psi)\,\varkappa\,\vec\nu+ \nabs\,\gamma(\psi), 
  \vec{\mathcal{V}} \right\rangle_{\Gamma(t)}
+ \left\langle \nabs\,\gamma(\psi), \vec {\mathcal{V}} - \vec u
\right\rangle_{\Gamma(t)} \nonumber \\ & \qquad
= -\left\langle \gamma(\psi)\,\varkappa\,\vec\nu+ \nabs\,\gamma(\psi), 
  \vec u \right\rangle_{\Gamma(t)},
\end{align}
where we have used (\ref{eq:weakc}) with $\chi = \gamma(\psi)\,\varkappa$ and
(\ref{eq:Fdd}). Of course, (\ref{eq:dbgn}) now cancels with the last term in
(\ref{eq:d1}), and so we obtain (\ref{eq:d5}). Moreover, the properties
(\ref{eq:conserved}) and (\ref{eq:totalpsi}) also hold
for a solution to (\ref{eq:weaka0}--f).

\setcounter{equation}{0}
\section{Semi-discrete finite element approximation} \label{sec:2}

For simplicity we consider $\Omega$ to be a polyhedral domain. Then 
let  ${\cal T}^h$ 
be a regular partitioning of $\Omega$ into disjoint open simplices
$\sigmaO^h_j$, $j = 1 ,\ldots, J_\Omega^h$.
Associated with ${\cal T}^h$ are the finite element spaces
\begin{equation*} 
 S^h_k := \{\chi \in C(\overline{\Omega}) : \chi\!\mid_{\sigmaO} \in
\mathcal{P}_k(\sigmaO) \quad \forall\ \sigmaO \in {\cal T}^h\} 
\subset H^1(\Omega)\,, \qquad k \in \mathbb{N}\,,
\end{equation*}
where $\mathcal{P}_k(\sigmaO)$ denotes the space of polynomials of degree $k$
on $\sigmaO$. We also introduce $S^h_0$, the space of  
piecewise constant functions on ${\cal T}^h$.
Let $\{\varphi_{k,j}^h\}_{j=1}^{K_k^h}$ be the standard basis functions 
for $S^h_k$, $k\geq 0$.
We introduce $\vec I^h_k:[C(\overline{\Omega})]^d \to [S^h_k]^d$, $k\geq 1$, 
the standard interpolation
operators, such that $(\vec I^h_k\, \vec\eta)(\vec{p}_{k,j}^h)= 
\vec\eta(\vec{p}_{k,j}^h)$ for $j=1,\ldots, K_k^h$; where
$\{\vec p_{k,j}^h\}_{j=1}^{K_k^h}$ denotes the coordinates of the degrees of
freedom of $S^h_k$, $k\geq 1$. In addition we define the standard projection
operator $I^h_0:L^1(\Omega)\to S^h_0$, such that
\begin{equation*}
(I^h_0 \eta)\!\mid_{o} = \frac1{\mathcal{L}^d(o)}\,\int_{o}
\eta \dL{d} \qquad \forall\ o \in \mathcal{T}^h\,.
\end{equation*}
Our approximation to the velocity and pressure on ${\cal T}^h$
will be finite element spaces
$\uspace^h\subset\uspace$ and $\pspace^h(t)\subset\pspace$.
We require also the spaces $\widehat\pspace^h(t):= 
\pspace^h(t) \cap \widehat\pspace$. 
Based on the authors' earlier work in \cite{spurious,fluidfbp}, we will select 
velocity/pressure finite element spaces 
that satisfy the LBB inf-sup condition,
see e.g.\ \citet[p.~114]{GiraultR86}, and augment the pressure space by a 
single additional basis function, namely by the characteristic function of the
inner phase. 
For the obtained spaces $(\uspace^h,\pspace^h(t))$ we are unable to prove that
they satisfy an LBB condition.
The extension of the given pressure finite element space, which is an example
of an XFEM approach, leads to exact volume conservation of the two
phases within the finite element framework.
For the non-augmented spaces we may choose, for example, 
the lowest order Taylor-Hood element 
P2--P1, the P2--P0 element or the P2--(P1+P0) element on setting 
$\uspace^h=[S^h_2]^d\cap\uspace$, 
and $\pspace^h = S^h_1,\,S^h_0$ or $S^h_1+S^h_0$, respectively.
We refer to \cite{spurious,fluidfbp} for more details.

The parametric finite element spaces in order to approximate $\vec x$,
$\vec\varkappa$ in (\ref{eq:weakGDa0}--f) and $\vec x$, $\varkappa$ in 
(\ref{eq:weaka0}--f), respectively, are defined as follows;
see also \cite{Dziuk91,gflows3d}.
Let $\Gamma^h(t)\subset\R^d$ be a $(d-1)$-dimensional {\em polyhedral surface}, 
i.e.\ a union of non-degenerate $(d-1)$-simplices with no
hanging vertices (see \citet[p.~164]{DeckelnickDE05} for $d=3$),
approximating the closed surface $\Gamma(t)$. In
particular, let $\Gamma^h(t)=\bigcup_{j=1}^{J_\Gamma}
\overline{\sigma^h_j(t)}$, where $\{\sigma^h_j(t)\}_{j=1}^{J_\Gamma}$ is a
family of mutually disjoint open $(d-1)$-simplices with vertices
$\{\vec{q}^h_k(t)\}_{k=1}^{K_\Gamma}$.
Then let
\begin{align*} 
\Vht & := \{\vec\chi \in [C(\Gamma^h(t))]^d:\vec\chi\!\mid_{\sigma^h_j}
\mbox{ is linear}\ \forall\ j=1\to J_\Gamma\} \\ & 
=: [\Wht]^d \subset [H^1(\Gamma^h(t))]^d\,,
\end{align*}
where $\Wht \subset H^1(\Gamma^h(t))$ is the space of scalar continuous
piecewise linear functions on $\Gamma^h(t)$, with 
$\{\chi^h_k(\cdot,t)\}_{k=1}^{K_\Gamma}$ 
denoting the standard basis of $\Wht$, i.e.\
\begin{equation} \label{eq:bf}
\chi^h_k(\vec q^h_l(t),t) = \delta_{kl}\qquad
\forall\ k,l \in \{1,\ldots,K_\Gamma\}\,, t \in [0,T]\,.
\end{equation}
For later purposes, we also introduce 
$\pi^h(t): C(\Gamma^h(t))\to \Wht$, the standard interpolation operator
at the nodes $\{\vec{q}_k^h(t)\}_{k=1}^{K_\Gamma}$, and similarly
$\vec\pi^h(t): [C(\Gamma^h(t))]^d\to \Vht$.

On choosing an arbitrary fixed $t_0 \in (0,T)$, we can represent each 
$\vec z \in \Gamma^h(t_0)$ as
\begin{equation} \label{eq:HG0}
\vec z = \sum_{k=1}^{K_\Gamma}
\chi^h_k(\vec z, t_0)\,\vec q^h_k(t_0) \,.
\end{equation}
Now we can parameterize $\Gamma^h(t)$ by
$\vec X^h(\cdot,t) : \Gamma^h(t_0) \to \R^d$, where
$\vec z \mapsto \sum_{k=1}^{K_\Gamma} \chi^h_k(\vec z, t_0)\,\vec q^h_k(t)$,
i.e.\ $\Gamma^h(t_0)$ plays the role of a reference manifold for
$(\Gamma^h(t))_{t\in[0,T]}$.
Then, 
similarly to (\ref{eq:V}), we define the discrete velocity
for $\vec z \in \Gamma^h(t_0)$ by
\begin{equation} \label{eq:Xht}
\vec{\mathcal{V}}^h(\vec z, t_0) := 
\ddt\, \vec X^h (\vec z, t_0) = \sum_{k=1}^{K_\Gamma}
\chi^h_k(\vec z, t_0) \, \ddt\,\vec q^h_k(t_0) \, ,
\end{equation}
which corresponds to \citet[(5.23)]{DziukE13}.
In addition, 
similarly to (\ref{eq:matpartx}), we define
\begin{equation} \label{eq:matpartxh}
\matpartxh\, \zeta (\vec z, t_0) = \ddt\,\zeta(\vec X^h(\vec z, t_0), t_0) =
\zeta_t(\vec z, t_0) + \vec{\mathcal{V}}^h(\vec z, t_0)\,.\,\nabla\,\zeta(\vec z, t_0)
\qquad\forall\ \zeta \in H^1(\GhT)\,,
\end{equation}
where, similarly to (\ref{eq:GT}), we have defined the discrete
space-time surface
\begin{equation} \label{eq:GhT}
\GhT := \bigcup_{t \in [0,T]} \Gamma^h(t) \times \{t\}\,.
\end{equation}
It immediately follows from (\ref{eq:matpartxh}) that
$\matpartxh\,\vec\id = \vec{\mathcal{V}}^h$ on $\Gamma^h(t)$.
For later use, we also introduce the finite element spaces
\begin{align*}
W(\GhT) & := \{ \chi \in C(\GhT) : 
\chi(\cdot, t) \in \Wht \quad \forall\ t \in [0,T] \}\,, \\
W_T(\GhT) & := \{ \chi \in W(\GhT) : \matpartxh\,\chi \in C(\GhT) \}\,.
\end{align*}

On differentiating (\ref{eq:bf}) with respect to $t$, we obtain
that
\begin{equation} \label{eq:mpbf}
\matpartxh\, \chi^h_k = 0
\quad\forall\ k \in \{1,\ldots,K_\Gamma\}\,,
\end{equation}
see also \citet[Lem.\ 5.5]{DziukE13}. 
It follows directly from (\ref{eq:mpbf}) that
\begin{equation} \label{eq:p1}
\matpartxh\,\zeta(\cdot,t) = \sum_{k=1}^{K_\Gamma} \chi^h_k(\cdot,t)\,
\ddt\,\zeta_k(t) \quad \text{on}\ \Gamma^h(t)
\end{equation}
for $\zeta(\cdot,t) = \sum_{k=1}^{K_\Gamma} \zeta_k(t)\,\chi^h_k(\cdot,t)
\in \Wht$.
Moreover, it holds that
\begin{equation} \label{eq:DEeq5.28}
\ddt\, \int_{\sigma^h_j(t)} \zeta \dH{d-1} 
= \int_{\sigma^h_j(t)} \matpartxh\,\zeta + \zeta\,\nabs\,.\,\vec{\mathcal{V}}^h 
\dH{d-1} \quad \forall\ \zeta \in H^1(\sigma^h(t))\,, j \in
\{1,\ldots,J_\Gamma\}\,,
\end{equation}
see \citet[Lem.\ 5.6]{DziukE13}. 
It immediately follows from (\ref{eq:DEeq5.28}) that
\begin{equation} \label{eq:DElem5.6}
\ddt \langle \eta, \zeta \rangle_{\Gamma^h(t)}
 = \langle \matpartxh\,\eta, \zeta \rangle_{\Gamma^h(t)}
 + \langle \eta, \matpartxh\,\zeta \rangle_{\Gamma^h(t)}
+ \langle \eta\,\zeta, \nabs\,.\,\vec{\mathcal{V}}^h \rangle_{\Gamma^h(t)}
\qquad \forall\ \eta,\zeta \in W_T(\GhT)\,,
\end{equation}
which is a discrete analogue of (\ref{eq:DElem5.2}). Here
$\langle\cdot,\cdot\rangle_{\Gamma^h(t)}$ denotes the $L^2$--inner product
on $\Gamma^h(t)$. It is not difficult to show that the
analogue of (\ref{eq:DElem5.6}) with numerical integration also holds. 
We state
this result in the next lemma, together with a discrete variant of
(\ref{eq:DEdef2.11}), on recalling (\ref{eq:LBop}), for the case $d=2$. Let the
mass lumped inner product $\langle\cdot,\cdot\rangle_{\Gamma^h(t)}^h$ 
on $\Gamma^h(t)$, for piecewise continuous functions with possible jumps
across the edges of $\{\sigma_j^h\}_{j=1}^{J_\Gamma}$, be defined by
\begin{equation} \label{eq:defNI}
\left\langle \eta, \zeta \right\rangle^h_{\Gamma^h(t)}  :=
\tfrac1d \sum_{j=1}^{J_\Gamma} \mathcal{H}^{d-1}(\sigma^h_j)\,\sum_{k=1}^{d} 
(\eta\,\zeta)((\vec{q}^h_{j_k})^-),
\end{equation}
where $\{\vec{q}^h_{j_k}\}_{k=1}^{d}$ are the vertices of $\sigma^h_j$,
and where
we define $(\eta\,\zeta)((\vec{q}^h_{j_k})^-):=
\underset{\sigma^h_j\ni \vec{p}\to \vec{q}^h_{j_k}}{\lim}\, 
(\eta\,\zeta)(\vec{p})$.

\begin{lem} \label{lem:DElem5.6NI}
It holds that
\begin{equation} \label{eq:DElem5.6NI}
\ddt \langle \eta, \zeta \rangle_{\Gamma^h(t)}^h
 = \langle \matpartxh\,\eta, \zeta \rangle_{\Gamma^h(t)}^h
 + \langle \eta, \matpartxh\,\zeta \rangle_{\Gamma^h(t)}^h
+ \langle \eta\,\zeta, \nabs\,.\,\vec{\mathcal{V}}^h \rangle_{\Gamma^h(t)}^h
\quad \forall\ \eta,\zeta \in W_T(\GhT)\,.
\end{equation}
In addition, if $d=2$, it holds that
\begin{equation} \label{eq:JWB}
\langle \zeta, \nabs\,.\,\vec\eta \rangle_{\Gamma^h(t)}
+ \langle \nabs\,\zeta, \vec\eta \rangle_{\Gamma^h(t)} = 
\langle \nabs\,\vec\id, \nabs\,\vec\pi^h\,(\zeta\,\vec\eta) 
\rangle_{\Gamma^h(t)}
\quad \forall\ \zeta \in \Wht\,, \vec\eta \in \Vht\,.
\end{equation}
\end{lem}
\begin{proof}
See the proof of Lemma~2.1 in \cite{tpfs}. 
\end{proof}

Similarly to (\ref{eq:Ps},b), we introduce
\begin{subequations}
\begin{equation} \label{eq:Psh}
\mat{\mathcal{P}}_{\Gamma^h} = \mat\Id - \vec \nu^h \otimes \vec \nu^h
\quad\text{on}\ \Gamma^h(t)\,,
\end{equation}
and
\begin{equation} \label{eq:Dsh}
\mat D_s^h(\vec \eta) = \tfrac12\,\mat{\mathcal{P}}_{\Gamma^h}\,
(\nabs\,\vec\eta + 
(\nabs\,\vec \eta)^T)\,\mat{\mathcal{P}}_{\Gamma^h}
\quad\text{on}\ \Gamma^h(t)\,,
\end{equation}
\end{subequations}
where here $\nabs = \mat{\mathcal{P}}_{\Gamma^h} \,\nabla$ 
denotes the surface gradient on $\Gamma^h(t)$. 
In addition, and similarly to (\ref{eq:hatDs}), we define
\begin{equation} \label{eq:hatDsh}
\mat {\hat D}_s^h(\vec \eta) 
= \mat D_s^h(\vec \eta) - \tfrac1{d-1}\left( \nabs\,.\,\vec \eta \right)
\mat{\mathcal{P}}_{\Gamma^h}
\quad\text{on}\ \Gamma^h(t)\,.
\end{equation}
Then it is straightforward to show that
\begin{align}
& 2\left\langle\mu_\Gamma(\chi)\, \mat D_s^h(\vec \eta), \mat D_s^h(\vec \eta) 
\right\rangle_{\Gamma^h(t)}^h 
+  \left\langle\lambda_\Gamma(\chi) \,
 \nabs\,.\, \vec \eta , \nabs\,.\, \vec \eta \right\rangle_{\Gamma^h(t)}^h 
\nonumber \\ & \hspace{1cm}
= 2\left\langle\mu_\Gamma(\chi)\, \mat {\hat D}_s^h(\vec \eta), 
\mat {\hat D}_s^h(\vec \eta) \right\rangle_{\Gamma^h(t)}^h 
+  \left\langle (\lambda_\Gamma(\chi) + \tfrac2{d-1}\,\mu_\Gamma(\chi))\,
 \nabs\,.\, \vec \eta , \nabs\,.\, \vec \eta \right\rangle_{\Gamma^h(t)}^h 
\nonumber \\ & \hspace{7cm}
\qquad \forall\ \vec\eta \in \Vht\,,\chi \in \Wht
\label{eq:HGh}
\end{align}
holds, which is a discrete analogue of (\ref{eq:HG}). 

Given $\Gamma^h(t)$, we 
let $\Omega^h_+(t)$ denote the exterior of $\Gamma^h(t)$ and let
$\Omega^h_-(t)$ denote the interior of $\Gamma^h(t)$, so that
$\Gamma^h(t) = \partial \Omega^h_-(t) = \overline{\Omega^h_-(t)} \cap 
\overline{\Omega^h_+(t)}$. 
We then partition the elements of the bulk mesh 
$\mathcal{T}^h$ into interior, exterior and interfacial elements as follows.
Let
\begin{align}
\mathcal{T}^h_-(t) & := \{ o \in \mathcal{T}^h : o \subset
\Omega^h_-(t) \} \,, \nonumber \\
\mathcal{T}^h_+(t) & := \{ o \in \mathcal{T}^h : o \subset
\Omega^h_+(t) \} \,, \nonumber \\
\mathcal{T}^h_{\Gamma^h}(t) & := \{ o \in \mathcal{T}^h : o \cap
\Gamma^h(t) \not = \emptyset \} \,. \label{eq:partTh}
\end{align}
Clearly $\mathcal{T}^h = \mathcal{T}^h_-(t) \cup \mathcal{T}^h_+(t) \cup
\mathcal{T}^h_{\Gamma}(t)$ is a disjoint partition.
In addition, we define the piecewise constant unit normal 
$\vec{\nu}^h(t)$ to $\Gamma^h(t)$ such that $\vec\nu^h(t)$ points into
$\Omega^h_+(t)$.
Moreover, we introduce the discrete density 
$\rho^h(t) \in S^h_0$ and the discrete viscosity $\mu^h(t) \in S^h_0$ as
\begin{equation} \label{eq:rhoha}
\rho^h(t)\!\mid_{o} = \begin{cases}
\rho_- & o \in \mathcal{T}^h_-(t)\,, \\
\rho_+ & o \in \mathcal{T}^h_+(t)\,, \\
\tfrac12\,(\rho_- + \rho_+) & o \in \mathcal{T}^h_{\Gamma^h}(t)\,,
\end{cases}
\quad\text{and}\quad
\mu^h(t)\!\mid_{o} = \begin{cases}
\mu_- & o \in \mathcal{T}^h_-(t)\,, \\
\mu_+ & o \in \mathcal{T}^h_+(t)\,, \\
\tfrac12\,(\mu_- + \mu_+) & o \in \mathcal{T}^h_{\Gamma^h}(t)\,.
\end{cases}
\end{equation}
Finally we note that from now on we assume that $\vec f_i \in
L^2(0,T;[C(\overline\Omega)]^d)$, $i=1,2$, so that 
$\vec I^h_2\,\vec f_i$, $i=1,2$, is well-defined for almost all $t \in (0,T)$.

In what follows we will introduce two different finite element approximations
for the free boundary problem
(\ref{eq:NSa}--e), (\ref{eq:sigma}), (\ref{eq:1c}--c), (\ref{eq:sigmaG}) and 
(\ref{eq:1surf}).
The first will be based on the weak formulation (\ref{eq:weakGDa0}--f), and the
second will be based on (\ref{eq:weaka0}--f). In each case, 
$\vec U^h(\cdot, t) \in \uspace^h$ will be an approximation to 
$\vec u(\cdot, t)$,
while $P^h(\cdot, t) \in \widehat\pspace^h(t)$ 
approximates $p(\cdot, t)$,
$\rho_\Gamma^h(\cdot, t) \in \Wht$
approximates $\rho_\Gamma(\cdot, t)$
and $\Psi^h(\cdot, t) \in \Wht$
approximates $\psi(\cdot, t)$.
When designing such a finite element approximation, a
careful decision has to be made about the {\em discrete tangential velocity} of
$\Gamma^h(t)$. The most natural choice is to select the velocity of the fluid,
i.e.\ $\vec{\mathcal{V}} = \vec u$ is appropriately discretized.
This leads to a discretization of (\ref{eq:weakGDa0}--f), where the arising
variational approximation of curvature, 
which directly discretizes $\vec\varkappa$, recall (\ref{eq:LBop}), 
goes back to the seminal paper \cite{Dziuk91}.
Overall, we obtain the following semidiscrete
continuous-in-time finite element approximation.

Given $\Gamma^h(0)$, $\rho_\Gamma^h(\cdot,0) \in \Whtz$, $\vec U^h(\cdot,0) \in \uspace^h$
and $\Psi^h(\cdot, 0) \in \Whtz$, 
find $\Gamma^h(t)$ such that $\vec \id \!\mid_{\Gamma^h(t)} \in \Vht$
for $t \in [0,T]$, and functions 
$\rho_\Gamma^h \in W_T(\GhT)$, 
$\vec U^h \in \utimespace^h_{\Gamma^h} := \{
\vec \phi \in H^1(0,T; \uspace^h) : 
\vec\chi \in [W_T(\GT)]^d \text{, where }  \vec\chi(\cdot,t) = 
\vec\pi^h\,[\vec\phi\!\mid_{\Gamma^h(t)}]\ \forall\ t\in[0,T]\}$, 
$P^h \in \pspace^h_T := \{ \varphi \in L^2(0,T; \widehat\pspace) :
\varphi(t) \in \widehat\pspace^h(t) \text{ for a.e. } t \in (0,T)\}$, 
$\vec\kappa^h \in [W(\GhT)]^d$
and $\Psi^h \in W_T(\GhT)$ 
such that for almost all $t \in (0,T)$ it holds that
\begin{subequations}
\begin{align}
& \ddt\left\langle \rho_\Gamma^h, \zeta \right\rangle_{\Gamma^h(t)}^h
= \left\langle \rho_\Gamma^h, \matpartxh\, \zeta \right\rangle_{\Gamma^h(t)}^h
\quad\forall\ \zeta \in W_T(\GhT)\,,
\label{eq:sdGDa0}\\
&
\tfrac12 \left[ \ddt \left( \rho^h\,\vec U^h , \vec \xi \right)
+ \left( \rho^h\,\vec U^h_t , \vec \xi \right)
- (\rho^h\,\vec U^h, \vec \xi_t) \right]
+ 2\left(\mu^h\,\mat D(\vec U^h), \mat D(\vec \xi) \right)
 \nonumber \\ & \qquad
+ \tfrac12\left(\rho^h, 
 [(\vec I^h_2\,\vec U^h\,.\,\nabla)\,\vec U^h]\,.\,\vec \xi
- [(\vec I^h_2\,\vec U^h\,.\,\nabla)\,\vec \xi]\,.\,\vec U^h \right)
- \left(P^h, \nabla\,.\,\vec \xi\right)
+ \ddt\left\langle \rho_\Gamma^h\,\vec U^h, \vec\xi
\right\rangle_{\Gamma^h(t)}^h
\nonumber \\ & \qquad
+ 2\left\langle \mu_\Gamma(\Psi^h) \, \mat D_s^h (\vec\pi^h\,\vec U^h) , 
\mat D_s^h (\vec\pi^h\, \vec \xi ) \right\rangle_{\Gamma^h(t)}^h
\nonumber \\ & \qquad
+ \left\langle\lambda_\Gamma(\Psi^h) \,
 \nabs\,.\, (\vec\pi^h\,\vec U^h) , \nabs\,.\, 
(\vec\pi^h\,\vec \xi) \right\rangle_{\Gamma^h(t)}^h
\nonumber \\ & \qquad
- \left\langle \gamma(\Psi^h)\,\vec\kappa^h + \nabs\,[\pi^h\,\gamma(\Psi^h)],
\vec \xi \right\rangle_{\Gamma^h(t)}^h 
= \left(\rho^h\,\vec f^h_1 + \vec f^h_2, \vec \xi\right)
+ \left\langle \rho_\Gamma^h\,\vec U^h, \matpartxh\, (\vec\pi^h\,\vec \xi) 
\right\rangle_{\Gamma^h(t)}^h \nonumber \\
& \hspace{9cm}\qquad \forall\ \vec\xi \in H^1(0,T;\uspace^h) \,, \label{eq:sdGDa}\\
& \left(\nabla\,.\,\vec U^h, \varphi\right) = 0 
\qquad \forall\ \varphi \in \widehat\pspace^h(t)\,,
\label{eq:sdGDb} \\
& \left\langle \vec{\mathcal{V}}^h ,
\vec\chi \right\rangle_{\Gamma^h(t)}^h
= \left\langle \vec U^h, \vec\chi \right\rangle_{\Gamma^h(t)}^h
 \qquad\forall\ \vec\chi \in \Vht\,,
\label{eq:sdGDc} 
\\
& \left\langle \vec\kappa^h , \vec\eta \right\rangle_{\Gamma^h(t)}^h
+ \left\langle \nabs\,\vec\id, \nabs\,\vec \eta \right\rangle_{\Gamma^h(t)}
 = 0  \qquad\forall\ \vec\eta \in \Vht\,,\label{eq:sdGDd} \\
& \ddt
\left\langle \Psi^h, \chi \right\rangle_{\Gamma^h(t)}^h
+ \Ds\left\langle \nabs\, \Psi^h, \nabs\, \chi
\right\rangle_{\Gamma^h(t)}
= 
\left\langle \Psi^h, \matpartxh\, \chi \right\rangle_{\Gamma^h(t)}^h 
\qquad \forall\ \chi \in W_T(\GhT)\,,
\label{eq:sdGDe}
\end{align}
\end{subequations}
where we recall (\ref{eq:Xht}). 
Here we have defined 
$\vec f^h_i(\cdot,t) := \vec I^h_2\,\vec f_i(\cdot,t)$, $i= 1\to 2$.
We observe that (\ref{eq:sdGDc}) collapses to
$\vec{\mathcal{V}}^h = \vec\pi^h\, \vec U^h\!\mid_{\Gamma^h(t)} \in \Vht$, 
which on recalling
(\ref{eq:matpartxh}) turns out to be crucial for
the stability analysis for (\ref{eq:sdGDa0}--f). It is for this reason that we
use mass lumping in (\ref{eq:sdGDc}).

In the following theorem we derive discrete analogues of (\ref{eq:d1}), 
and the surface mass conservation property in (\ref{eq:totalpsi}), as well as a
nonnegativity result for the discrete surface material density.
\begin{thm} \label{lem:stabGD}
Let $\{(\Gamma^h, \rho^h_\Gamma,\vec U^h, P^h, \vec\kappa^h, \Psi^h)(t)
\}_{t\in[0,T]}$ 
be a solution to {\rm (\ref{eq:sdGDa0}--f)}. Then
\begin{align} \label{eq:lemGD}
& \tfrac12\,\ddt \left(\|[\rho^h]^\frac12\,\vec U^h \|_0^2 
+ \left\langle \rho_\Gamma^h\,\vec U^h, \vec U^h \right\rangle_{\Gamma^h(t)}^h
\right)
+ 2\,\| [\mu^h]^\frac12\,\mat D(\vec U^h)\|_0^2 
\nonumber \\ & \hspace{1cm}
+ 2 \left\langle \mu_\Gamma(\Psi^h)\,\mat {\hat D}_s^h (\vec\pi^h\,\vec U^h) , 
\mat {\hat D}_s^h (\vec\pi^h\, \vec U^h ) \right\rangle_{\Gamma^h(t)}^h
\nonumber \\ & \hspace{1cm}
+  \left\langle (\lambda_\Gamma(\Psi^h) + \tfrac2{d-1}\,\mu_\Gamma(\Psi^h))\,
 \nabs\,.\, (\vec\pi^h\,\vec U^h) , \nabs\,.\, 
 (\vec\pi^h\,\vec U^h) \right\rangle_{\Gamma^h(t)}^h
\nonumber \\ & \hspace{3cm}
= \left(\rho^h\,\vec f_1^h + 
\vec f_2^h, \vec U^h \right) 
+ \left\langle \gamma(\Psi^h) \,\vec\kappa^h + 
\nabs\,\pi^h\,[\gamma (\Psi^h)] , \vec U^h \right\rangle_{\Gamma^h(t)}^h.
\end{align}
In addition, it holds that
\begin{equation} \label{eq:totalrhoh}
\ddt \left\langle \rho_\Gamma^h, 1 \right\rangle_{\Gamma^h(t)} = 0
\end{equation}
and
\begin{equation} \label{eq:sddmprho}
\rho_\Gamma^h(\cdot,t) \begin{cases} > 0 \\ \geq 0 \end{cases}
\quad \forall\ t \in (0,T]
 \qquad \text{if}\quad \rho_\Gamma^h(\cdot,0) 
\begin{cases} > 0 \\ \geq 0 \end{cases} \,.
\end{equation}
\end{thm} 
\begin{proof}
On recalling (\ref{eq:HGh}),
the desired result (\ref{eq:lemGD}) follows on choosing
$\vec \xi = \vec U^h$ in (\ref{eq:sdGDa}), $\varphi = P^h$ in (\ref{eq:sdGDb})
and $\zeta \in W_T(\GT)$ with 
$\zeta(\cdot,t) = -\frac12\,\pi^h\,[|\vec U^h\!\mid_{\Gamma^h(t)}|^2]$ 
for all $t\in[0,T]$, recall $\vec U^h \in \utimespace^h_{\Gamma^h}$,
in both 
(\ref{eq:sdGDa0}) and (\ref{eq:p1}), where we observe that the latter implies
that
\begin{equation} \label{eq:phU2}
\tfrac12\,\matpartxh\,\pi^h\,[\vec U^h|^2 = \pi^h\,[ \vec U^h\,.\,(\matpartxh\,
\vec\pi^h\,\vec U^h)] \quad \text{on }\ \Gamma^h(t)\,.
\end{equation}
In addition, the conservation property (\ref{eq:totalrhoh}) follows from
choosing $\zeta = 1$ in (\ref{eq:sdGDa0}).
Finally, it follows from (\ref{eq:sdGDa0}), on recalling (\ref{eq:mpbf}), that 
\begin{equation} \label{eq:rhohpos}
\ddt\left\langle \rho_\Gamma^h, \chi_k^h
\right\rangle_{\Gamma^h(t)}^h
=\ddt \left[ \left\langle 1, \chi^h_k \right\rangle_{\Gamma^h(t)}
\rho_\Gamma^h(\vec q^h_k(t), t) \right] = 0\,,
\end{equation}
for $k = 1,\ldots,K_\Gamma$, which yields our desired result 
(\ref{eq:sddmprho}).
\end{proof}

In the following two theorems we derive discrete analogues of (\ref{eq:d5}) 
for the scheme (\ref{eq:sdGDa0}--f). First we consider the case of constant
surface tension, recall (\ref{eq:Fconst}). 

\begin{thm} \label{thm:stabGDconst}
Let $\gamma$ be defined as in {\rm (\ref{eq:Fconst})}, let 
{\rm (\ref{eq:muconst})} hold and let \linebreak
$\{(\Gamma^h, \rho_\Gamma^h, \vec U^h, P^h, \vec\kappa^h)(t)\}_{t\in[0,T]}$ 
be a solution to {\rm (\ref{eq:sdGDa0}--e)}. Then
it holds that 
\begin{align} \label{eq:stabGDconst}
& \ddt\left(\tfrac12\,\|[\rho^h]^\frac12\,\vec U^h\|^2_{0} 
+ \tfrac12 
\left\langle \rho_\Gamma^h\,\vec U^h, \vec U^h \right\rangle_{\Gamma^h(t)}^h
+ \overline\gamma \, \mathcal{H}^{d-1}(\Gamma^h(t)) \right)
+ 2\,\| [\mu^h]^\frac12\,\mat D(\vec U^h)\|_0^2 
\nonumber \\ & \hspace{5mm}
+ 2\,\overline\mu_\Gamma \left\langle \mat {\hat D}_s^h (\vec\pi^h\,\vec U^h) , 
\mat {\hat D}_s^h (\vec\pi^h\, \vec U^h ) \right\rangle_{\Gamma^h(t)}
+ (\overline\lambda_\Gamma + \tfrac2{d-1}\,\overline\mu_\Gamma) 
\left\langle \nabs\,.\, (\vec\pi^h\,\vec U^h) , \nabs\,.\, 
(\vec\pi^h\,\vec U^h) \right\rangle_{\Gamma^h(t)}
\nonumber \\ & \hspace{4cm}
= (\rho^h\,\vec f_1^h + \vec f_2^h, \vec U^h) \,.
\end{align}
\end{thm} 
\begin{proof}
Similarly to (\ref{eq:constgam}), it follows from (\ref{eq:sdGDc},e) and
(\ref{eq:DElem5.6}) that
\begin{align}
\overline\gamma \left\langle \vec\kappa^h, \vec U^h \right\rangle_{\Gamma^h(t)}^h
&= \overline\gamma \left\langle \vec\kappa^h, \vec{\mathcal{V}}^h 
\right\rangle_{\Gamma^h(t)}^h
= - \overline\gamma \left\langle \nabs\,\vec\id, \nabs\,\vec{\cal V}^h 
\right\rangle_{\Gamma^h(t)}
=- \overline\gamma \left\langle 1, \nabs \,.\,\vec{\cal V}^h 
\right\rangle_{\Gamma^h(t)} \nonumber \\ &
= - \overline\gamma\, \ddt\,{\cal H}^{d-1}(\Gamma^h(t))\,.
\label{eq:constgamGD}
\end{align}  
Combining (\ref{eq:constgamGD}) and (\ref{eq:lemGD}) for the special case
(\ref{eq:Fconst}) yields the desired result (\ref{eq:stabGDconst}). 
\end{proof}

Next we generalize the results from Theorem~\ref{thm:stabGDconst} to the case
of a general surface tension function $\gamma$ as introduced in
(\ref{eq:gammaprime}), using the techniques introduced in 
\cite{tpfs}.
Here, similarly to (\ref{eq:d4}), it will be crucial to test (\ref{eq:sdGDe}) 
with an appropriate discrete variant of $F'(\Psi^h)$. It is for this reason
that we have to make the following well-posedness assumption:
\begin{equation} \label{eq:assPsih}
\Psi^h(\cdot, t) < \psi_\infty \quad\text{on $\Gamma^h(t)$}\,,
\quad \forall\ t \in [0,T]\,.
\end{equation}
The theorem also establishes nonnegativity of $\Psi^h$ 
under the assumption, if $\Ds >0$,
that
\begin{equation} \label{eq:chij}
\int_{\sigma^h_j(t)} \nabs \chi^h_i \,.\,\nabs \chi^h_k \dH{d-1} \leq 0 
\quad \forall\ i \neq k\,,\quad \forall\ t \in [0,T]\,,
\qquad j = 1,\ldots,J_\Gamma\,.
\end{equation}
We note that (\ref{eq:chij}) always holds for $d=2$, and it holds for $d=3$ if
all the triangles $\sigma^h_j(t)$ of $\Gamma^h(t)$ have no obtuse angles. A
direct consequence of (\ref{eq:chij}) is that for any monotonic
function $G \in C^{0,1}(\R)$ it holds for all $\xi \in \Wht$ that
\begin{align} \label{eq:LG}
& L_G\,\int_{\sigma^h_j(t)}
\nabs\, \xi\,.\, \nabs\, \pi^h\,[G(\xi)] \dH{d-1} \geq
\int_{\sigma^h_j(t)} \nabs\, \pi^h\,[G(\xi)] \,.\, \nabs\, \pi^h\,[G(\xi)] 
\dH{d-1} 
\quad \forall\ t \in [0,T]\,, \nonumber \\ & \hspace{11cm} 
\quad j = 1,\ldots,J_\Gamma\,,
\end{align}
where $L_G \in \R_{>0}$ denotes the Lipschitz constant of $G$.
For example, (\ref{eq:LG}) holds for 
\begin{equation} \label{eq:pm}
G(r) = [r]_\pm := \pm \max\{0, \pm r\} \qquad \forall\ r \in \R
\end{equation}
with $L_G = 1$.

For the following theorem, we denote the $L^\infty$--norm on $\Gamma^h(t)$ by
$\| \cdot \|_{\infty, \Gamma^h(t)}$, i.e.\
$\| z \|_{\infty, \Gamma^h(t)} := \esssup_{\Gamma^h(t)} |z|$ for
$z : \Gamma^h(t) \to \R$.

\begin{thm} \label{thm:stabGD}
Let $\{(\Gamma^h, \rho_\Gamma^h, \vec U^h, P^h, \vec\kappa^h, \Psi^h)(t)\}_{t\in[0,T]}$ 
be a solution to {\rm (\ref{eq:sdGDa0}--f)}. Then
\begin{equation} \label{eq:totalpsih}
\ddt \left\langle \Psi^h, 1 \right\rangle_{\Gamma^h(t)} = 0\,.
\end{equation}
In addition, if $\Ds=0$ or if {\rm (\ref{eq:chij})} and
\begin{equation} \label{eq:Xinfty}
\max_{0\leq t \leq T} \| \nabs\,.\,\vec{\mathcal{V}}^h \|_{\infty,\Gamma^h(t)} <
\infty
\end{equation}
hold, then
\begin{equation} \label{eq:sddmp}
\Psi^h(\cdot,t) \geq 0 \quad \forall\ t \in (0,T]
 \qquad \text{if}\quad \Psi^h(\cdot,0) \geq 0\,.
\end{equation}
Moreover, if $d=2$ and if {\rm (\ref{eq:sddmp})} and
{\rm (\ref{eq:assPsih})} hold, 
then
\begin{align} \label{eq:stabGD}
& \ddt\left(\tfrac12\,\|[\rho^h]^\frac12\,\vec U^h\|^2_{0} 
+ \tfrac12 
\left\langle \rho_\Gamma^h\,\vec U^h, \vec U^h \right\rangle_{\Gamma^h(t)}^h
+ \left\langle F(\Psi^h) , 1 \right\rangle_{\Gamma^h(t)}^h \right) 
+ 2\,\| [\mu^h]^\frac12\,\mat D(\vec U^h)\|_0^2 
\nonumber \\ & \hspace{5mm}
+ 2 \left\langle \mu_\Gamma(\Psi^h)\, \mat {\hat D}_s^h (\vec\pi^h\,\vec U^h) , 
\mat {\hat D}_s^h (\vec\pi^h\, \vec U^h ) \right\rangle_{\Gamma^h(t)}^h
\nonumber \\ & \hspace{5mm}
+ \left\langle (\lambda_\Gamma(\Psi^h) + \tfrac2{d-1}\,\mu_\Gamma(\Psi^h))\,
 \nabs\,.\, (\vec\pi^h\,\vec U^h) , \nabs\,.\, 
(\vec\pi^h\,\vec U^h) \right\rangle_{\Gamma^h(t)}^h
\leq (\rho^h\,\vec f_1^h + \vec f_2^h, \vec U^h) \,.
\end{align}
\end{thm} 
\begin{proof}
The conservation property (\ref{eq:totalpsih}) follows immediately from
choosing $\chi = 1$ in (\ref{eq:sdGDe}).
A proof of the result (\ref{eq:sddmp}) can be found in 
\citet[Theorem~3.3]{tpfs}. Also in \citet[Theorem~3.3]{tpfs},
on using (\ref{eq:LG}), it was shown that
\begin{align} \label{eq:sp40}
\ddt\, \left\langle F(\Psi^h ), 1 \right\rangle_{\Gamma^h(t)}^h & \leq
- \left\langle\gamma(\Psi^h)\,\vec\kappa^h + \nabs\,\pi^h\,[\gamma(\Psi^h)], 
\vec U^h \right\rangle_{\Gamma^h(t)}^h ,
\end{align}
which is a discrete analogue of (\ref{eq:d45}). Combining (\ref{eq:sp40}) 
with (\ref{eq:lemGD}) yields the desired result (\ref{eq:stabGD}).
\end{proof}

We note that while (\ref{eq:sdGDa0}--f) is a very natural approximation,
a drawback in practice is that the finitely many vertices of 
the triangulations $\Gamma^h(t)$ are moved with the flow, which can lead to
coalescence. If a remeshing procedure is applied to $\Gamma^h(t)$, then
theoretical results like stability are no longer valid.

It is with this in mind that we would like to introduce an alternative finite
element approximation. It will be based on the weak formulation 
(\ref{eq:weaka0}--f), and on the schemes from \cite{spurious,fluidfbp} for the
two-phase flow problem in the bulk. 

The main difference to (\ref{eq:sdGDa0}--f) is that (\ref{eq:sdGDc}) is replaced
with a discrete variant of (\ref{eq:weakc}). In particular, the discrete
tangential velocity of $\Gamma^h(t)$ is not defined via $\vec U^h(\cdot, t)$, 
but it
is chosen totally independent from the surrounding fluid. In fact, the discrete
tangential velocity is not prescribed directly, but it is implicitly
introduced via the novel approximation of curvature which was first introduced
by the authors in \cite{triplej} for the case $d=2$, and in \cite{gflows3d} for
the case $d=3$. This discrete tangential velocity is such that, 
in the case $d=2$, $\Gamma^h(t)$
will remain equidistributed for all times $t \in (0,T]$. For $d=3$, a weaker 
property can be shown, which still guarantees good meshes in practice.
We refer to \cite{triplej,gflows3d} for more details.

Following similar ideas in \cite{surf,surf2d}, we introduce regularizations
$F_\epsilon \in C^2(-\infty,\psi_\infty)$ of $F\in C^2(0,\psi_\infty)$, where
$\epsilon > 0$ is a regularization parameter. In particular, we set
\begin{subequations}
\begin{equation} \label{eq:Freg}
F_\epsilon(r) = \begin{cases}
F(r) & r \geq \epsilon\,, \\
F(\epsilon) + F'(\epsilon)\,(r-\epsilon) +
\frac12\,F''(\epsilon)\,(r-\epsilon)^2 & r \leq \epsilon\,,
\end{cases}
\end{equation}
which in view of (\ref{eq:F}) leads to
\begin{equation} \label{eq:geps}
\gamma_\epsilon(r) = \begin{cases}
\gamma(r) & r \geq \epsilon\,, \\
\gamma(\epsilon) + 
\frac12\,F''(\epsilon)\,(\epsilon^2 - r^2) & r \leq \epsilon\,,
\end{cases}
\end{equation}
\end{subequations}
so that
\begin{equation} \label{eq:Feps}
\gamma_\epsilon(r) = F_\epsilon(r) - r\,F'_\epsilon(r) 
\quad\text{and}\quad \gamma_\epsilon'(r) = - r\,F_\epsilon''(r) 
\qquad \forall\ r < \psi_\infty \,.
\end{equation}
We also introduce the matrix functions 
$\mat\Xi^h(\cdot,t) : \Wht \to [L^\infty(\Gamma^h(t))]^{d \times d}$ 
defined such that for
all $z^h \in \Wht$ it holds that
\begin{equation} \label{eq:Xih}
\mat\Xi^h(z^h,t) \! \mid_{\sigma^h_j(t)} \in \R^{d\times d} \quad\text{and}\quad
\mat\Xi^h(z^h,t) \, \nabs\,z^h
= \tfrac12\,\nabs\,\pi^h \,[|z^h|^2] \ \text{ on } \sigma^h_j(t)\,,
j = 1,\ldots,J_\Gamma\,.
\end{equation}
Here we introduce (\ref{eq:Xih}) in order to be able to mimic
(\ref{eq:nbgn}) on the discrete level.
The construction for $\mat\Xi^h$ is given as follows. Let $\hat\sigma$
denote the standard $(d-1)$-dimensional reference simplex in 
$\R^{d-1} \times \{0\} \subset \R^d$, 
with vertices $\{\vec 0, \vec\ek_1, \ldots, \vec\ek_{d-1} \}$. 
For each $\sigma = \sigma^h_j(t)$, $j= 1,\ldots,J_\Gamma$,
with vertices $\{ \vec {\rm p}_i \}_{i=0}^{d-1}$
there exists an affine linear map 
$\vec{\mathcal{M}}_\sigma : \hat\sigma \to \sigma$ with
$\vec{\mathcal{M}}_\sigma (\vec z) = \vec {\rm p}_0 + \mat M_\sigma\,\vec z$ 
for all $\vec z \in \R^d$, where $\mat M_\sigma \in \R^{d \times d}$ is 
nonsingular, such that
$\vec{\mathcal{M}}_\sigma (\vec \ek_i) = \vec {\rm p}_i$, $i=1,\ldots,d-1$. 
In particular, the columns of $\mat M_\sigma$ are given by
$\vec {\rm p}_i - \vec {\rm p}_0$, $i=1,\ldots,d$, where
$\vec {\rm p}_d \in \R^d$ is an arbitrary point that does not lie within the
hyperplane that contains $\sigma$.
On choosing $\vec {\rm p}_d$ such that 
$(\vec {\rm p}_d - \vec {\rm p}_0) \,.\,(\vec {\rm p}_i - \vec {\rm p}_0) = 0$
for $i=1,\ldots,d-1$, we observe that
$\nabs\,\xi  = (\mat M_\sigma^T)^{-1}\,
[\nabs\,(\xi \circ \vec{\mathcal{M}}_\sigma)] 
\circ (\vec{\mathcal{M}}_\sigma)^{-1}$
on $\sigma$, where we note that 
$\nabs\, \eta= \nabla\,\eta - (\vec\ek_d\,.\,\nabla\,\eta)\,\vec\ek_d$ 
on $\hat\sigma$.
Hence we define
\begin{subequations}
\begin{equation} \label{eq:Xiho}
\mat\Xi^h( z^h , t) \!\mid_\sigma = 
(\mat M_\sigma^T)^{-1}\,\mat{\hat\Xi}^h_\sigma (z^h)\,\mat M_\sigma^T\,,
\end{equation}
where $\mat{\hat\Xi}^h_{\sigma} (z^h) \in \R^{d \times d}$ is the diagonal matrix
with entries
\begin{equation} \label{eq:Xihojj}
[\mat{\hat\Xi}^h_{\sigma} (z^h)]_{ii} = 
\begin{cases}
\frac12\,(z^h(\vec {\rm p}_0) + z^h(\vec {\rm p}_i)) & i = 1,\ldots, d-1\,, \\
0 & i = d \,.
\end{cases}
\end{equation}
\end{subequations}

We propose the following semidiscrete
analogue of the weak formulation (\ref{eq:weaka0}--f).
Given $\Gamma^h(0)$, $\rho_\Gamma^h(\cdot,0) \in \Whtz$, $\vec U^h(\cdot,0) \in \uspace^h$
and $\Psi^h(\cdot, 0) \in \Whtz$, 
find $\Gamma^h(t)$ such that $\vec \id \!\mid_{\Gamma^h(t)} \in \Vht$
for $t \in [0,T]$, and functions 
$\rho_\Gamma^h \in W_T(\GhT)$, 
$\vec U^h \in \utimespace^h_{\Gamma^h}$,
$P^h \in \pspace^h_T$, 
$\kappa^h \in W(\GhT)$
and $\Psi^h \in W_T(\GhT)$ 
such that for almost all $t \in (0,T)$ it holds that
\begin{subequations}
\begin{align}
& \ddt\left\langle \rho_\Gamma^h, \zeta \right\rangle_{\Gamma^h(t)}^h
= \left\langle \rho_\Gamma^h, \matpartxh\, \zeta \right\rangle_{\Gamma^h(t)}^h
- \left\langle\rho^h_{\Gamma,\star} , (\vec{\mathcal{V}}^h - \vec U^h) \,.\, 
\nabs\,\zeta \right\rangle_{\Gamma^h(t)}^h
\quad\forall\ \zeta \in W_T(\GhT)\,,
\label{eq:sdHGa0}\\
&\tfrac12 \left[ \ddt \left( \rho^h\,\vec U^h , \vec \xi \right)
+ \left( \rho^h\,\vec U^h_t , \vec \xi \right)
- (\rho^h\,\vec U^h, \vec \xi_t) \right]
+ 2\left(\mu^h\,\mat D(\vec U^h), \mat D(\vec \xi) \right)
\nonumber \\ & \quad
+ \tfrac12\left(\rho^h, 
 [(\vec I^h_2\,\vec U^h\,.\,\nabla)\,\vec U^h]\,.\,\vec \xi
- [(\vec I^h_2\,\vec U^h\,.\,\nabla)\,\vec \xi]\,.\,\vec U^h \right)
- \left(P^h, \nabla\,.\,\vec \xi\right) 
+ \ddt\left\langle \rho_\Gamma^h\,\vec U^h, \vec\xi
\right\rangle_{\Gamma^h(t)}^h
\nonumber \\ & \quad
+ 2 \left\langle \mu_\Gamma(\Psi^h)\, \mat D_s^h (\vec\pi^h\,\vec U^h) , 
\mat D_s^h (\vec\pi^h\, \vec \xi ) \right\rangle_{\Gamma^h(t)}^h
\nonumber \\ & \quad
+  \left\langle \lambda_\Gamma(\Psi^h) \,
\nabs\,.\, (\vec\pi^h\,\vec U^h) , \nabs\,.\, 
(\vec\pi^h\,\vec \xi) \right\rangle_{\Gamma^h(t)}^h
\nonumber \\ & \quad
- \left\langle \pi^h\,[\gamma_\epsilon(\Psi^h)\,\kappa^h]\,\vec\nu^h,
\vec \xi \right\rangle_{\Gamma^h(t)}
- \left\langle \nabs\,[\pi^h\,\gamma_\epsilon(\Psi^h)],
\vec \xi \right\rangle_{\Gamma^h(t)}^h
\nonumber \\ & \quad\ 
= \left(\rho^h\,\vec f^h_1 + \vec f^h_2, \vec \xi\right)
+ \left\langle \rho_\Gamma^h \,\vec U^h, \matpartxh\,(\vec\pi^h\,\vec \xi) 
\right\rangle_{\Gamma^h(t)}^h
\nonumber \\ & \quad\quad 
- \sum_{i=1}^d
 \left\langle\rho_{\Gamma,\star}^h\,( \vec{\mathcal{V}}^h - \vec U^h) ,  
 \mat\Xi^h(\pi^h\,U^h_i)\, \nabs\,(\pi^h\,\xi_i) \right\rangle_{\Gamma^h(t)}^h
\qquad \forall\ \vec\xi \in H^1(0,T;\uspace^h)
\,, \label{eq:sdHGa}\\
& \left(\nabla\,.\,\vec U^h, \varphi\right)  = 0 
\quad \forall\ \varphi \in \widehat\pspace^h(t)\,,
\label{eq:sdHGb} \\
& \left\langle \vec{\mathcal{V}}^h ,
\chi\,\vec\nu^h \right\rangle_{\Gamma^h(t)}^h
= \left\langle \vec U^h, 
\chi\,\vec\nu^h \right\rangle_{\Gamma^h(t)} 
 \quad\forall\ \chi \in \Wht\,,
\label{eq:sdHGc} \\
& \left\langle \kappa^h\,\vec\nu^h, \vec\eta \right\rangle_{\Gamma^h(t)}^h
+ \left\langle \nabs\,\vec\id, \nabs\,\vec \eta \right\rangle_{\Gamma^h(t)}
 = 0  \quad\forall\ \vec\eta \in \Vht\,, \label{eq:sdHGd} \\
& \ddt
\left\langle \Psi^h, \chi \right\rangle_{\Gamma^h(t)}^h 
+ \Ds\left\langle \nabs\, \Psi^h, \nabs\, \chi
\right\rangle_{\Gamma^h(t)}
\nonumber \\ & \hspace{1cm}
= \left\langle \Psi^h, \matpartxh\, \chi \right\rangle_{\Gamma^h(t)}^h 
- \left\langle \Psi^h_{\star,\epsilon}, \left( \vec{\mathcal{V}}^h - \vec U^h \right) 
.\, \nabs\,\chi \right\rangle_{\Gamma^h(t)}^h
\qquad \forall\ \chi \in W_T(\GhT)\,,\label{eq:sdHGe}
\end{align}
\end{subequations}
where we recall (\ref{eq:Xht}), and where e.g.\ 
$\vec U^h = (U^h_1,\ldots,U^h_d)^T$. 
The value $\Psi^h_{\star,\epsilon}$ in (\ref{eq:sdHGe}) is chosen in a special
way to enable us to prove stability for the scheme (\ref{eq:sdHGa0}--f). As we
are unable to prove stability for $d=3$ for general surface tensions,
due to the need for (\ref{eq:JWB}), we simply set 
$\Psi^h_{\star,\epsilon} = \Psi^h$ if $d=3$. For $d=2$, on recalling 
(\ref{eq:Feps}), we define
\begin{equation} \label{eq:Psihstar}
\Psi^h_{\star,\epsilon} = \begin{cases}
- \frac{\gamma_\epsilon(\Psi^h_k) - \gamma_\epsilon(\Psi^h_{k-1})}
{F'_\epsilon(\Psi^h_k) - F'_\epsilon(\Psi^h_{k-1})} & 
F'_\epsilon(\Psi^h_{k-1}) \not= F'_\epsilon(\Psi^h_k)\,, \\
\frac12\,(\Psi^h_{k-1} + \Psi^h_k) 
&  F'_\epsilon(\Psi^h_{k-1}) = F'_\epsilon(\Psi^h_k)\,,
\end{cases}
\quad\text{on}\quad [\vec q^h_{k-1}, \vec q^h_{k}]
\quad\forall\ k \in \{1,\ldots,K_\Gamma\} \,.
\end{equation}
Here we have introduced the shorthand notation
$\Psi^h_k(t) = \Psi^h(\vec q^h_k(t), t)$, for $k=1,\ldots,K_\Gamma$,
and for notational convenience we have
dropped the dependence on $t$ in (\ref{eq:Psihstar}). 
The definition in (\ref{eq:Psihstar}) is chosen such that for $d=2$
it holds that
\begin{align} \label{eq:doctored}
& \left\langle \Psi^h_{\star,\epsilon}\, \vec\eta ,
\nabs\,\pi^h\,[F'_\epsilon(\Psi^h)] \right\rangle_{\Gamma^h(t)}^h
= \left\langle \Psi^h_{\star,\epsilon}\, \vec\eta ,
\nabs\,\pi^h\,[F'_\epsilon(\Psi^h)] \right\rangle_{\Gamma^h(t)}
= - \left\langle \vec\eta, \nabs\,\pi^h\,[\gamma_\epsilon(\Psi^h)]
\right\rangle_{\Gamma^h(t)} \nonumber\\ & \hspace{11cm}\forall\ \vec\eta \in \Vht\,,
\end{align}
which will be crucial for the stability proof for (\ref{eq:sdHGa0}--f). 
Note that here the regularization (\ref{eq:Freg},b) is required in order to 
make the definition (\ref{eq:Psihstar}) well-defined. 
We observe that (\ref{eq:doctored}) for 
$\vec\eta = \vec{\mathcal{V}}^h - \vec\pi^h\,\vec U^h \!\mid_{\Gamma^h(t)}$ 
mimics (\ref{eq:dbgn}) on the discrete level.
In addition $\rho^h_{\Gamma,\star}$ in (\ref{eq:sdHGa0},b) is defined by
\begin{equation} \label{eq:rhohstar}
\rho^h_{\Gamma,\star} = 
\begin{cases}
\frac1{\mathcal{H}^{d-1}(\sigma^h_j)}\,
\int_{\sigma^h_j} \rho^h_\Gamma  \dH{d-1} & \rho^h_\Gamma \geq 0\
\text{on}\ \overline{\sigma^h_j}\,,\\
0 & \min_{\overline{\sigma^h_j}} \rho^h_\Gamma < 0\,,
\end{cases}
\quad\text{on}\quad \sigma^h_j
\quad\forall\ j \in \{1,\ldots,J_\Gamma\} \,.
\end{equation}

In the following lemma we derive a discrete analogue of (\ref{eq:d1}),
as well as a discrete surface mass conservation property, for the
scheme (\ref{eq:sdHGa0}--f). 
\begin{thm} \label{lem:stabHG}
Let $\{(\Gamma^h, \rho^h_\Gamma,\vec U^h, P^h, \kappa^h, \Psi^h)(t)
\}_{t\in[0,T]}$ 
be a solution to {\rm (\ref{eq:sdHGa0}--f)}. Then
\begin{align} \label{eq:lemHG}
& \tfrac12\,\ddt \left(\|[\rho^h]^\frac12\,\vec U^h \|_0^2 +
\left\langle \rho_\Gamma^h\,\vec U^h, \vec U^h \right\rangle_{\Gamma^h(t)}^h
\right)
+ 2\,\| [\mu^h]^\frac12\,\mat D(\vec U^h)\|_0^2 
\nonumber \\ & \hspace{1cm}
+ 2 \left\langle\mu_\Gamma(\Psi^h)\, \mat {\hat D}_s^h (\vec\pi^h\,\vec U^h) , 
\mat {\hat D}_s^h (\vec\pi^h\, \vec U^h ) \right\rangle_{\Gamma^h(t)}^h
\nonumber \\ & \hspace{1cm}
+  \left\langle (\lambda_\Gamma(\Psi^h) + \tfrac2{d-1}\,\mu_\Gamma(\Psi^h))\,
\nabs\,.\, (\vec\pi^h\,\vec U^h) , \nabs\,.\, 
(\vec\pi^h\,\vec U^h) \right\rangle_{\Gamma^h(t)}^h
\nonumber \\ & \hspace{1.5cm}
= \left(\rho^h\,\vec f_1^h + 
\vec f_2^h, \vec U^h \right) 
+ \left\langle \pi^h\,[\gamma_\epsilon(\Psi^h) \,\kappa^h]\,\vec\nu^h, \vec U^h
\right\rangle_{\Gamma^h(t)} + \left\langle
\nabs\,\pi^h\,[\gamma_\epsilon(\Psi^h)] , \vec U^h 
\right\rangle_{\Gamma^h(t)}^h .
\end{align}
In addition, it holds that 
\begin{equation} \label{eq:totalrhohHG}
\ddt \left\langle \rho_\Gamma^h, 1 \right\rangle_{\Gamma^h(t)} = 0
\end{equation}
and, if
\begin{equation} \label{eq:XinftyHG}
\max_{0\leq t \leq T} \| \nabs\,.\,\vec{\mathcal{V}}^h \|_{\infty,\Gamma^h(t)} 
< \infty\,,
\end{equation}
then
\begin{equation} \label{eq:sddmprhoHG}
\rho_\Gamma^h(\cdot,t) \geq 0 
\quad \forall\ t \in (0,T]
 \qquad \text{if}\quad \rho_\Gamma^h(\cdot,0) 
\geq 0 \,.
\end{equation}
\end{thm} 
\begin{proof}
On recalling (\ref{eq:HGh}), 
the desired result (\ref{eq:lemHG}) follows on choosing
$\vec \xi = \vec U^h$ in (\ref{eq:sdHGa}), $\varphi = P^h$ in (\ref{eq:sdHGb})
and $\zeta \in W_T(\GT)$ with 
$\zeta(\cdot,t) = -\frac12\,\pi^h\,[|\vec U^h\!\mid_{\Gamma^h(t)}|^2]$ 
for all $t\in[0,T]$, recall $\vec U^h \in \utimespace^h_{\Gamma^h}$,
in (\ref{eq:sdHGa0}),
where we recall (\ref{eq:phU2}) and (\ref{eq:Xih}). 

The conservation property (\ref{eq:totalrhohHG}) follows from
choosing $\zeta = 1$ in (\ref{eq:sdHGa0}). Moreover, choosing 
$\zeta = \pi^h [ \rho_\Gamma^h]_{-}$ in (\ref{eq:sdHGa0}) yields,
on recalling (\ref{eq:p1}) and (\ref{eq:DElem5.6NI}), that
\begin{align}
& \ddt \left\langle [\rho^h_\Gamma]_-^2, 1 \right\rangle_{\Gamma^h(t)}^h 
+ \left\langle\rho^h_{\Gamma,\star} , (\vec{\mathcal{V}}^h - \vec U^h) \,.\, 
\nabs\,\pi^h\,[\rho^h_\Gamma]_- \right\rangle_{\Gamma^h(t)}^h
= \left\langle \rho_\Gamma^h, 
\matpartxh\, \pi^h\,[\rho^h_\Gamma]_- \right\rangle_{\Gamma^h(t)}^h 
\nonumber \\ & \qquad
= \tfrac12
\left\langle \matpartxh\,\pi^h\left[[\rho_\Gamma^h]_-^2\right], 1
\right\rangle_{\Gamma^h(t)}^h  
= \tfrac12\,\ddt 
\left\langle [\rho^h_\Gamma]_-^2, 1 \right\rangle_{\Gamma^h(t)}^h 
- \tfrac12 \left\langle [\rho_\Gamma^h]_-^2, \nabs\,.\, \vec{\mathcal{V}}^h
\right\rangle_{\Gamma^h(t)}^h .
\label{eq:rhonew1}
\end{align}
It follows from (\ref{eq:rhohstar}) that the second term on the left hand side
of (\ref{eq:rhonew1}) vanishes, and hence we obtain that
\begin{equation} \label{eq:rhonew2}
\ddt \left\langle [\rho^h_\Gamma]_-^2, 1 \right\rangle_{\Gamma^h(t)}^h 
= - \left\langle [\rho_\Gamma^h]_-^2, \nabs\,.\, \vec{\mathcal{V}}^h
\right\rangle_{\Gamma^h(t)}^h
\leq \| \nabs\,.\,\vec{\mathcal{V}}^h \|_{\infty,\Gamma^h(t)}
\left\langle [\rho^h_\Gamma]_-^2 , 1 \right\rangle_{\Gamma^h(t)}^h\,.
\end{equation}
A Gronwall inequality, together with (\ref{eq:XinftyHG}),
now yields our desired result (\ref{eq:sddmprhoHG}). 
\end{proof}

In the following two theorems we derive discrete analogues of (\ref{eq:d5}) 
for the scheme (\ref{eq:sdHGa0}--f). First we consider the case of constant
surface tension, recall (\ref{eq:Fconst}). 

\begin{thm} \label{thm:stabHGconst}
Let $\gamma$ be defined as in {\rm (\ref{eq:Fconst})}, let 
{\rm (\ref{eq:muconst})} hold and let \linebreak
$\{(\Gamma^h, \rho_\Gamma^h, \vec U^h, P^h, \kappa^h)(t)\}_{t\in[0,T]}$ 
be a solution to {\rm (\ref{eq:sdHGa0}--e)}. Then
it holds that 
\begin{align} \label{eq:stabHGconst}
& \ddt\left(\tfrac12\,\|[\rho^h]^\frac12\,\vec U^h\|^2_{0} 
+ \tfrac12
 \left\langle \rho_\Gamma^h\,\vec U^h, \vec U^h \right\rangle_{\Gamma^h(t)}^h
+ \overline\gamma \, \mathcal{H}^{d-1}(\Gamma^h(t)) \right)
+ 2\,\| [\mu^h]^\frac12\,\mat D(\vec U^h)\|_0^2 
\nonumber \\ & \hspace{5mm}
+ 2\,\overline\mu_\Gamma \left\langle \mat {\hat D}_s^h (\vec\pi^h\,\vec U^h) , 
\mat {\hat D}_s^h (\vec\pi^h\, \vec U^h ) \right\rangle_{\Gamma^h(t)}
+ (\overline\lambda_\Gamma + \tfrac2{d-1}\,\overline\mu_\Gamma )
\left\langle \nabs\,.\, (\vec\pi^h\,\vec U^h) , \nabs\,.\, 
(\vec\pi^h\,\vec U^h) \right\rangle_{\Gamma^h(t)}
\nonumber \\ & \hspace{4cm}
= (\rho^h\,\vec f_1^h + \vec f_2^h, \vec U^h) \,.
\end{align}
\end{thm} 
\begin{proof}
Similarly to (\ref{eq:constgamGD}), it follows from 
$\gamma_\epsilon(\cdot) = \gamma(\cdot) =\overline\gamma$, (\ref{eq:sdHGc},e) and
(\ref{eq:DElem5.6}) that
\begin{align}
\overline\gamma \left\langle \kappa^h\,\vec\nu^h,\vec U^h \right\rangle_{\Gamma^h(t)}
&= \overline\gamma \left\langle \kappa^h\,\vec\nu^h, \vec{\mathcal{V}}^h 
\right\rangle_{\Gamma^h(t)}^h
= - \overline\gamma \left\langle \nabs\,\vec\id, \nabs\,\vec{\cal V}^h 
\right\rangle_{\Gamma^h(t)} \nonumber \\ &
=- \overline\gamma \left\langle 1, \nabs \,.\,\vec{\cal V}^h 
\right\rangle_{\Gamma^h(t)} 
= - \overline\gamma\, \ddt\,{\cal H}^{d-1}(\Gamma^h(t))\,.
\label{eq:constgamHG}
\end{align}  
Combining (\ref{eq:constgamHG}) and (\ref{eq:lemHG}) for the special case
(\ref{eq:Fconst}) yields the desired result (\ref{eq:stabHGconst}). 
\end{proof}

Next we generalize the results from Theorem~\ref{thm:stabHGconst} to the case
of a general surface tension function $\gamma$ as introduced in 
(\ref{eq:gammaprime}). 
\begin{thm} \label{thm:stabHG}
Let $\{(\Gamma^h, \rho^h_\Gamma, \vec U^h, P^h, \kappa^h, \Psi^h)(t)
\}_{t\in[0,T]}$ be a solution to {\rm (\ref{eq:sdHGa0}--f)}. Then
\begin{equation} \label{eq:totalPsih}
\ddt \left\langle \Psi^h, 1 \right\rangle_{\Gamma^h(t)} = 0\,.
\end{equation}
Moreover, if $\charfcn{\Omega_-^h(t)} \in \pspace^h(t)$ then 
\begin{equation}
\ddt\, \vol(\Omega_-^h(t)) = 0\,. \label{eq:cons}
\end{equation}
In addition, if $d=2$ and if the assumption {\rm (\ref{eq:assPsih})} holds,
then
\begin{align}
& \ddt\left(\tfrac12\,\|[\rho^h]^\frac12\,\vec U^h\|^2_{0} 
+ \tfrac12
 \left\langle \rho_\Gamma^h\,\vec U^h, \vec U^h \right\rangle_{\Gamma^h(t)}^h
+ \left\langle F_\epsilon(\Psi^h) , 1 \right\rangle_{\Gamma^h(t)}^h \right) 
+ 2\,\|[\mu^h]^\frac12\,\mat D(\vec U^h)\|^2_{0}
\nonumber \\ & \hspace{5mm}
+ 2 \left\langle\mu_\Gamma(\Psi^h)\, \mat {\hat D}_s^h (\vec\pi^h\,\vec U^h) , 
\mat {\hat D}_s^h (\vec\pi^h\, \vec U^h ) \right\rangle_{\Gamma^h(t)}^h
\nonumber \\ & \hspace{5mm}
+  \left\langle (\lambda_\Gamma(\Psi^h) + \tfrac2{d-1}\,\mu_\Gamma(\Psi^h)) \,
\nabs\,.\, (\vec\pi^h\,\vec U^h) , \nabs\,.\, 
(\vec\pi^h\,\vec U^h) \right\rangle_{\Gamma^h(t)}^h
\leq \left(\rho^h\,\vec f^h_1 + \vec f^h_2, \vec U^h\right) .
\label{eq:stabHG}
\end{align}
\end{thm}
\begin{proof}
The conservation property (\ref{eq:totalPsih}) follows immediately from
choosing $\chi = 1$ in (\ref{eq:sdHGe}). Moreover, 
choosing $\chi = 1$ in (\ref{eq:sdHGc}) and
$\varphi= (\charfcn{\Omega_-^h(t)} -
\frac{\mathcal{L}^d(\Omega_-^h(t))}{\mathcal{L}^d(\Omega)})
\in \widehat\pspace^h(t)$ in (\ref{eq:sdHGb}), we obtain that
\begin{equation*}
\frac{\rm d}{{\rm d}t} \vol(\Omega_-^h(t)) = 
\left\langle \vec{\mathcal{V}}^h , \vec\nu^h \right\rangle_{\Gamma^h(t)}
= \left\langle \vec{\mathcal{V}}^h , \vec\nu^h \right\rangle^h_{\Gamma^h(t)}
= \left\langle \vec U^h, \vec\nu^h \right\rangle_{\Gamma^h(t)}
= \int_{\Omega_-^h(t)} \nabla\,.\,\vec U^h \dL{d} 
=0\,, 
\end{equation*}
which proves the desired result (\ref{eq:cons}). 
In \citet[Theorem~3.7]{tpfs} it was shown that
\begin{align}
& \ddt\, \left\langle F_\epsilon(\Psi^h), 1 \right\rangle_{\Gamma^h(t)}^h +
\Ds \left\langle \nabs\, \Psi^h, \nabs\, \pi^h\,[F'_\epsilon(\Psi^h)]
\right\rangle_{\Gamma^h(t)} 
\nonumber \\ & \qquad\qquad
= - \left\langle \pi^h\,[\gamma_\epsilon(\Psi^h)\,\kappa^h]\,\vec\nu^h, \vec U^h
  \right\rangle_{\Gamma^h(t)}
- \left\langle \nabs\,\pi^h\,[\gamma_\epsilon(\Psi^h)], \vec U^h \right\rangle_{\Gamma^h(t)}^h 
\,,
\label{eq:ps1compact}
\end{align}
which, similarly to (\ref{eq:sp40}), is a discrete analogue of (\ref{eq:d45}).
The desired result (\ref{eq:stabHG}) now follows from combining 
(\ref{eq:ps1compact}) with (\ref{eq:lemHG}).
\end{proof}

We remark that it is possible to prove that the vertices of the solution 
$\Gamma^h(t)$ to \mbox{(\ref{eq:sdHGa0}--f)} are
well distributed. As this follows already from the equations 
{\rm (\ref{eq:sdHGd})}, we
refer to our earlier work in \cite{triplej,gflows3d} for further details. In
particular, we observe that in the case $d=2$, i.e.\ for the planar two-phase
problem, an equidistribution property for the vertices of $\Gamma^h(t)$ can be
shown. These good mesh properties mean that for fully discrete schemes based on
(\ref{eq:sdHGa0}--f) no remeshings are required in practice for either $d=2$ or
$d=3$, and this is the main advantage of the scheme (\ref{eq:sdHGa0}--f) over
(\ref{eq:sdGDa0}--f). 
Another advantage is that the volume of the two phases is preserved for 
the approximation (\ref{eq:sdHGa0}--f), recall (\ref{eq:cons}), while it does 
not appear possible to prove a similar result for (\ref{eq:sdGDa0}--f). 
A minor disadvantage is the fact that it does not appear possible to
derive a maximum principle for the discrete surfactant concentration $\Psi^h$
similarly to (\ref{eq:sddmp}). However, the
following remark demonstrates that also for the scheme (\ref{eq:sdHGa0}--f) 
the negative part of $\Psi^h$ can be controlled.
Moreover, in practice we observe that for a fully discrete variant of
(\ref{eq:sdHGa0}--f) the fully discrete analogues of $\Psi^h(\cdot,t)$ 
remain positive for positive initial data.

\begin{rem} \label{rem:Psi-}
The convex nature of $F$, together with the fact that $F'$ is 
singular at the origin, allows us to derive upper bounds on the negative part
of $\Psi^h$ for the two cases {\rm (\ref{eq:gamma1},b)}. 
On recalling {\rm (\ref{eq:Freg})} and {\rm (\ref{eq:F})}, it holds that 
\begin{equation*}
F_\epsilon(r) 
= \gamma(\epsilon) + F'(\epsilon)\,r + \tfrac12\,F''(\epsilon)\,(r-\epsilon)^2
\geq \tfrac12\,F''(\epsilon)\,r^2 \geq 
\tfrac12\,\epsilon^{-1}\,\overline\gamma\,\beta\,r^2
\qquad \forall\ r \leq 0\,,
\end{equation*}
provided that $\epsilon$ is sufficiently small. 
Hence the bound {\rm (\ref{eq:stabHG})}, via a Korn's inequality, 
and on assuming that 
$\left\langle \rho_\Gamma^h\,\vec U^h, \vec U^h \right\rangle_{\Gamma^h(t)}^h
\geq -C_0$
for some positive constant $C_0$ that is independent of $\epsilon$,
implies that
\begin{equation*} 
\left\langle [\Psi^h]_-^2, 1 \right\rangle_{\Gamma^h(t)}^h \leq C\,\epsilon
\qquad\forall\ t \in [0,T]\,,
\end{equation*}
for some positive constant $C$, and for $\epsilon$ sufficiently small.
\end{rem}

\begin{rem} \label{rem:osch}
In order to be able to add numerical diffusion to our fully discrete schemes,
we also consider a variant of {\rm (\ref{eq:sdHGa0}--f)}, where we add
\begin{equation*} 
 \vartheta(h_\Gamma(t))\left\langle
\left| \mat {\mathcal{P}}_{\Gamma^h}
\left(\vec{\mathcal{V}}^h - \vec U^h \right)\right| 
\nabs\,\rho_\Gamma^h, \nabs\, \chi_k^h \right \rangle_{\Gamma^h(t)}^h
\end{equation*}
to the left hand side of {\rm (\ref{eq:sdHGa0})}. 
To maintain stability, we accordingly add the term
$-\tfrac12\,\vartheta(h_\Gamma(t)) \left\langle
\left| \mat{\mathcal{P}}_{\Gamma^h}
\left(\vec{\mathcal{V}}^h - \vec U^h \right)\right|
 \nabs\,\rho_\Gamma^h, \nabs\, \pi^h\,[\vec U^h\,.\,\vec\xi] 
\right \rangle_{\Gamma^h(t)}^h$
to the right hand side of {\rm (\ref{eq:sdHGa})}. Here $\vartheta(s) \geq 0$ is a 
discrete diffusion coefficient with $\vartheta(s) \to 0$ as $s \to 0$,
and $h_\Gamma(t) := \max_{j=1,\ldots,J_\Gamma} \diam{\sigma^h_j(t)}$.
Then it is easy to show that all the results in {\rm
Theorems~\ref{lem:stabHG}}, {\rm \ref{thm:stabHGconst}} and 
{\rm \ref{thm:stabHG}} still remain true. For example, in 
{\rm (\ref{eq:rhonew1})} we note that, on recalling {\rm (\ref{eq:LG})}, the 
bound {\rm (\ref{eq:rhonew2})} still holds.
\end{rem}

\begin{rem}
We recall that the stability proofs in {\rm Theorems~\ref{thm:stabGD}} and
{\rm \ref{thm:stabHG}} are restricted to the case $d=2$. 
However, it is possible to
prove stability for $d = 2$ and $d=3$ for a variant of 
{\rm (\ref{eq:sdGDa0}--f)}, which, on recalling {\rm (\ref{eq:newGD})}, 
is given by
\begin{align}
&
\tfrac12 \left[ \ddt \left( \rho^h\,\vec U^h , \vec \xi \right)
+ \left( \rho^h\,\vec U^h_t , \vec \xi \right)
- (\rho^h\,\vec U^h, \vec \xi_t) \right]
+ 2\left(\mu^h\,\mat D(\vec U^h), \mat D(\vec \xi) \right)
 \nonumber \\ & \quad
+ \tfrac12\left(\rho^h, 
 [(\vec I^h_2\,\vec U^h\,.\,\nabla)\,\vec U^h]\,.\,\vec \xi
- [(\vec I^h_2\,\vec U^h\,.\,\nabla)\,\vec \xi]\,.\,\vec U^h \right)
- \left(P^h, \nabla\,.\,\vec \xi\right)
+ \ddt\left\langle \rho_\Gamma^h\,\vec U^h, \vec\xi
\right\rangle_{\Gamma^h(t)}^h
\nonumber \\ & \quad
+ 2 \left\langle\mu_\Gamma(\Psi^h)\, \mat D_s^h (\vec\pi^h\,\vec U^h) , 
\mat D_s^h (\vec\pi^h\, \vec \xi ) \right\rangle_{\Gamma^h(t)}^h
\nonumber \\ & \quad
+  \left\langle\lambda_\Gamma(\Psi^h) \,
 \nabs\,.\, (\vec\pi^h\,\vec U^h) , \nabs\,.\, 
(\vec\pi^h\,\vec \xi) \right\rangle_{\Gamma^h(t)}^h
\nonumber \\ & \quad
+ \left\langle \gamma(\Psi^h), 
 \nabs\,.\,\vec\pi^h\,\vec\xi\right\rangle_{\Gamma^h(t)}^h
= \left(\rho^h\,\vec f^h_1 + \vec f^h_2, \vec \xi\right)
+ \left\langle \rho_\Gamma^h \,\vec U^h, \matpartxh\,(\vec\pi^h\,\vec \xi) 
\right\rangle_{\Gamma^h(t)}^h
\nonumber \\ & \hspace{10cm}
\qquad \forall\ \vec\xi \in H^1(0,T;\uspace^h) \,, \label{eq:sdGDa2}
\end{align}
together with {\rm (\ref{eq:sdGDa0},c,d,f)}. Here we observe that in this new
discretization it is no longer necessary to compute the discrete curvature
vector $\vec\kappa^h$. It is then not difficult to prove stability for this
scheme for $d=2$ and $d=3$, as {\rm (\ref{eq:JWB})} is now avoided.
See \citet[Theorem~2.7]{tpfs} for an analogous proof. 
\end{rem}

\setcounter{equation}{0}
\section{Fully discrete finite element approximation} \label{sec:3}

In this section we consider fully discrete variants of the schemes
(\ref{eq:sdGDa0}--f) and (\ref{eq:sdHGa0}--f) from \S\ref{sec:2}. 
Here we will
choose the time discretization such that existence and uniqueness of the
discrete solutions can be guaranteed, and such that we inherit as much of the
structure of the stable schemes in \cite{spurious,fluidfbp} as possible, see
below for details.

We consider the partitioning $t_m = m\,\tau$, $m = 0,\ldots, M$, 
of $[0,T]$ into uniform time steps $\tau = T / M$.
The time discrete spatial discretizations then directly follow from the finite
element spaces introduced in \S\ref{sec:2}, where in order to allow for
adaptivity in space we consider bulk finite element spaces that change in time.

For all $m\ge 0$, let $\mathcal{T}^m$ 
be a regular partitioning of $\Omega$ into disjoint open simplices
$\sigmaO^m_j$, $j = 1, \ldots, J_\Omega$. 
We set $h^m:= \max_{j=1 \rightarrow J^m_\Omega}\mbox{diam}( \sigmaO^m_j)$.
Associated with ${\cal T}^m$ are the finite element spaces
$S^m_k$ for $k\geq 0$.
We introduce also $\vec I^m_k:[C(\overline{\Omega})]^d\to [S^m_k]^d$, 
$k\geq 1$, the standard interpolation operators, and the standard projection
operator $I^m_0:L^1(\Omega)\to S^m_0$.
For the approximation to the velocity and pressure on ${\cal T}^m$ 
we will use the finite element spaces
$\uspace^m\subset\uspace$ and $\pspace^m\subset\pspace$, which are the direct
time discrete analogues of $\uspace^h$ and $\pspace^h(t_m)$,
as well as $\widehat\pspace^m := \pspace^m \cap \widehat\pspace$.
We recall that $(\uspace^m, \pspace^m)$ are said to satisfy 
the LBB inf-sup condition if
there exists a constant $C_0 \in \R_{>0}$ independent of $h^m$ such that
\begin{equation} \label{eq:LBB}
\inf_{\varphi \in \widehat\pspace^m} \sup_{\vec \xi \in \uspace^m}
\frac{( \varphi, \nabla \,.\,\vec \xi)}
{\|\varphi\|_0\,\|\vec \xi\|_1} \geq C_0\,.
\end{equation}

Moreover, 
the parametric finite element spaces are given by
\begin{equation} \label{eq:Vh}
\Vh := \{\vec\chi \in [C(\Gamma^m)]^d:\vec\chi\!\mid_{\sigma^m_j}
\mbox{ is linear}\ \forall\ j=1\to J_\Gamma\} 
=: [\Wh]^d \subset [H^1(\Gamma^m)]^d\,,
\end{equation}
for $m=0 \to M-1$. Here
$\Gamma^m=\bigcup_{j=1}^{J_\Gamma} 
\overline{\sigma^m_j}$,
where $\{\sigma^m_j\}_{j=1}^{J_\Gamma}$ is a family of mutually disjoint open 
$(d-1)$-simplices 
with vertices $\{\vec{q}^m_k\}_{k=1}^{K_\Gamma}$. 
We also introduce 
$\pi^m: C(\Gamma^m)\to \Wh$, the standard interpolation operator
at the nodes $\{\vec{q}_k^m\}_{k=1}^{K_\Gamma}$,
and similarly $\vec\pi^m: [C(\Gamma^m)]^d\to \Vh$.
Throughout this paper, we will parameterize the new closed surface 
$\Gamma^{m+1}$ over $\Gamma^m$, with the help of a parameterization
$\vec X^{m+1} \in \Vh$, i.e.\ $\Gamma^{m+1} = \vec X^{m+1}(\Gamma^m)$.

We also introduce the $L^2$--inner 
product $\langle\cdot,\cdot\rangle_{\Gamma^m}$ over
the current polyhedral surface $\Gamma^m$, as well as the 
the mass lumped inner product
$\langle\cdot,\cdot\rangle_{\Gamma^m}^h$.
Similarly to (\ref{eq:Psh},b), we introduce
\begin{subequations}
\begin{equation} \label{eq:Psm}
\mat{\mathcal{P}}_{\Gamma^m} = \mat\Id - \vec \nu^m \otimes \vec \nu^m
\quad\text{on}\ \Gamma^m\,,
\end{equation}
and
\begin{equation} \label{eq:Dsm}
\mat D_s^m(\vec \eta) = \tfrac12\,\mat{\mathcal{P}}_{\Gamma^m}
\,(\nabs\,\vec\eta + (\nabs\,\vec \eta)^T)\,\mat{\mathcal{P}}_{\Gamma^m}
\quad\text{on}\ \Gamma^m\,,
\end{equation}
\end{subequations}
where here $\nabs = \mat{\mathcal{P}}_{\Gamma^m} \,\nabla$ 
denotes the surface gradient on $\Gamma^m$.
In addition, and similarly to (\ref{eq:hatDsh}), we define
\begin{equation} \label{eq:hatDsm}
\mat {\hat D}_s^m(\vec \eta) 
= \mat D_s^m(\vec \eta) - \tfrac1{d-1}\left( \nabs\,.\,\vec \eta \right)
\mat{\mathcal{P}}_{\Gamma^m}
\quad\text{on}\ \Gamma^m\,.
\end{equation}
Then it is straightforward to show that
\begin{align}
& 2\left\langle\mu_\Gamma(\chi)\, \mat D_s^m(\vec \eta), \mat D_s^m(\vec \eta) 
\right\rangle_{\Gamma^m}^h 
+  \left\langle\lambda_\Gamma(\chi) \,
 \nabs\,.\, \vec \eta , \nabs\,.\, \vec \eta \right\rangle_{\Gamma^m}^h 
\nonumber \\ & \hspace{1cm}
= 2\left\langle\mu_\Gamma(\chi)\, \mat {\hat D}_s^m(\vec \eta), 
\mat {\hat D}_s^m(\vec \eta) \right\rangle_{\Gamma^m}^h 
+  \left\langle (\lambda_\Gamma(\chi) + \tfrac2{d-1}\,\mu_\Gamma(\chi))\,
 \nabs\,.\, \vec \eta , \nabs\,.\, \vec \eta \right\rangle_{\Gamma^m}^h 
\nonumber \\ & \hspace{7cm}
\qquad \forall\ \vec\eta \in \Vh\,,\chi \in \Wh
\label{eq:HGm}
\end{align}
holds, which is the fully discrete analogue of (\ref{eq:HGh}). 

Given $\Gamma^m$, we 
let $\Omega^m_+$ denote the exterior of $\Gamma^m$ and let
$\Omega^m_-$ denote the interior of $\Gamma^m$, so that
$\Gamma^m = \partial \Omega^m_- = \overline{\Omega^m_-} \cap 
\overline{\Omega^m_+}$. 
We then partition the elements of the bulk mesh 
$\mathcal{T}^m$ into interior, exterior and interfacial elements as before, and
we introduce  
$\rho^m,\,\mu^m \in S^m_0$, for $m\geq 0$, as 
\begin{equation} \label{eq:rhoma}
\rho^m\!\mid_{o^m} = \begin{cases}
\rho_- & o^m \in \mathcal{T}^m_-\,, \\
\rho_+ & o^m \in \mathcal{T}^m_+\,, \\
\tfrac12\,(\rho_- + \rho_+) & o^m \in \mathcal{T}^m_{\Gamma^m}\,,
\end{cases}
\quad\text{and}\quad
\mu^m\!\mid_{o^m} = \begin{cases}
\mu_- & o^m \in \mathcal{T}^m_-\,, \\
\mu_+ & o^m \in \mathcal{T}^m_+\,, \\
\tfrac12\,(\mu_- + \mu_+) & o^m \in \mathcal{T}^m_{\Gamma^m}\,.
\end{cases}
\end{equation}

We introduce the following pullback and pushforward operators 
for the discrete interfaces $\Gamma^m$ and $\Gamma^{m-1}$. 
Let $\vec\Pi_{m}^{m-1} : [C(\Gamma^{m})]^d \to
\Vhm$ such that
\begin{subequations}
\begin{equation} \label{eq:Pi}
(\vec\Pi_{m}^{m-1}\,\vec z)(\vec q^{m-1}_k) = \vec z(\vec q^{m}_k)\,,
\qquad k = 1,\ldots,K_\Gamma\,,\qquad
\forall\ \vec z \in [C(\Gamma^{m})]^d\,,
\end{equation}
for $m=1,\ldots,M-1$, and set $\vec\Pi_{0}^{-1} := \vec \pi^0$.
Similarly, let $\vec\Pi_{m-1}^m : [C(\Gamma^{m-1})]^d \to \Vh$ such that
\begin{equation} \label{eq:Pib}
(\vec\Pi_{m-1}^m\,\vec z)(\vec q^m_k) = \vec z(\vec q^{m-1}_k)\,,
\qquad k = 1,\ldots,K_\Gamma\,,\qquad
\forall\ \vec z \in [C(\Gamma^{m-1})]^d\,,
\end{equation}
\end{subequations}
for $m=1,\ldots,M-1$, and set $\vec\Pi_{-1}^0 := \vec \pi^0$. Analogously to
(\ref{eq:Pib}) we also introduce $\Pi_{m-1}^m : C(\Gamma^{m-1}) \to \Wh$.

We 
set $\rho^{-1} := \rho^0$,
$\Gamma^{-1} := \Gamma^0$, $\vec X^{-1} := \vec X^0$ and
$\rho^{-1}_\Gamma:=\rho^0_\Gamma$.
Our proposed fully discrete equivalent of (\ref{eq:sdGDa0}--f) is then given as
follows.
Let $\Gamma^0$, an approximation to $\Gamma(0)$, 
and $\vec U^0\in \uspace^0$, 
$\vec\kappa^0 \in \underline{V}(\Gamma^0)$,
$\rho_\Gamma^0 \in W(\Gamma^0)$ and $\Psi^0 \in W(\Gamma^0)$ be given.
For $m=0\to M-1$, find $\vec U^{m+1} \in \uspace^m$, 
$P^{m+1} \in \widehat\pspace^m$, $\vec X^{m+1}\in\Vh$
and $\vec\kappa^{m+1} \in \Vh$ such that 
\begin{subequations}
\begin{align}
&
\tfrac12 \left( \frac{\rho^m\,\vec U^{m+1} - (I^m_0\,\rho^{m-1})
\,\vec I^m_2\,\vec U^m}{\tau}
+(I^m_0\,\rho^{m-1}) \,\frac{\vec U^{m+1}- \vec I^m_2\,\vec{U}^m}{\tau}, \vec \xi 
\right)
 \nonumber \\ & \quad
+ 2\left(\mu^m\,\mat D(\vec U^{m+1}), \mat D(\vec \xi) \right)
+ \tfrac12\left(\rho^m, 
 [(\vec I^m_2\,\vec U^m\,.\,\nabla)\,\vec U^{m+1}]\,.\,\vec \xi
- [(\vec I^m_2\,\vec U^m\,.\,\nabla)\,\vec \xi]\,.\,\vec U^{m+1} \right)
\nonumber \\ & \quad
- \left(P^{m+1}, \nabla\,.\,\vec \xi\right)
+ \frac1{\tau} \left\langle \rho_\Gamma^m\,\vec U^{m+1} , \vec\xi
 \right\rangle_{\Gamma^m}^h
+ 2 \left\langle \mu_\Gamma(\Psi^m)\, \mat D_s^m (\vec\pi^m\,\vec U^{m+1}) , 
\mat D_s^m ( \vec\pi^m\,\vec \xi ) \right\rangle_{\Gamma^m}^h
\nonumber \\ & \quad
+  \left\langle\lambda_\Gamma(\Psi^m) \,
 \nabs\,.\,(\vec\pi^m\,\vec U^{m+1}) , 
\nabs\,.\, (\vec\pi^m\,\vec \xi) \right\rangle_{\Gamma^m}^h
\nonumber \\ & \quad
- \left\langle \gamma(0)\,(\vec\kappa^{m+1} - \vec\Pi_{m-1}^m\,\vec\kappa^m)
+ \gamma(\Psi^m)\,\vec\Pi_{m-1}^m\,\vec\kappa^{m} 
 + \nabs\,[\pi^m\,\gamma(\Psi^m)], \vec \xi \right\rangle_{\Gamma^m}^h 
\nonumber \\ & \hspace{1cm}\quad
= \left(\rho^m\,\vec f^{m+1}_1 + \vec f^{m+1}_2, \vec \xi\right)
+ \frac1{\tau} \left\langle \rho^{m-1}_\Gamma\,\vec I^m_2\,\vec U^m,
 \vec\Pi_m^{m-1}\,\vec \xi\!\mid_{\Gamma^m} \right\rangle_{\Gamma^{m-1}}^h
\qquad \forall\ \vec\xi \in \uspace^m \,, \label{eq:GDa}\\
& \left(\nabla\,.\,\vec U^{m+1}, \varphi\right) = 0 
\qquad \forall\ \varphi \in \widehat\pspace^m\,,
\label{eq:GDb} \\
& \left\langle \frac{\vec X^{m+1} - \vec \id}{\tau} ,
\vec\chi \right\rangle_{\Gamma^m}^h
= \left\langle \vec U^{m+1}, \vec\chi \right\rangle_{\Gamma^m}^h
 \qquad\forall\ \vec\chi \in \Vh\,,
\label{eq:GDc} \\
& \left\langle \vec\kappa^{m+1} , \vec\eta \right\rangle_{\Gamma^m}^h
+ \left\langle \nabs\,\vec X^{m+1}, \nabs\,\vec \eta \right\rangle_{\Gamma^m}
 = 0  \qquad\forall\ \vec\eta \in \Vh\label{eq:GDd} \\
\intertext{and set $\Gamma^{m+1} = \vec X^{m+1}(\Gamma^m)$. 
Then find 
$\rho_\Gamma^{m+1} \in \Whp$ 
and 
$\Psi^{m+1} \in \Whp$ such that}
& 
\left\langle \rho_\Gamma^{m+1}, \chi^{m+1}_k \right\rangle_{\Gamma^{m+1}}^h
= 
\left\langle \rho_{\Gamma}^{m}, \chi^{m}_k \right\rangle_{\Gamma^m}^h 
\qquad\forall\ k \in \{1,\ldots,K_\Gamma\}\,,
\label{eq:GDa0}\\
& \frac1{\tau}
\left\langle \Psi^{m+1}, \chi^{m+1}_k \right\rangle_{\Gamma^{m+1}}^h
+ \Ds\left\langle \nabs\, \Psi^{m+1}, \nabs\, \chi^{m+1}_k
\right\rangle_{\Gamma^{m+1}}
= \frac1{\tau}
\left\langle \Psi^{m}, \chi^{m}_k \right\rangle_{\Gamma^m}^h 
\nonumber \\ & \hspace{10cm}
\quad\forall\ k \in \{1,\ldots,K_\Gamma\}\,.
\label{eq:GDe}
\end{align}
\end{subequations}
Here we have defined $\vec f^{m+1}_i := \vec I^m_2\,\vec
f_i(\cdot,t_{m+1})$, $i=1,2$. 
We observe that (\ref{eq:GDa}--f) is a linear scheme in that
it leads to a linear system of equations for the unknowns  \linebreak
$(\vec U^{m+1}, P^{m+1}, \vec X^{m+1}, \vec\kappa^{m+1}, \rho_{\Gamma}^{m+1},\Psi^{m+1})$ 
at each time level. In particular, the system \mbox{(\ref{eq:GDa}--f)} clearly
decouples into (\ref{eq:GDa}--d) for $(\vec U^{m+1}, P^{m+1}, \vec X^{m+1},
\vec \kappa^{m+1})$, (\ref{eq:GDa0}) for $\rho_\Gamma^{m+1}$ and (\ref{eq:GDe}) for $\Psi^{m+1}$.

We note that the right hand side in (\ref{eq:GDa}) was obtained from
\begin{align} \label{eq:rhsderiv}
 \frac1{\tau} \left\langle \rho^{m-1}_\Gamma\,\vec I^m_2\,\vec U^m,
 \vec\Pi_m^{m-1}\,\vec \xi\!\mid_{\Gamma^m} \right\rangle_{\Gamma^{m-1}}^h
& =
\frac1{\tau} \left\langle \rho^{m-1}_\Gamma\,\vec I^m_2\,\vec U^{m} , \vec\xi
 \right\rangle_{\Gamma^{m-1}}^h \nonumber \\ & \qquad
+ \frac1{\tau} \left\langle \rho^{m-1}_\Gamma\,\vec I^m_2\,\vec U^m,
 \vec\Pi_m^{m-1}\,\vec \xi\!\mid_{\Gamma^m} - \vec\xi 
 \right\rangle_{\Gamma^{m-1}}^h\,,
\end{align}
where we recall from (\ref{eq:mpbf}) and (\ref{eq:p1}) that the last term in
(\ref{eq:rhsderiv}) is a fully discrete approximation of the last term in
(\ref{eq:sdGDa}). 

When the velocity/pressure space pair $(\uspace^m,\widehat\pspace^m)$ does not
satisfy (\ref{eq:LBB}), we need to consider the following reduced version of
(\ref{eq:GDa}--d), where the pressure $P^{m+1}$ is eliminated, 
in order to prove existence of a solution.
Let 
$$\uspace^m_0 := 
\{ \vec U \in \uspace^m : (\nabla\,.\,\vec U, \varphi) = 0 \ \
\forall\ \varphi \in \widehat\pspace^m \} \,.$$ 
Then any solution $(\vec U^{m+1}, P^{m+1}, \vec X^{m+1}, \vec\kappa^{m+1}) \in 
\uspace^m\times\widehat\pspace^m \times [\Vh]^2$ to {\rm (\ref{eq:GDa}--d)}
is such that $(\vec U^{m+1}, \vec X^{m+1}, \vec\kappa^{m+1})\in 
\uspace^m_0 \times [\Vh]^2$ satisfy
(\ref{eq:GDa},c,d) with $\uspace^m$ replaced by $\uspace^m_0$.

In order to prove the existence of a unique solution to (\ref{eq:GDa}--f) we
make the following very mild well-posedness assumption.

\begin{itemize}
\item[$(\mathcal{A})$]
We assume for $m=0,\ldots, M-1$ that $\mathcal{H}^{d-1}(\sigma^m_j) > 0$ 
for all $j=1,\ldots, J_\Gamma$,
and that $\Gamma^m \subset \Omega$.
\end{itemize}

Moreover, and similarly to (\ref{eq:chij}), we note that the assumption
\begin{equation} \label{eq:chijm}
\int_{\sigma^{m+1}_j} \nabs \chi^{m+1}_i \,.\,\nabs \chi^{m+1}_k 
\dH{d-1} \leq 0 
\quad \forall\ i \neq k\,,\qquad j = 1,\ldots,J_\Gamma\,,
\end{equation}
is always satisfied for $d=2$, and for $d=3$ if all the triangles 
$\sigma^{m+1}$ of $\Gamma^{m+1}$ have no obtuse angles.

\begin{thm} \label{thm:GD}
Let the assumption $(\mathcal{A})$ hold and let $\rho^m_\Gamma \geq 0$.
If the LBB condition {\rm (\ref{eq:LBB})} holds, then there exists a unique
solution $(\vec U^{m+1}, P^{m+1}, \vec X^{m+1}, \vec\kappa^{m+1}) 
\in \uspace^m\times\widehat\pspace^m \times [\Vh]^2$ 
to {\rm (\ref{eq:GDa}--d)}. In all other
cases there exists a unique solution 
$(\vec U^{m+1}, \vec X^{m+1}, \vec\kappa^{m+1}) \in 
\uspace^m_0 \times [\Vh]^2$ to the
reduced system {\rm (\ref{eq:GDa},c,d)} with $\uspace^m$ replaced by
$\uspace^m_0$.
In either case, there exists a unique solution 
$(\rho_\Gamma^{m+1}, \Psi^{m+1})\in [\Whp]^2$ to 
{\rm (\ref{eq:GDa0},f)} that satisfies 
\begin{subequations}
\begin{equation} \label{eq:consm}
\left\langle \rho_\Gamma^{m+1}, 1 \right\rangle_{\Gamma^{m+1}} =
\left\langle \rho_\Gamma^{m}, 1 \right\rangle_{\Gamma^{m}}
\quad\text{and}\quad
\left\langle \Psi^{m+1}, 1 \right\rangle_{\Gamma^{m+1}} =
\left\langle \Psi^{m}, 1 \right\rangle_{\Gamma^{m}}
\end{equation}
and
\begin{equation} \label{eq:dmprho}
\rho_\Gamma^{m+1} \geq 0\,. 
\end{equation}
Moreover, if $\Ds=0$ or if the assumption {\rm (\ref{eq:chijm})} holds, then
\begin{equation} \label{eq:dmp}
\Psi^{m+1} \geq 0 \qquad \text{if}\quad \Psi^m \geq 0\,.
\end{equation}
\end{subequations}
\end{thm}
\begin{proof}
As all the systems are linear, existence follows from uniqueness.
In order to establish the latter, we will consider the homogeneous 
system in each case. We begin with:
Find $(\vec U, P, \vec X, \vec\kappa) \in \uspace^m\times\widehat\pspace^m
\times [\Vh]^2$ such that
\begin{subequations}
\begin{align}
&
\tfrac1{2\,\tau} \left( (\rho^m+I^m_0\,\rho^{m-1})\,\vec U, \vec \xi \right)
+ 2\left(\mu^m\,\mat D(\vec U), \mat D(\vec \xi) \right)
- \left(P, \nabla\,.\,\vec \xi\right)
\nonumber \\ & \qquad
+ \tfrac12\left(\rho^m, [(\vec I^m_2\,\vec U^m\,.\,\nabla)\,\vec U]\,.\,\vec \xi
- [(\vec I^m_2\,\vec U^m\,.\,\nabla)\,\vec \xi]\,.\,\vec U \right)
\nonumber \\ & \quad
+ \tfrac1{\tau} \left\langle \rho^m_\Gamma\,\vec U , \vec\xi
 \right\rangle_{\Gamma^m}^h
+ 2 \left\langle\mu_\Gamma(\Psi^m)\, \mat D_s^m (\vec\pi^m\,\vec U) , 
\mat D_s^m ( \vec\pi^m\,\vec \xi ) \right\rangle_{\Gamma^m}^h
\nonumber \\ & \quad
+  \left\langle\lambda_\Gamma(\Psi^m) \,
 \nabs\,.\,(\vec\pi^m\,\vec U) , 
\nabs\,.\, (\vec\pi^m\,\vec \xi) \right\rangle_{\Gamma^m}^h
- \gamma(0)\left\langle \vec\kappa, \vec \xi \right\rangle_{\Gamma^m}^h 
= 0 
 \qquad \forall\ \vec\xi \in \uspace^m \,, \label{eq:proofa}\\
& \left(\nabla\,.\,\vec U, \varphi\right)  = 0 
\qquad \forall\ \varphi \in \widehat\pspace^m\,,
\label{eq:proofb} \\
& \tfrac1{\tau} \left\langle \vec X,\vec\chi \right\rangle_{\Gamma^m}^h
= \left\langle \vec U, \vec\chi \right\rangle_{\Gamma^m}^h
 \qquad\forall\ \vec\chi \in \Vh\,,
\label{eq:proofc} \\
& \left\langle \vec\kappa , \vec\eta \right\rangle_{\Gamma^m}^h
+ \left\langle \nabs\,\vec X, \nabs\,\vec \eta \right\rangle_{\Gamma^m}
 = 0  \qquad\forall\ \vec\eta \in \Vh \,.\label{eq:proofd} 
\end{align}
\end{subequations}
Choosing $\vec\xi=\vec U$ in (\ref{eq:proofa}),
$\varphi =  P$ in (\ref{eq:proofb}),
$\vec\chi = \gamma(0)\,\vec\kappa$ in (\ref{eq:proofc}) 
and $\vec\eta=\gamma(0)\,\vec X$ in (\ref{eq:proofd})
yields, on recalling (\ref{eq:HGm}), that
\begin{align}
& \tfrac12\left((\rho^m + I^m_0\,\rho^{m-1})\,\vec U, \vec U \right) + 
2\,\tau\left(\mu^m\,\mat D(\vec U), \mat D(\vec U) \right)
+ \left\langle \rho^m_\Gamma\,\vec U , \vec U \right\rangle_{\Gamma^m}^h
\nonumber \\ & \quad
+ 2\,\tau \left\langle\mu_\Gamma(\Psi^m)\, 
\mat {\hat D}_s^m (\vec\pi^m\,\vec U) , 
\mat {\hat D}_s^m ( \vec\pi^m\,\vec U) \right\rangle_{\Gamma^m}^h
\nonumber \\ & \quad
+ \tau 
 \left\langle (\lambda_\Gamma(\Psi^m) + \tfrac2{d-1}\,\mu_\Gamma(\Psi^m))\,
  \nabs\,.\,(\vec\pi^m\,\vec U) , 
\nabs\,.\, (\vec\pi^m\,\vec U) \right\rangle_{\Gamma^m}^h
+ \gamma(0)\left\langle \nabs\,\vec X, \nabs\,\vec X \right\rangle_{\Gamma^m} 
=0\,. \label{eq:proof2GD}
\end{align}
It immediately follows from (\ref{eq:proof2GD}), on recalling $\rho_\pm > 0$
and (\ref{eq:mulambda}),
that $\vec U = \vec 0 \in \uspace^m$.
Moreover, (\ref{eq:proofa}) with $\vec U = \vec 0$ implies,
together with (\ref{eq:LBB}), that $P = 0 \in \widehat\pspace^m$. This shows
existence and uniqueness of 
$(\vec U^{m+1}, P^{m+1}) \in \uspace^m\times\widehat\pspace^m$.
The proof for the reduced equation is very similar. The homogeneous system to
consider is (\ref{eq:proofa}) with $\uspace^m$ replaced by
$\uspace^m_0$, where we note that the latter is a linear subspace of
$\uspace^m$. As before, (\ref{eq:proof2GD}) 
yields that $\vec U = \vec 0 \in \uspace^m_0$, and so the existence of a unique
solution $\vec U^{m+1} \in \uspace^m_0$ to the reduced equation.
In addition, it follows from (\ref{eq:proof2GD}) that 
$\vec X = \vec X_c \in \R^d$. Hence (\ref{eq:proofd}) yields that
$\vec\kappa=\vec0$, while (\ref{eq:proofc}) with  $\vec U = \vec 0$ implies
that $\vec X = \vec 0$.

The two equations (\ref{eq:GDa0},f) are clearly symmetric, positive definite 
linear systems with unique solutions $\rho_\Gamma^{m+1} \in \Whp$ and
$\Psi^{m+1} \in \Whp$, respectively. The desired results in
(\ref{eq:consm}) follow on summing (\ref{eq:GDa0}) and (\ref{eq:GDe})
for $k = 1, \ldots, K_\Gamma$, respectively. 
In order to prove (\ref{eq:dmprho}) we note that $\rho_\Gamma^m\geq0$ in 
(\ref{eq:GDa0}) implies that
\begin{equation} \label{eq:dmprho1}
\left\langle [\rho_\Gamma^{m+1}]_-, [\rho_\Gamma^{m+1}]_- 
\right\rangle_{\Gamma^{m+1}}^h 
= \left\langle \rho_\Gamma^{m+1}, [\rho_\Gamma^{m+1}]_- 
\right\rangle_{\Gamma^{m+1}}^h 
\leq0\,,
\end{equation}
i.e.\ $\rho_\Gamma^{m+1} \geq 0$.
Similarly, on assuming $\Psi^m\geq0$ we observe from 
(\ref{eq:GDe}) that this implies that
\begin{equation} \label{eq:dmp1}
\left\langle \Psi^{m+1}, [\Psi^{m+1}]_- \right\rangle_{\Gamma^{m+1}}^h
+ \tau\,\Ds\left\langle \nabs\, \Psi^{m+1}, \nabs\, \pi^{m+1}\,[\Psi^{m+1}]_-
\right\rangle_{\Gamma^{m+1}} \leq 0\,.
\end{equation}
Similarly to (\ref{eq:LG}) it follows that under our assumptions the
second term in (\ref{eq:dmp1}) is nonnegative,
which yields that 
$\Psi^{m+1} \geq 0$, similarly to (\ref{eq:dmprho1}). 
\end{proof}

Let 
\begin{equation*} 
\mathcal{E}(\xi,\vec V,\mathcal{M}) := 
\tfrac12\,(\xi\,\vec V, \vec V) + \gamma(0)\, \mathcal{H}^{d-1}(\mathcal{M})\,,
\end{equation*}
for $\xi \in L^\infty(\Omega)$, $\vec V \in \uspace$ and $\mathcal{M} \subset
\R^d$ being a $(d-1)$-dimensional manifold.

\begin{thm} \label{thm:fdstabGD}
Let $\gamma$ be defined as in {\rm (\ref{eq:Fconst})}, let 
{\rm (\ref{eq:muconst})} hold, let $\rho_\Gamma^m=\rho_\Gamma^{m-1}=0$ 
and let 
$(\vec U^{m+1}, P^{m+1}, \vec X^{m+1}, \vec\kappa^{m+1},\rho_\Gamma^{m+1})$
be a solution to {\rm (\ref{eq:GDa}--e)}. 
Then $\rho_\Gamma^{m+1} = 0$ and
\begin{align}
& \mathcal{E}(\rho^m,\vec U^{m+1}, \Gamma^{m+1})
+ \tfrac12\left((I^m_0\rho^{m-1})\,(\vec U^{m+1} - \vec I^m_2\,\vec U^m), 
\vec U^{m+1} - \vec I^m_2\,\vec U^m \right) 
\nonumber \\ & \hspace{1cm}
+ 2\,\tau\left(\mu^m\,\mat D(U^{m+1}), \mat D(U^{m+1}) \right)
\nonumber \\ & \hspace{1cm}
+ 2\,\tau\,\overline\mu_\Gamma
\left\langle \mat {\hat D}_s^{m+1} (\vec\pi^{m+1}\,\vec U^{m+1}) , 
\mat {\hat D}_s^{m+1} (\vec\pi^{m+1}\, \vec U^{m+1} ) 
\right\rangle_{\Gamma^{m+1}}
\nonumber \\ & \hspace{1cm}
+ \tau\,(\overline\lambda_\Gamma + \tfrac2{d-1}\,\overline\mu_\Gamma)
 \left\langle \nabs\,.\, (\vec\pi^{m+1}\,\vec U^{m+1}) , \nabs\,.\, 
(\vec\pi^{m+1}\,\vec U^{m+1}) \right\rangle_{\Gamma^{m+1}}
\nonumber \\ & \hspace{3cm}
\leq \mathcal{E}(I^m_0\,\rho^{m-1},\vec I^m_2\,\vec U^m,\Gamma^{m}) 
+ \tau\left( \rho^m\,\vec f^{m+1}_1 + \vec f^{m+1}_2, 
\vec U^{m+1} \right).
\label{eq:fdstabGD}
\end{align}
\end{thm}
\begin{proof}
It follows immediately from (\ref{eq:GDa0}) that $\rho_\Gamma^{m+1} = 0$.
Choosing $\vec\xi = \vec U^{m+1}$ in (\ref{eq:GDa}), 
$\varphi = P^{m+1}$ in (\ref{eq:GDb}), 
$\vec\chi = \overline\gamma\,\vec\kappa^{m+1}$ in (\ref{eq:GDc}) and
$\vec\eta=\overline\gamma\,(\vec X^{m+1}-\vec \id\!\mid_{\Gamma^m})$ 
in (\ref{eq:GDd}) yields that
\begin{align*}
& \tfrac12\left(\rho^m\,\vec U^{m+1}, \vec U^{m+1}\right)
+ \tfrac12\left((I^m_0\,\rho^{m-1})\,(\vec U^{m+1} - \vec I^m_2\,\vec U^m), 
\vec U^{m+1} - \vec I^m_2\,\vec U^m \right) 
\nonumber \\ & \quad
+ 2\,\tau\left(\mu^m\,\mat D(U^{m+1}), \mat D(U^{m+1}) \right)
+ 2\,\tau\,\overline\mu_\Gamma 
\left\langle \mat D_s^{m+1} (\vec\pi^{m+1}\,\vec U^{m+1}) , 
\mat D_s^{m+1} (\vec\pi^{m+1}\, \vec U^{m+1} ) \right\rangle_{\Gamma^{m+1}}
\nonumber \\ & \quad
+ \tau\,\overline\lambda_\Gamma 
 \left\langle \nabs\,.\, (\vec\pi^{m+1}\,\vec U^{m+1}) , \nabs\,.\, 
(\vec\pi^{m+1}\,\vec U^{m+1}) \right\rangle_{\Gamma^{m+1}}
\nonumber \\ & \quad
+ \overline\gamma\,
 \left\langle \nabs\,\vec X^{m+1}, \nabs\,(\vec X^{m+1} - \vec\id) 
\right\rangle_{\Gamma^m} 
\nonumber \\ & \hspace{1cm}
= \tfrac12\left((I^m_0\,\rho^{m-1})\,\vec I^m_2\,\vec U^{m}, \vec I^m_2\,\vec U^{m}\right)
+ \tau\left( \rho^m\,\vec f^{m+1}_1 + \vec f^{m+1}_2,
  \vec U^{m+1} \right).
\end{align*}
and hence (\ref{eq:fdstabGD}), on recalling (\ref{eq:HGm}), 
follows immediately, where we have used the result that
\begin{equation*}
\left\langle \nabs\,\vec X^{m+1}, \nabs\,(\vec X^{m+1} - \vec\id) 
\right\rangle_{\Gamma^m}
\geq \mathcal{H}^{d-1}(\Gamma^{m+1}) - \mathcal{H}^{d-1}(\Gamma^{m})
\end{equation*}
see e.g.\ \cite{triplej} and \cite{gflows3d} 
for the proofs for $d=2$ and $d=3$, respectively.
\end{proof}

In order to define a fully discrete equivalent of (\ref{eq:sdHGa0}--f), 
we introduce the matrix functions 
$\mat\Xi^m : \Wh \to [L^\infty(\Gamma^m)]^{d \times d}$ 
defined such that for all $z^h \in \Wh$ it holds that
\begin{equation} \label{eq:Xim}
\mat\Xi^m(z^h) \! \mid_{\sigma^m_j} \in \R^{d\times d} \quad\text{and}\quad
\mat\Xi^m(z^h) \, \nabs\,z^h
= \tfrac12\,\nabs\,\pi^m \,[|z^h|^2] \ \text{ on } \sigma^m_j\,,
j = 1,\ldots,J_\Gamma\,,
\end{equation}
which can be constructed in a fashion analogous to (\ref{eq:Xiho},b). 
We let $\mat\Xi^{-1} := \mat\Xi^0$, 
as well as $\vec U^{-1} := \vec U^0$ and $\pi^{-1} := \pi^0$.

Let $\Gamma^0$, an approximation to $\Gamma(0)$, 
and $\vec U^0\in \uspace^0$, 
$\kappa^0 \in W(\Gamma^0)$,
$\rho_\Gamma^0 \in W(\Gamma^0)$ and $\Psi^0 \in W(\Gamma^0)$ be given.
For $m=0\to M-1$, find $\vec U^{m+1} \in \uspace^m$, 
$P^{m+1} \in \widehat\pspace^m$, 
$\vec X^{m+1}\in\Vh$ and $\kappa^{m+1} \in \Wh$ such that
\begin{subequations}
\begin{align}
&
\tfrac12 \left( \frac{\rho^m\,\vec U^{m+1} - (I^m_0\,\rho^{m-1})
\,\vec I^m_2\,\vec U^m}{\tau}
+(I^m_0\,\rho^{m-1}) \,\frac{\vec U^{m+1}- \vec I^m_2\,\vec{U}^m}{\tau}, \vec \xi 
\right)
 \nonumber \\ & \
+ 2\left(\mu^m\,\mat D(\vec U^{m+1}), \mat D(\vec \xi) \right)
+ \tfrac12\left(\rho^m, 
 [(\vec I^m_2\,\vec U^m\,.\,\nabla)\,\vec U^{m+1}]\,.\,\vec \xi
- [(\vec I^m_2\,\vec U^m\,.\,\nabla)\,\vec \xi]\,.\,\vec U^{m+1} \right)
\nonumber \\ & \
- \left(P^{m+1}, \nabla\,.\,\vec \xi\right)
+ \frac1{\tau} \left\langle [\rho^m_\Gamma]_+\,\vec U^{m+1} 
+ [\rho^m_\Gamma]_-\,\vec I^m_2\,\vec U^{m} , \vec\xi
 \right\rangle_{\Gamma^m}^h
\nonumber \\ & \
+ 2 \left\langle \mu_\Gamma(\Psi^m)\, \mat D_s^m (\vec\pi^m\,\vec U^{m+1}) , 
\mat D_s^m ( \vec\pi^m\,\vec \xi ) \right\rangle_{\Gamma^m}^h
\nonumber \\ & \
+ \left\langle \lambda_\Gamma(\Psi^m) \,
\nabs\,.\,(\vec\pi^m\,\vec U^{m+1}) , 
\nabs\,.\, (\vec\pi^m\,\vec \xi) \right\rangle_{\Gamma^m}^h
\nonumber \\ & \
- \left\langle \gamma(0)\,(\kappa^{m+1} - \Pi_{m-1}^m\,\kappa^m)\,\vec\nu^m
+ \pi^m\,[\gamma_\epsilon(\Psi^m)\,\Pi_{m-1}^m\,\kappa^{m}]\,\vec\nu^m
,\vec \xi \right\rangle_{\Gamma^m}
\nonumber \\ & \
- \left\langle \nabs\,[\pi^m\,\gamma_\epsilon(\Psi^m)],
\vec \xi \right\rangle_{\Gamma^m}^h 
\nonumber \\ & \
= \left(\rho^m\,\vec f^{m+1}_1 + \vec f^{m+1}_2, \vec \xi\right)
+ \frac1{\tau} \left\langle [\rho^{m-1}_\Gamma]_+\,\vec I^m_2\,\vec U^m +
 [\rho^{m-1}_\Gamma]_-\,\vec I^m_2\,\vec U^{m-1} , 
\vec\Pi_m^{m-1}\,\vec \xi\!\mid_{\Gamma^m} \right\rangle_{\Gamma^{m-1}}^h
\nonumber \\ & \quad
- \sum_{i=1}^d \left\langle \rho_{\Gamma,\star}^{m-1}\left(
\frac{\vec X^m - \vec \id}{\tau} - \vec I^m_2\,\vec U^m \right) ,
\mat\Xi^{m-1}(\pi^{m-1}\,I^m_2\,U^m_i)\, \nabs\,(\pi^{m-1}\, \xi_i) 
\right\rangle_{\Gamma^{m-1}}^h
\nonumber \\ & \hspace{12cm} \forall\ \vec\xi \in \uspace^m \,, \label{eq:HGa}\\
& \left(\nabla\,.\,\vec U^{m+1}, \varphi\right)  = 0 
\quad \forall\ \varphi \in \widehat\pspace^m\,,
\label{eq:HGb} \\
&  \left\langle \frac{\vec X^{m+1} - \vec \id}{\tau} ,
\chi\,\vec\nu^m \right\rangle_{\Gamma^m}^h
= \left\langle \vec U^{m+1}, 
\chi\,\vec\nu^m \right\rangle_{\Gamma^m} 
 \quad\forall\ \chi \in \Wh\,,
\label{eq:HGc} \\
& \left\langle \kappa^{m+1}\,\vec\nu^m, \vec\eta \right\rangle_{\Gamma^m}^h
+ \left\langle \nabs\,\vec X^{m+1}, \nabs\,\vec \eta \right\rangle_{\Gamma^m}
 = 0  \quad\forall\ \vec\eta \in \Vh\,, \label{eq:HGd} 
\intertext{and set $\Gamma^{m+1} = \vec X^{m+1}(\Gamma^m)$. Here we have
recalled the definition (\ref{eq:pm}). Then find 
$\rho_\Gamma^{m+1} \in \Whp$ 
and 
$\Psi^{m+1} \in \Whp$ such that}
& \frac1{\tau}
\left\langle \rho_\Gamma^{m+1}, \chi^{m+1}_k \right\rangle_{\Gamma^{m+1}}^h
= \frac1{\tau}
\left\langle \rho_{\Gamma}^{m}, \chi^{m}_k \right\rangle_{\Gamma^m}^h 
 - \left\langle \rho_{\Gamma,\star}^m , 
\left(\frac{\vec X^{m+1} - \vec \id}{\tau} - \vec U^{m+1} \right) .\,
\nabs\,\chi^{m}_k \right\rangle_{\Gamma^m}^h
\nonumber \\ & \hspace{10cm}
\qquad\forall\ k \in \{1,\ldots,K_\Gamma\}\,,
\label{eq:HGa0}\\
& \frac1{\tau}
\left\langle \Psi^{m+1}, \chi^{m+1}_k \right\rangle_{\Gamma^{m+1}}^h 
+ \Ds\left\langle \nabs\, \Psi^{m+1}, \nabs\, \chi^{m+1}_k
\right\rangle_{\Gamma^{m+1}}
\nonumber \\ & \quad
= \frac1{\tau}
 \left\langle \Psi^{m}, \chi^{m}_k \right\rangle_{\Gamma^m}^h 
 - \left\langle \Psi^m_{\star,\epsilon}, \left( 
\frac{\vec X^{m+1} - \vec \id}{\tau} - \vec U^{m+1} \right) .\,
\nabs\,\chi^{m}_k \right\rangle_{\Gamma^m}^h
\quad\forall\ k \in \{1,\ldots,K_\Gamma\}\,,
\label{eq:HGe}
\end{align}
\end{subequations}
where $\Psi^m_{\star,\epsilon} = \Psi^m$ for $d=3$ and, 
on recalling (\ref{eq:Fdd}),
\begin{equation*} 
\Psi^m_{\star,\epsilon} = \begin{cases}
- \frac{\gamma_\epsilon(\Psi^m_k) - \gamma_\epsilon(\Psi^m_{k-1})}
{F'_\epsilon(\Psi^m_k) - F'_\epsilon(\Psi^m_{k-1})} &
F'_\epsilon(\Psi^m_{k-1}) \not= F'_\epsilon(\Psi^m_k)\,, \\
\frac12\,(\Psi^m_{k-1} + \Psi^m_k) 
& F'_\epsilon(\Psi^m_{k-1}) = F'_\epsilon(\Psi^m_k)\,,
\end{cases}
\quad\text{on}\quad [\vec q^m_{k-1}, \vec q^m_{k}]
\quad\forall\ k \in \{1,\ldots,K_\Gamma\}
\end{equation*}
for $d=2$, where $\Psi^m = \sum_{k=1}^{K_\Gamma} \Psi^m_k\,\chi^m_k$.
Moreover, on recalling (\ref{eq:rhohstar}), we set
\begin{equation} \label{eq:rhomstar}
\rho^m_{\Gamma,\star} = 
\begin{cases}
\frac1{\mathcal{H}^{d-1}(\sigma^m_j)}\,
\int_{\sigma^m_j} \rho^m_\Gamma  \dH{d-1} & \rho^m_\Gamma \geq 0\
\text{on}\ \overline{\sigma^m_j}\,,\\
0 & \min_{\overline{\sigma^m_j}} \rho^h_\Gamma < 0\,,
\end{cases}
\quad\text{on}\quad \sigma^m_j
\quad\forall\ j \in \{1,\ldots,J_\Gamma\} \,.
\end{equation}
We observe that (\ref{eq:HGa}--f) is a linear scheme in that
it leads to a linear system of equations for the unknowns 
$(\vec U^{m+1}, P^{m+1}, \vec X^{m+1}, \kappa^{m+1}, \rho^{m+1}_\Gamma,$ $
\Psi^{m+1})$ 
at each time level. In particular, the system (\ref{eq:HGa}--f) clearly
decouples into (\ref{eq:HGa}--d) for $(\vec U^{m+1}, P^{m+1},
\vec X^{m+1},\kappa^{m+1})$, (\ref{eq:HGa0}) for
$\rho_\Gamma^{m+1}$ and (\ref{eq:HGe}) for $\Psi^{m+1}$.

In order to prove the existence of a unique solution to (\ref{eq:HGa}--f) we
need to make the following very mild additional assumption.

\begin{itemize}
\item[$(\mathcal{B})$]
For $k= 1 , \ldots, K_\Gamma$, let
$\Theta_k^m:= \{\sigma^m_j : \vec{q}^m_k \in \overline{\sigma^m_j}\}$
and set
\begin{equation*}
\Lambda_k^m := \bigcup_{\sigma^m_j \in \Theta_k^m} \overline{\sigma^m_j}
 \qquad \mbox{and} \qquad
\vec\omega^m_k := \frac{1}{\mathcal{H}^{d-1}(\Lambda^m_k)}
\sum_{\sigma^m_j\in \Theta_k^m} \mathcal{H}^{d-1}(\sigma^m_j)
\;\vec{\nu}^m_j\,. 
\end{equation*}
Then we further assume that 
$\dim \spa\{\vec{\omega}^m_k\}_{k=1}^{K_\Gamma} = d$, $m=0,\ldots, M-1$.
\end{itemize}
We refer to \cite{triplej} and \cite{gflows3d} for more details and for an
interpretation of this assumption. 
Given the above definitions, we introduce the piecewise linear 
vertex normal function
$\vec\omega^m := \sum_{k=1}^{K_\Gamma} \chi^m_k\,\vec\omega^m_k \in \Vh$,
and note that 
\begin{equation} 
\left\langle \vec{v}, w\,\vec\nu^m\right\rangle_{\Gamma^m}^h =
\left\langle \vec{v}, w\,\vec\omega^m\right\rangle_{\Gamma^m}^h 
\qquad \forall\ \vec{v} \in \Vh\,,\ w \in \Wh \,.
\label{eq:NI}
\end{equation}

\begin{thm} \label{thm:BGN}
Let the assumptions $(\mathcal{A})$ and ($\mathcal{B}$) hold.
If the LBB condition {\rm (\ref{eq:LBB})} holds, then there exists a unique
solution $(\vec U^{m+1}, P^{m+1}, \vec X^{m+1}, \kappa^{m+1}) 
\in \uspace^m\times\widehat\pspace^m \times \Vh \times \Wh$ 
to {\rm (\ref{eq:HGa}--d)}. In all other
cases there exists a unique solution 
$(\vec U^{m+1} , \vec X^{m+1}, \kappa^{m+1})\in 
\uspace^m_0 \times \Vh \times \Wh$ to the
reduced system {\rm (\ref{eq:HGa},c,d)} with $\uspace^m$ replaced by
$\uspace^m_0$.
In either case, there exists a unique solution 
$(\rho_\Gamma^{m+1}, \Psi^{m+1})\in [\Whp]^2$ to 
{\rm (\ref{eq:GDa0},f)} that satisfies {\rm (\ref{eq:consm})}. 
\end{thm}
\begin{proof}
The existence and uniqueness results for 
$(\vec U^{m+1}, P^{m+1}, \vec X^{m+1}, \kappa^{m+1})$ can be shown similarly to
the proof in Theorem~\ref{thm:GD}, and analogous to the proof in
\citet[Theorem~4.1]{fluidfbp}.
The results for $\rho_\Gamma^{m+1}$ and $\Psi^{m+1}$ can be shown exactly
as in the proof of Theorem~\ref{thm:GD}. 
\end{proof}

We remark that it does not appear possible to prove the analogues of
(\ref{eq:dmprho},c) for the scheme (\ref{eq:HGa}--f). 

\begin{thm} \label{thm:fdstabHG}
Let $\gamma$ be defined as in {\rm (\ref{eq:Fconst})}, let 
{\rm (\ref{eq:muconst})} hold, let $\rho_\Gamma^m=\rho_\Gamma^{m-1}=0$ 
and let 
$(\vec U^{m+1}, P^{m+1}, \vec X^{m+1}, \kappa^{m+1},\rho_\Gamma^{m+1})$
be a solution to {\rm (\ref{eq:HGa}--e)}. 
Then $\rho_\Gamma^{m+1} = 0$ and {\rm (\ref{eq:fdstabGD})} holds.
\end{thm}
\begin{proof}
It follows immediately from (\ref{eq:HGa0}) that $\rho_\Gamma^{m+1} = 0$.
Choosing $\vec\xi = \vec U^{m+1}$ in (\ref{eq:HGa}), 
$\varphi = P^{m+1}$ in (\ref{eq:HGb}), 
$\chi = \overline\gamma\,\kappa^{m+1}$ in (\ref{eq:HGc}) and
$\vec\eta=\overline\gamma\,({\vec X^{m+1}-\vec\id\!\mid_{\Gamma^m}})$ 
in (\ref{eq:HGd}) yields,
similarly to the proof of Theorem~\ref{thm:fdstabGD}, that (\ref{eq:fdstabGD})
holds.
\end{proof}

\begin{rem} \label{rem:oscm}
We may want to add numerical diffusion to {\rm (\ref{eq:HGa0})}, 
in order to avoid oscillations in $\rho_\Gamma^{m+1}$. Here we recall
{\rm Remark~\ref{rem:osch}}, and hence we would add the term 
$$ 
- \vartheta(h^m_\Gamma)\left\langle
\left| \mat {\mathcal{P}}_{\Gamma^m}
\left(\frac{\vec X^{m+1} - \vec\id}\tau - \vec U^{m+1}
\right)\right| \nabs\,\rho_\Gamma^{m}, \nabs\,
\chi_k^{m} \right \rangle_{\Gamma^{m}}^h$$ to the right hand side of
{\rm (\ref{eq:HGa0})}, and similarly the term
$$-\tfrac12\,\vartheta(h^m_\Gamma) \left\langle
\left| \mat {\mathcal{P}}_{\Gamma^m}
\left(\frac{\vec X^{m+1} - \vec \id}\tau - \vec U^{m+1}
\right)\right|
 \nabs\,\rho_\Gamma^{m}, \nabs\,
\pi^m\,[\vec U^m\,.\,\vec\xi] \right \rangle_{\Gamma^{m}}^h$$
to the right hand side of {\rm (\ref{eq:HGa})}. Here 
$h^m_\Gamma := \max_{j=1,\ldots,J_\Gamma} \diam{\sigma^m_j}$.
\end{rem}

\setcounter{equation}{0}
\section{Solution methods}  \label{sec:5}

As is standard practice for the solution of linear systems arising from
discretizations of Stokes and Navier--Stokes equations, we avoid the
complications of the constrained pressure space $\widehat\pspace^m$ in practice
by considering an overdetermined linear system with $\pspace^m$ instead. 
The assembly and the solution of the linear systems for the schemes
(\ref{eq:GDa}--f) and (\ref{eq:HGa}--f) at each time step are very similar to
the analogue procedures in \cite{fluidfbp,tpfs}, and so we omit most of the 
precise details here. 

\subsection{Assembly of bulk-interface cross terms}

In this subsection we give some more details about the assembly of the 
bulk-interface cross terms in (\ref{eq:GDa}--f) and (\ref{eq:HGa}--f) that are
new in this paper, and where the assembly is nontrivial.

For $\langle \rho_\Gamma^m\,\vec U^{m+1} , \vec\xi \rangle_{\Gamma^m}^h$,
with $\vec\xi\in\uspace^m$, we recall from (\ref{eq:defNI}) that
\begin{equation} \label{eq:rhoGammaNI}
\left\langle \rho_\Gamma^m\, \varphi^{\uspace^m}_l, \varphi^{\uspace^m}_i
 \right\rangle_{\Gamma^m}^h =
\tfrac1d \sum_{j=1}^{J_\Gamma} \mathcal{H}^{d-1}(\sigma^m_j)\,
\sum_{k=1}^{d} \rho_\Gamma^m(\vec{q}^m_{j_k})\,
\varphi^{\uspace^m}_l(\vec{q}^m_{j_k})
\,\varphi^{\uspace^m}_i(\vec{q}^m_{j_k})
\,,
\end{equation}
where $\{\vec{q}^m_{j_k}\}_{k=1}^{d}$ are the vertices of $\sigma^m_j$,
$j = 1,\ldots, J_\Gamma$, and
$\{\varphi_i^{\uspace^m}\}_{i=1}^{K_{\uspace}^m}$ 
denote the standard basis functions of $\uspace^m$.

\begin{algorithm}[H]
\caption{Calculate the matrix contributions for (\ref{eq:rhoGammaNI}). 
}
\label{algo:rhoGamma}
For all elements $\sigma^m$ of $\Gamma^m$ do \\
\quad For each vertex $\vec Q_i$ of $\sigma^m$, find the bulk element in which
$\vec Q_i$ lies and denote the \\ 
\quad local $S^m_2$ bulk basis functions on these elements with
$\varphi^{local,i}_k$, $k=1,\ldots,K$. \\
\quad For all $i=1,\ldots,d$ do \\
\quad \quad For all $k=1,\ldots,K$ do \\
\quad \quad \quad For all $l=1,\ldots,K$ do \\
\quad \quad \quad \quad Add 
$ \tfrac1d\, \mathcal{H}^{d-1}(\sigma^m)\, \rho_\Gamma^m(\vec Q_i) \,
\varphi_{k}^{local,i}(\vec Q_i)\,
\varphi_{l}^{local,i}(\vec Q_i)$ to
the contributions for \\ \quad \quad \quad \quad
$\left\langle \rho_\Gamma^m\, \varphi^{\uspace^m}_{global\_dof(k)}, 
\varphi^{\uspace^m}_{global\_dof(l)} \right\rangle_{\Gamma^m}^h$. \\
\quad \quad \quad end do \\
\quad \quad end do \\
\quad end do \\
end do
\end{algorithm}

In the above algorithm $\varphi_k^{local,i}$ 
is the hat-function for the local degree of freedom (DOF) $k$ 
on the element in which $\vec Q_i$ lies,
and $global\_dof(k)$ is a map that gives the global DOF in
$S^m_2$ for the local DOF $k$.

For $\left\langle \rho^{m-1}_\Gamma\,\vec I^m_2\,\vec U^{m},
 \vec\Pi_{m}^{m-1}\,\vec \xi\!\mid_{\Gamma^{m}} 
\right\rangle_{\Gamma^{m-1}}^h$, with $\vec\xi\in\uspace^m$, 
we note similarly that
\begin{equation} \label{eq:rhoGammaNI2}
\left\langle \rho_\Gamma^{m-1}\, \varphi^{\uspace^m}_l, 
\Pi_{m}^{m-1}\,\varphi^{\uspace^m}_i\!\mid_{\Gamma^{m}}
 \right\rangle_{\Gamma^{m-1}}^h =
\tfrac1d \sum_{j=1}^{J_\Gamma} \mathcal{H}^{d-1}(\sigma^{m-1}_j)\,
\sum_{k=1}^{d} \rho_\Gamma^{m-1}(\vec{q}^{m-1}_{j_k})\,
\varphi^{\uspace^m}_l(\vec{q}^{m-1}_{j_k})
\,\varphi^{\uspace^m}_i(\vec{q}^{m}_{j_k})
\,.
\end{equation}

\begin{algorithm}[H]
\caption{Calculate the matrix contributions for (\ref{eq:rhoGammaNI2}). 
}
\label{algo:rhoGamma2}
For all elements $\sigma^{m-1}$ of $\Gamma^{m-1}$ do \\
\quad For all $i=1,\ldots,d$ do \\
\quad\quad For each vertex $\vec Q_i^{m-1}$ of $\sigma^{m-1}$, 
find the bulk element in which $\vec Q_i^{m-1}$ lies and \\  
\quad \quad denote the local $S^m_2$ bulk basis functions on this element with
$\varphi^{local,i}_k$, \\ \quad \quad $k=1,\ldots,K$. 
Similarly, let $\widetilde\varphi^{local,i}_k$, $k = 1,\ldots, K$,
denote the local basis \\ \quad \quad 
functions on the element in which the vertex 
$\vec Q^{m}_i$ of $\sigma^{m}$ lies. \\
\quad \quad For all $k=1,\ldots,K$ do \\
\quad \quad \quad For all $l=1,\ldots,K$ do \\
\quad \quad \quad \quad Add 
$\tfrac1d\,\mathcal{H}^{d-1}(\sigma^{m-1})\,\rho_\Gamma^{m-1}(\vec Q^{m-1}_i)\,
\varphi_{k}^{local,i}(\vec Q^{m-1}_i)\,
\widetilde\varphi_{l}^{local,i}(\vec Q_i^{m})$ to the 
\\ \quad \quad \quad \quad
contributions for 
$\left\langle \rho_\Gamma^{m-1}\, \varphi^{\uspace^m}_{global\_dof(k)}, 
\Pi_{m}^{m-1}\,\varphi^{\uspace^m}_{global\_dof(l)}\!\mid_{\Gamma^{m}}
\right\rangle_{\Gamma^{m-1}}^h$. \\
\quad \quad \quad end do \\
\quad \quad end do \\
\quad end do \\
end do
\end{algorithm}

For the scheme (\ref{eq:HGa}--f) we note that for the terms
\begin{equation} \label{eq:Ximterm}
\left\langle\rho_{\Gamma,\star}^{m-1}\left(
 \frac{\vec X^{m} - \vec \id}{\tau} - \vec I^m_2\,\vec U^{m} \right),
\mat\Xi^{m-1}(\pi^{m-1}\,I^m_2\, U^{m}_i)\, \nabs\,(\pi^{m-1}\, \xi_i) 
\right\rangle_{\Gamma^{m-1}}^h,
\end{equation}
for $i=1,\ldots,d$, where $\vec\xi = (\xi_1,\ldots,\xi_d)^T\in\uspace^m$,
we need to consider the matrix entries
\begin{subequations}
\begin{equation} \label{eq:rhodXXinabs}
\left\langle\rho_{\Gamma,\star}^{m-1}\, \chi^{m-1}_k ,  
\mat\Xi^{m-1}(\pi^{m-1}\,I^m_2\, U^m_r)\,
\nabs\,(\pi^{m-1}\, \varphi^{\uspace^m}_i) 
\right\rangle_{\Gamma^{m-1}}^h
\end{equation}
and
\begin{equation} \label{eq:rhoUXinabs}
\left\langle\rho_{\Gamma,\star}^{m-1}\, \varphi^{\uspace^m}_j ,  
\mat\Xi^{m-1}(\pi^{m-1}\,I^m_2\,U^m_r)\,
\nabs\,(\pi^{m-1}\, \varphi^{\uspace^m}_i) 
\right\rangle_{\Gamma^{m-1}}^h .
\end{equation}
\end{subequations}
Here and throughout $\{\chi^m_k\}_{k=1}^{K^m_\Gamma}$ denotes the standard 
basis of $\Wh$, $m=0,\ldots,M-1$.

\begin{algorithm}[H]
\caption{Calculate the matrix contributions for (\ref{eq:rhodXXinabs}). 
}
\label{algo:rhodXXinabs}
For all elements $\sigma^{m-1}$ of $\Gamma^{m-1}$ do \\
\quad Compute 
$\vec G_{j} = \nabs\,\chi^{m-1}_{Q_j}$,
$j=1,\ldots,d$ for the $d$ vertices $\vec Q_1,\ldots,\vec Q_d$ of 
$\sigma^{m-1}$. \\
\quad Let $\mat M \in \R^{d \times d}$ be defined by its $d$ columns
$\vec Q_j - \vec Q_1$, $j= 2 \to d$, and $\vec \nu^{m-1}\!\mid_{\sigma^{m-1}}$.
\\ 
\quad Let $\mat{\hat\Lambda} \in \R^{d\times d}$ be the diagonal matrix with
diagonal entries \\ \quad
$\tfrac12\,((I^m_2\,U^m_r)(\vec Q_1) + (I^m_2\,U^m_r)(\vec Q_j))$, 
$j= 2 \to d$, and $0$. \\ 
\quad Define $\mat\Lambda = (\mat M^T)^{-1}\,\mat{\hat\Lambda}\,\mat M^T$. \\
\quad For each vertex $\vec Q_i$ of $\sigma^{m-1}$, 
find the bulk element in which $\vec Q_i$ lies and denote the \\ 
\quad local $S^m_2$ bulk basis functions on these elements with
$\varphi^{local,i}_k$, $k=1,\ldots,K$. \\
\quad For all $i=1,\ldots,d$ do \\
\quad \quad For all $j=1,\ldots,d$ do \\
\quad \quad \quad For all $l=1,\ldots,K$ do \\
\quad \quad \quad \quad Add 
$\tfrac1d\,\mathcal{H}^{d-1}(\sigma^{m-1})\,\rho^{m-1}_{\Gamma,\star}\,
\varphi_{l}^{local,j}(\vec Q_j)\,\mat\Lambda\,\vec G_{j}$ to the contributions 
for \\ \quad \quad \quad \quad  
$\left\langle\rho^{m-1}_{\Gamma,\star}\, \chi^{m-1}_{global\_dof(i)} ,
\mat\Xi^{m-1}(\pi^{m-1}\,I^m_2\,U^m_r)\,
 \nabs\, (\pi^{m-1}\,\varphi^{\uspace^m}_{global\_dof(l)}) 
\right\rangle_{\Gamma^{m-1}}^h$.
 \\
\quad \quad \quad \quad end do \\
\quad \quad end do \\
\quad end do \\
end do
\end{algorithm}
\begin{algorithm}[H]
\caption{Calculate the matrix contributions for (\ref{eq:rhoUXinabs}). 
}
\label{algo:rhoUXinabs}
For all elements $\sigma^{m-1}$ of $\Gamma^{m-1}$ do \\
\quad Compute 
$\vec G_{j} = \nabs\,\chi^{m-1}_{Q_j}$,
$j=1,\ldots,d$ for the $d$ vertices $\vec Q_1,\ldots,\vec Q_d$ of 
$\sigma^{m-1}$. \\
\quad Let $\mat M \in \R^{d \times d}$ be defined by its $d$ columns
$\vec Q_j - \vec Q_1$, $j= 2 \to d$, and $\vec \nu^{m-1}\!\mid_{\sigma^{m-1}}$.
\\ 
\quad Let $\mat{\hat\Lambda} \in \R^{d\times d}$ be the diagonal matrix with
diagonal entries \\ \quad
$\tfrac12\,((I^m_2\,U^m_r)(\vec Q_1) + (I^m_2\,U^m_r)(\vec Q_j))$, 
$j= 2 \to d$, and $0$. \\ 
\quad Define $\mat\Lambda = (\mat M^T)^{-1}\,\mat{\hat\Lambda}\,\mat M^T$. \\
\quad For each vertex $\vec Q_i$ of $\sigma^{m-1}$, 
find the bulk element in which $\vec Q_i$ lies and denote the \\ 
\quad local $S^m_2$ bulk basis functions on these elements with
$\varphi^{local,i}_k$, $k=1,\ldots,K$. \\
\quad For all $i=1,\ldots,d$ do \\
\quad \quad For all $j=1,\ldots,d$ do \\
\quad \quad \quad For all $k=1,\ldots,K$ do \\
\quad \quad \quad \quad For all $l=1,\ldots,K$ do \\
\quad \quad \quad \quad \quad Add 
$\tfrac1d\,\mathcal{H}^{d-1}(\sigma^{m-1})\,\rho^{m-1}_{\Gamma,\star}\,
\varphi_{k}^{local,i}(\vec Q_i)\,
\varphi_{l}^{local,j}(\vec Q_j)\,\mat\Lambda\,\vec G_{j}$ to the \\
\quad \quad \quad \quad \quad 
 contributions for  \\ 
\mbox{\quad \quad \quad \quad \quad 
$\left\langle\rho^{m-1}_{\Gamma,\star}\, \varphi^{\uspace^m}_{global\_dof(k)},
\mat\Xi^{m-1}(\pi^{m-1}\,I^m_2\,U^m_r)\,
 \nabs\, (\pi^{m-1}\,\varphi^{\uspace^m}_{global\_dof(l)}) 
\right\rangle_{\Gamma^{m-1}}^h$}.
 \\
\quad \quad \quad \quad end do \\
\quad \quad \quad end do \\
\quad \quad end do \\
\quad end do \\
end do
\end{algorithm}

The remaining new terms are
\begin{align} \label{eq:assembleml}
2 \left\langle \mu_\Gamma(\Psi^m)\, \mat D_s^m (\vec\pi^m\,\vec U^{m+1}) , 
\mat D_s^m ( \vec\pi^m\,\vec \xi ) \right\rangle_{\Gamma^m}^h
\quad\text{and}\quad
\left\langle\lambda_\Gamma(\Psi^m) \,
 \nabs\,.\,(\vec\pi^m\,\vec U^{m+1}) , 
\nabs\,.\, (\vec\pi^m\,\vec \xi) \right\rangle_{\Gamma^m}^h 
\end{align}
in (\ref{eq:GDa}), where $\vec\xi\in\uspace^m$. 
For an element $\sigma^m \subset \Gamma^m$ let
$\{\vec t_j\}_{j=1}^{d-1} \cup \{\vec\nu^m\}$ be an ONB of $\R^d$. 
Then it holds in the case $d=2$ that
\begin{subequations}
\begin{align}
& 2
\left(\left\langle \mu_\Gamma(\Psi^m)\,
\mat D_s^m (\pi^m\, \varphi^{\uspace^m}_j\,\vec\ek_k) , 
\mat D_s^m ( \pi^m\, \varphi^{\uspace^m}_i\,\vec\ek_l ) 
\right\rangle_{\sigma^m}^h \right)_{k,l=1}^d \nonumber \\ & \qquad
= 2\, \left\langle \mu_\Gamma(\Psi^m), 1 \right\rangle_{\sigma^m}^h
 \partial_{\vec t_1} (\pi^m\, \varphi^{\uspace^m}_j)
\,  \partial_{\vec t_1} (\pi^m\, \varphi^{\uspace^m}_i) 
\, \vec t_1 \otimes \vec t_1  \nonumber \\ & \qquad
=: 2\, \left\langle \mu_\Gamma(\Psi^m), 1 \right\rangle_{\sigma^m}^h
 \mat L (\pi^m\, \varphi^{\uspace^m}_j, \pi^m\, \varphi^{\uspace^m}_i )
\,. 
\label{eq:mgt2d}
\end{align}
Similarly, it holds in the case $d=3$ that
\begin{align}
& 2
\left(\left\langle \mu_\Gamma(\Psi^m)\,
\mat D_s^m (\pi^m\, \varphi^{\uspace^m}_j\,\vec\ek_k) , 
\mat D_s^m ( \pi^m\, \varphi^{\uspace^m}_i\,\vec\ek_l ) 
\right\rangle_{\sigma^m}^h \right)_{k,l=1}^d \nonumber  \\ & \qquad
= \left\langle \mu_\Gamma(\Psi^m), 1 \right\rangle_{\sigma^m}^h
 \left[ \vphantom{\sum_{b=1}^d}
 \partial_{\vec t_1} (\pi^m\, \varphi^{\uspace^m}_j)\,
 \partial_{\vec t_1} (\pi^m\, \varphi^{\uspace^m}_i) 
\, \vec t_2 \otimes \vec t_2
+  \partial_{\vec t_2} (\pi^m\, \varphi^{\uspace^m}_j)\,
 \partial_{\vec t_1} (\pi^m\, \varphi^{\uspace^m}_i) 
\, \vec t_1 \otimes \vec t_2  \right. \nonumber \\ & \qquad \qquad \left.
+ \partial_{\vec t_1} (\pi^m\, \varphi^{\uspace^m}_j)\,
 \partial_{\vec t_2} (\pi^m\, \varphi^{\uspace^m}_i) 
\, \vec t_2 \otimes \vec t_1
+ \partial_{\vec t_2} (\pi^m\, \varphi^{\uspace^m}_j)\,
 \partial_{\vec t_2} (\pi^m\, \varphi^{\uspace^m}_i) 
\, \vec t_1 \otimes \vec t_1 \right. \nonumber \\ & \qquad \qquad \left.
+ 2 \sum_{b=1}^2 
\partial_{\vec t_b} (\pi^m\, \varphi^{\uspace^m}_j)
\,  \partial_{\vec t_b} (\pi^m\, \varphi^{\uspace^m}_i) 
\, \vec t_b \otimes \vec t_b \right] \nonumber \\ & \qquad
=: 2\, \left\langle \mu_\Gamma(\Psi^m), 1 \right\rangle_{\sigma^m}^h
 \mat L (\pi^m\, \varphi^{\uspace^m}_j, \pi^m\, \varphi^{\uspace^m}_i )
\,.  \label{eq:mgt3d}
\end{align}
\end{subequations}

Moreover, we have that
\begin{align}
&
\left(\left\langle \lambda_\Gamma(\Psi^m) \,
\nabs\,.\, (\pi^m\, \varphi^{\uspace^m}_j\,\vec\ek_k) , 
\nabs\,.\, ( \pi^m\, \varphi^{\uspace^m}_i\,\vec\ek_l ) 
\right\rangle_{\sigma^m}^h \right)_{k,l=1}^d \nonumber \\ & \qquad
= \left(\left\langle \lambda_\Gamma(\Psi^m) \,
\nabs\, (\pi^m\, \varphi^{\uspace^m}_j)\,.\,\vec\ek_k , 
\nabs\,( \pi^m\, \varphi^{\uspace^m}_i)\,.\,\vec\ek_l 
\right\rangle_{\sigma^m}^h \right)_{k,l=1}^d \nonumber \\ & \qquad
=  \left\langle \lambda_\Gamma(\Psi^m), 1
\right\rangle_{\sigma^m}^h
\nabs\, (\pi^m\, \varphi^{\uspace^m}_j) \otimes 
\nabs\,( \pi^m\, \varphi^{\uspace^m}_i)  \,. \label{eq:lgt}
\end{align}

\begin{algorithm}[H]
\caption{Calculate the matrix contributions for (\ref{eq:assembleml}). 
}
\label{algo:ml}
For all elements $\sigma^m$ of $\Gamma^m$ do \\
\quad Compute 
$\vec S_{i} = \nabs\,\chi^m_{Q_i}$,
$i=1,\ldots,d$ for the $d$ vertices $\vec Q_1,\ldots,\vec Q_d$ \\ 
\quad of $\sigma^m$. \\
\quad Compute $\mat K_{ij} = 
\left\langle \lambda_\Gamma(\Psi^m) , 1 
\right\rangle_{\sigma^m}^h\,\vec S_i \otimes \vec S_j
+ 2\, \left\langle \mu_\Gamma(\Psi^m), 1 \right\rangle_{\sigma^m}^h
 \mat L (\chi^m_{Q_i}, \chi^m_{Q_j} )$, \\ \quad
$i,j=1,\ldots,d$. \\
\quad For each vertex $\vec Q_i$ of $\sigma^m$, find the bulk element in which
$\vec Q_i$ lies and denote the \\ 
\quad local $S^m_2$ bulk basis functions on these elements with
$\varphi^{local,i}_k$, $k=1,\ldots,K$. \\
\quad For all $i=1,\ldots,d$ do \\
\quad \quad For all $j=1,\ldots,d$ do \\
\quad \quad \quad For all $k=1,\ldots,K$ do \\
\quad \quad \quad \quad For all $l=1,\ldots,K$ do \\
\quad \quad \quad \quad \quad Add 
$\varphi_{k}^{local,i}(\vec Q_i)\,
\varphi_{l}^{local,j}(\vec Q_j)\,\mat K_{ij}$ to
the contributions for \\ \quad \quad \quad \quad \quad 
$ \left(
\left\langle \lambda_\Gamma(\Psi^m) \,
\nabs\,.\,(\varphi_{global\_dof(k)}^{\uspace^m}\,
\vec \ek_r), \nabs\,.\,(\varphi_{global\_dof(l)}^{\uspace^m}\,
\vec \ek_s) \right\rangle_{\Gamma^m}^h \right)_{r,s=1}^d$ 
\\ \quad \quad \quad \quad \quad 
$ + 2
\left(\left\langle \mu_\Gamma(\Psi^m)\,
\mat D_s^m (\pi^m\, \varphi_{global\_dof(k)}^{\uspace^m}\,\vec\ek_r) , 
\mat D_s^m ( \pi^m\, \varphi_{global\_dof(l)}^{\uspace^m}\,\vec \ek_s ) 
\right\rangle_{\Gamma^m}^h \right)_{r,s=1}^d$. \\
\quad \quad \quad \quad end do \\
\quad \quad \quad end do \\
\quad \quad end do \\
\quad end do \\
end do
\end{algorithm}

\subsection{Inhomogenous boundary data} \label{sec:37}
With a view towards some numerical test cases in Section~\ref{sec:6},
we also allow for 
an inhomogeneous Dirichlet boundary condition $\vec g$ on $\partial\Omega$
and for ease of exposition
consider only piecewise quadratic velocity approximations.
Then we reformulate e.g.\ 
\mbox{(\ref{eq:HGa}--d)} as follows.
Find $\vec U^{m+1} \in \uspace^m(\vec g) :=
\{\vec U \in [S^m_2]^d : \vec U = \vec I^m_2\,\vec g\ 
\text{on}\ \partial\Omega\}$,
$P^{m+1} \in \pspace^m$, $\vec X^{m+1}\in\Vh$ and 
$\kappa^{m+1} \in \Wh$ such that (\ref{eq:HGa},c,d)
with $\uspace^m = [S^m_2]^d \cap \uspace$ hold together with
\begin{equation} \label{eq:LAb}
 \left(\nabla\,.\,\vec U^{m+1}, \varphi\right) = 
\frac{\left(\varphi, 1\right)}{\mathcal{L}^d(\Omega)}\, 
\int_{\partial\Omega} (\vec I^m_2\,\vec g) \,.\, \unitn \dH{d-1} 
\quad \forall\ \varphi \in \pspace^m\,.
\end{equation}
If $(\uspace^m, \pspace^m)$ satisfy the LBB condition (\ref{eq:LBB}), then
the existence and uniqueness proof for a solution to (\ref{eq:HGa},c,d), 
(\ref{eq:LAb}) is as before. In the absence of (\ref{eq:LBB}), the existence
and uniqueness of a solution to the reduced system that is analogous to 
(\ref{eq:HGa},c,d), with $\uspace^m$ replaced by $\uspace^m_0$,
hinges on the nonemptiness of the set
$\uspace^m_0(\vec g) := 
\{ \vec U \in \uspace^m(\vec g) : (\nabla\,.\,\vec U, \varphi) = 0 \quad
\forall\ \varphi \in \widehat\pspace^m \}$.

\setcounter{equation}{0}
\section{Numerical results}  \label{sec:6}

For the bulk mesh adaptation we use the strategy from
\cite{fluidfbp}, which results in a fine mesh size
$h_f$ around $\Gamma^m$ and a coarse mesh size $h_c$ 
further away from it. 
Here $h_{f} = \frac{2\,\min\{H_1,H_2\}}{N_{f}}$ and 
$h_{c} =  \frac{2\,\min\{H_1,H_2\}}{N_{c}}$
are given by two integer numbers $N_f >  N_c$, where we assume from now on that
the convex hull of 
$\Omega$ is given by $\times_{i=1}^d (-H_i,H_i)$. 
We remark that we implemented the schemes (\ref{eq:GDa}--f) and
(\ref{eq:HGa}--f) with the help of
the finite element toolbox ALBERTA, see \cite{Alberta}.

For the scheme (\ref{eq:HGa}--f) we fix $\epsilon = 10^{-8}$, and in all our
numerical experiments presented in this section
the discrete surfactant concentration $\Psi^m$ remained above 
$\epsilon$ throughout the evolution, so that $\gamma_\epsilon(\Psi^m) =
\gamma(\Psi^m)$, recall (\ref{eq:geps}). 
Similarly, the discrete surface material density $\rho_\Gamma^m$ always
remained nonnegative in all our numerical simulations.
Unless otherwise stated we use the linear equation of state (\ref{eq:gamma1}) 
for the surface tension, and for the numerical simulations without surfactant 
we set $\beta = 0$ in (\ref{eq:gamma1}). Similarly, we set the 
numerical diffusion in (\ref{eq:HGa0}) to be zero, 
$\vartheta(s) = 0$ for all $s\in \R$, unless otherwise stated. 
We set $\Psi^0 = \psi_0 = 1$
and $\rho_\Gamma^0 = \rho_{\Gamma,0} = 1$, unless
stated otherwise.
In addition, we employ the lowest order
Taylor--Hood element P2--P1 in all computations and
set $\vec U^0 = \vec I^0_2\,\vec u_0$, where $\vec u_0 = \vec 0$
unless stated otherwise.
For the initial interface we always choose a circle/sphere of radius $R_0$ 
and set
$\kappa^0 = -\frac{d-1}{R_0}$ for the scheme (\ref{eq:HGa}--f). For the scheme
(\ref{eq:GDa}--f) we let $\vec\kappa^0 \in \Vhz$ be the solution of 
(\ref{eq:GDd}) with $m$ and $m+1$ replaced by zero.
To summarize the discretization parameters we use the shorthand notation
$n\,{\rm adapt}_{k,l}$ from \cite{fluidfbp}. 
The subscripts refer to the fineness of the spatial discretizations, i.e.\
for the set $n\,{\rm adapt}_{k, l}$ it holds that 
$N_f = 2^k$ and $N_c = 2^l$. For the case $d=2$ we have in addition that 
$K_\Gamma = J_\Gamma = 2^k$, while for $d=3$ it holds that
$(K_\Gamma, J_\Gamma) = (1538, 3072)$ for $k = 5$. 
Finally, the uniform time step size 
for the set $n\,{\rm adapt}_{k,l}$ is given by $\tau = 10^{-3} / n$,
and if $n=1$ we write ${\rm adapt}_{k, l}$.

\subsection{Convergence experiments}
In order to test our finite element approximations, we
consider the true solution of an expanding circle/sphere,
as it has been considered e.g.\ \cite{spurious} for the special case
$\rho = \rho_{\Gamma,0}=0$, (\ref{eq:Fconst}) and (\ref{eq:muconst}) 
with $\overline\mu_\Gamma = \overline\lambda_\Gamma = 0$, so that the model
(\ref{eq:NSa}--e), (\ref{eq:sigma}), (\ref{eq:1c}--c),
(\ref{eq:sigmaG}), (\ref{eq:1surf}) collapses to
(\ref{eq:NSa}--e) with $\rho=0$, 
$[\mat\sigma\,\vec \nu]_-^+ = -\gamma\,\varkappa\,\vec\nu$ and
(\ref{eq:1d}).

Throughout this subsection we only consider the case that both 
(\ref{eq:Fconst}) 
and (\ref{eq:muconst}) hold, and that $\partial_1\Omega = \partial\Omega$
and $\vec f_2 = \vec 0$. A nontrivial 
divergence free and radially symmetric solution $\vec u$ can be
constructed on a domain that does not contain the origin. To this end, consider
e.g.\
$\Omega = (-H,H)^d \setminus [-H_0, H_0]^d$, with $0 < H_0 < H$. Then
$\Gamma(t) := \{ \vec z \in \R^d : |\vec z| = r(t)\}$, where
\begin{subequations}
\begin{equation} \label{eq:radialr2}
r(t) = ([r(0)]^d + \alpha\,t\,d)^\frac1d \,,
\end{equation}
together with 
\begin{equation} \label{eq:radialup2}
\vec u(\vec z, t) = \alpha\,|\vec z|^{-d}\,\vec z \,, \quad
p(\vec z, t) = \theta(t)\left[ \charfcn{\Omega_-(t)} - 
\frac{\mathcal{L}^d(\Omega_-(t))}{\mathcal{L}^d(\Omega)}\right],\quad
\rho_\Gamma(\vec z,t) = \left[ \frac{r(0)}{r(t)} \right]^{d-1}
\overline\rho_{\Gamma,0}
\,,
\end{equation}
\end{subequations}
where $\overline\rho_{\Gamma,0} \in \Rgeq$ and
\begin{align*}
\theta(t) &= \left[\overline\gamma + \frac\alpha{[r(t)]^{d}} \left(
2\,\overline\mu_\Gamma + (d-1)\,\overline\lambda_\Gamma 
\right)  
- \alpha^2 \left[\frac{r(0)}{[r(t)]^3} \right]^{d-1} \overline\rho_{\Gamma,0} 
\right] \frac{d-1}{r(t)} \\ & \quad + 
2\,\alpha\,\frac{d-1}{[r(t)]^{d}}\,(\mu_+ - \mu_-)\,,
\end{align*}
is an exact solution to the problem 
(\ref{eq:NSa}--e), (\ref{eq:sigma}), (\ref{eq:1c}--c), 
(\ref{eq:sigmaG}) with 
$\vec f_1(\vec z, t) = \alpha^2\,(1-d)\,\vec z \,|\vec z|^{-2\,d}$ and
with the homogeneous right
hand side in (\ref{eq:NSc}) replaced by $\vec g$, where 
$\vec g(\vec z) = \alpha\,|\vec z|^{-d}\,\vec z$.

We perform convergence experiments for the solution (\ref{eq:radialr2},b)
for the case $d=2$. In particular, 
we fix 
$\Omega = (-1,1)^2 \setminus [-\frac13,\frac13]^2$ and use the parameters
$$
\alpha = 0.15 \quad\text{and}\quad \rho=0\,,\quad
\mu = \overline\mu_\Gamma =
\overline\lambda_\Gamma = \overline\gamma = \rho_{\Gamma,0} = 1
$$
for the true solution (\ref{eq:radialr2},b)
and set $\Gamma(0) = \{ \vec z \in \R^d : |\vec z| = \frac12 \}$.

With $T=1$ we obtain that $\Gamma(T)$ is a circle of radius
$r(1) = \sqrt{0.55} \approx 0.742$. Some errors for the approximation
(\ref{eq:GDa}--f), where we use uniform bulk meshes with $h_c = h_f = h$
and $h^m_\Gamma \approx h/3$, 
are shown in Table~\ref{tab:GDrho}.
Here we define the errors 
$$\errorXx :=
\max_{m=1,\ldots, M} \|\vec X^m - \vec{x}(\cdot,t_m)\|_{L^\infty}\,,$$ 
where $\|\vec X(t_m) - \vec{x}(\cdot,t_m)\|_{L^\infty}
:=\max_{k=1,\ldots, K_\Gamma}
\left\{\min_{\vec{y}\in \Upsilon} |\vec{q}^m_k - 
\vec{x}(\vec{y},t_m)|\right\}$ and
$$\errorUu2 :=
\max_{m=1,\ldots, M}\|U^m - \vec I^m_2\,u(\cdot,t_m)\|_{L^\infty(\Omega)}\,.$$
In order to evaluate the errors in the pressure, we define
$
\LerrorPpc :=[\tau\,\sum_{m=1}^M \|P^m_c - p_c(\cdot,t_m)
\|_{L^2(\Omega)}^2 ]^\frac12
$
and $\LerrorLl := [\tau\,\sum_{m=1}^M |\theta^m - \theta(t_m)|^2]^\frac12$. 
Here
$p_c(\cdot,t_m):= p(\cdot,t_m)-\theta(t_m)\,\charfcn{\Omega_-(t_m)} \in \R$
for the test problem (\ref{eq:radialr2},b),
and $P^{m}_c := P^{m} - \theta^{m}\,\charfcn{\Omega^{m-1}_-}$
is piecewise polynomial on $\mathcal{T}^{m-1}$.
\begin{table}
\center
\begin{tabular}{llllll}
\hline
 $1/h$ & $\tau$ & $\errorXx$ & $\errorUu2$ & $\LerrorPpc$ 
 & $\LerrorLl$ \\ 
 \hline 
   3  & $10^{-2}$ & 5.6209e-03 & 1.2984e-01 & 5.7124e-01 & 8.2619e-01 \\
   6  & $10^{-3}$ & 5.8122e-04 & 4.7725e-02 & 7.0856e-02 & 5.5830e-02 \\
   12 & $10^{-4}$ & 7.3525e-05 & 2.5878e-02 & 2.1928e-02 & 1.7614e-02 \\
\hline
\end{tabular}
\caption{($\alpha = 0.15$,
$\mu=\overline\gamma=\overline\mu_\Gamma=\overline\lambda_\Gamma=\overline\rho_{\Gamma,0}=1$)
Expanding bubble problem on $(-1,1)^2 \setminus [-\frac13,\frac13]^2$ 
over the time interval $[0,1]$ 
for the P2--P1 element with \XFEMGAMMA\ for the scheme (\ref{eq:GDa}--f).
Here we use uniform meshes.}
\label{tab:GDrho}
\end{table}%

In Table~\ref{tab:GDrho} the convergence in $\errorUu2$ appears to be very slow.
It is for this reason that we repeat the convergence experiment also on a
sequence of refined bulk meshes. Here we use adaptively refined grids with
$h_f = h_c / 8$ and $h_\Gamma \approx h_c / 12 = \tfrac23\,h_f$. 
The corresponding errors can be found in Table~\ref{tab:GDrhor2},
where now the error $\errorUu2$ appears to converge with an improved rate.
\begin{table}
\center
\begin{tabular}{llllll}
\hline
 $1/h_{f}$ & $\tau$ & $\errorXx$ & $\errorUu2$ & $\LerrorPpc$ & $\LerrorLl$ \\
 \hline 
   24 & $10^{-2}$ & 6.2091e-04 & 1.1700e-02 & 2.8168e-01 & 2.5660e-01 \\
   48 & $10^{-3}$ & 9.0002e-05 & 1.9780e-03 & 3.1748e-02 & 3.5368e-02 \\
   96 & $10^{-4}$ & 8.9183e-06 & 3.2252e-04 & 7.9251e-03 & 8.2088e-03 \\
\hline
\end{tabular}
\caption{($\alpha = 0.15$, 
$\mu=\overline\gamma=\overline\mu_\Gamma=\overline\lambda_\Gamma=
\overline\rho_{\Gamma,0}=1$)
Expanding bubble problem on $(-1,1)^2 \setminus [-\frac13,\frac13]^2$ 
over the time interval $[0,1]$ for the P2--P1 element with \XFEMGAMMA\ 
for the scheme (\ref{eq:GDa}--f). Here we use adaptive meshes.}
\label{tab:GDrhor2}
\end{table}%
The errors for the finite element approximation (\ref{eq:HGa}--f) are very
similar, see Tables~\ref{tab:rho} and \ref{tab:rhor2}. 
\begin{table}
\center
\begin{tabular}{llllll}
\hline
 $1/h$ & $\tau$ & $\errorXx$ & $\errorUu2$ & $\LerrorPpc$ 
 & $\LerrorLl$ \\ 
 \hline 
   3  & $10^{-2}$ & 2.7615e-03 & 1.7447e-02 & 3.1799e-01 & 3.6399e-01 \\ 
   6  & $10^{-3}$ & 2.0666e-04 & 2.5265e-03 & 4.0843e-02 & 9.7883e-02 \\ 
   12 & $10^{-4}$ & 3.3724e-05 & 7.7310e-04 & 8.3360e-03 & 9.1464e-03 \\ 
\hline
\end{tabular}
\caption{($\alpha = 0.15$,
$\mu=\overline\gamma=\overline\mu_\Gamma=\overline\lambda_\Gamma=\overline\rho_{\Gamma,0}=1$)
Expanding bubble problem on $(-1,1)^2 \setminus [-\frac13,\frac13]^2$ 
over the time interval $[0,1]$ 
for the P2--P1 element with \XFEMGAMMA\ for the scheme (\ref{eq:HGa}--f).
Here we use uniform meshes.}
\label{tab:rho}
\end{table}%
\begin{table}
\center
\begin{tabular}{llllll}
\hline
 $1/h_{f}$ & $\tau$ & $\errorXx$ & $\errorUu2$ & $\LerrorPpc$ & $\LerrorLl$ \\
 \hline 
   24 & $10^{-2}$ & 6.4263e-04 & 1.1700e-02 & 2.7907e-01 & 2.5256e-01 \\ 
   48 & $10^{-3}$ & 9.5236e-05 & 1.9780e-03 & 3.1747e-02 & 3.5357e-02 \\ 
   96 & $10^{-4}$ & 1.0197e-05 & 3.2348e-04 & 7.9294e-03 & 8.2135e-03 \\ 
\hline
\end{tabular}
\caption{($\alpha = 0.15$,
$\mu=\overline\gamma=\overline\mu_\Gamma=\overline\lambda_\Gamma=\overline\rho_{\Gamma,0}=1$)
Expanding bubble problem on $(-1,1)^2 \setminus [-\frac13,\frac13]^2$ 
over the time interval $[0,1]$ for the P2--P1 element with \XFEMGAMMA\ 
for the scheme (\ref{eq:HGa}--f). Here we use adaptive meshes.}
\label{tab:rhor2}
\end{table}%

\subsection{Numerical experiments in 2d}

\subsubsection{Bubble in shear flow} \label{sec:611}

In the literature on numerical methods for two-phase flow with insoluble
surfactant it is often common to consider shear flow experiments for an
initially circular bubble in order to study the effect of surfactants and of
different equations of state. In this subsection, we will perform such
simulations for our preferred scheme (\ref{eq:HGa}--f).
Here we consider the setup from \citet[Fig.~1]{LaiTH08}. In particular,
we let $\Omega = (-5,5) \times (-2,2)$ and
prescribe the inhomogeneous Dirichlet boundary condition
$\vec g(\vec z) = (\frac12\,z_2, 0)^T$ on $\partial\Omega = \partial_1\Omega$.
Moreover, $\Gamma_0 = \{\vec z \in \R^2 : |\vec z | = 1 \}$. The physical
parameters are given by
\begin{equation} \label{eq:Lai}
\rho =  1\,, \quad \mu = 0.1\,,\quad 
\overline\gamma = 0.2\,,\quad \Ds = 0.1\,,\quad
\vec f = \vec 0\,,\quad \vec u_0 = \vec g\,.
\end{equation}
First we investigate the effect of different surface viscosity strengths on
the evolution in the absence of surfactants and surface mass. 
I.e.\ we have
$\rho_{\Gamma,0} = 0$ and the surface tension is constant,
see (\ref{eq:Fconst}). See Figure~\ref{fig:rho0} for some time evolutions
for different values of $\overline\mu_\Gamma = \overline\lambda_\Gamma$.
We note that for larger values of the surface viscosities, the effect of the
shearing flow on the shape of the bubble is reduced.
\begin{figure}
\center
\newcommand\localwidth{0.30\textwidth}
\includegraphics[angle=-90,width=\localwidth]{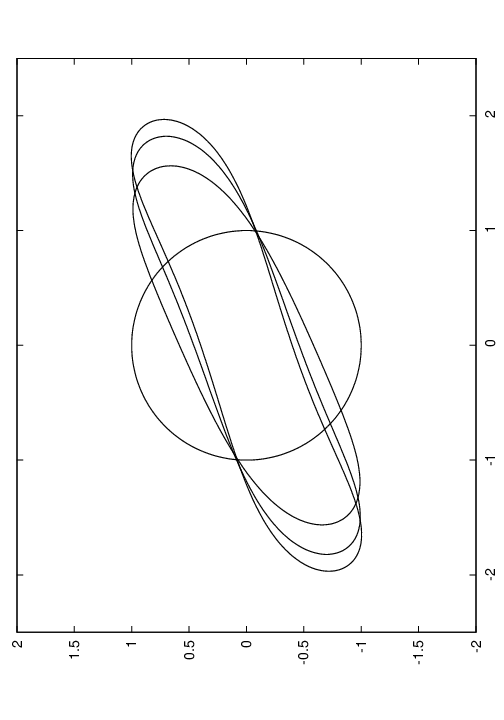} 
\includegraphics[angle=-90,width=\localwidth]{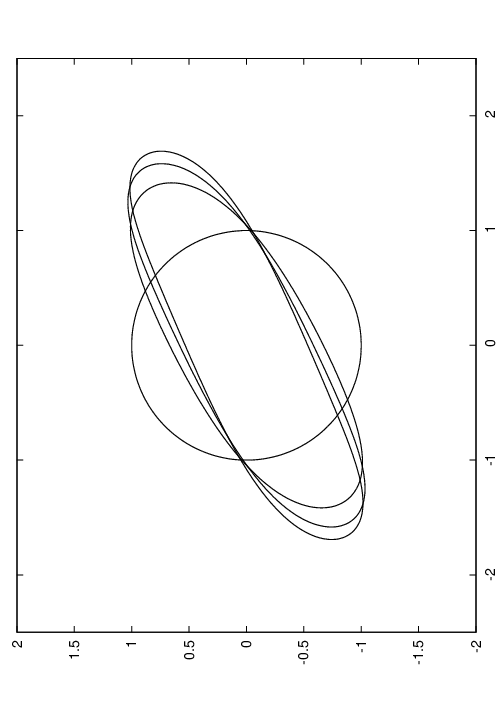}
\includegraphics[angle=-90,width=\localwidth]{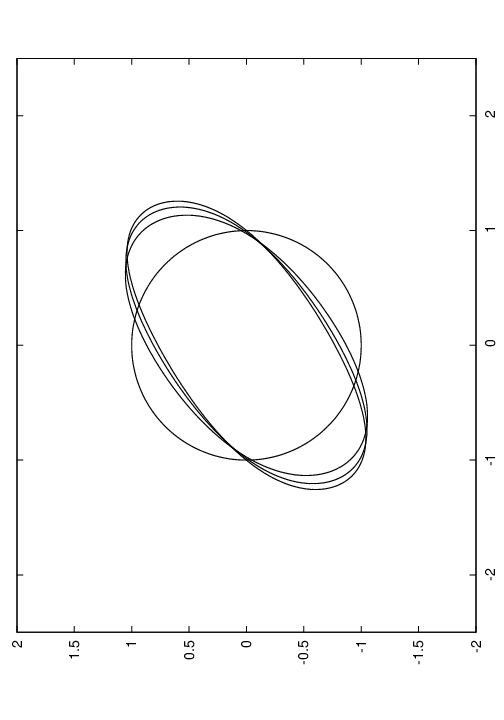}
\caption{(2\,adapt$_{9,4}$)
The time evolution of a drop in shear flow with $\rho_{\Gamma,0}=0$
for (\ref{eq:Fconst}) and (\ref{eq:muconst}) with 
$\overline\mu_\Gamma = \overline\lambda_\Gamma = 0.01$ (left), 
$\overline\mu_\Gamma = \overline\lambda_\Gamma = 1$ (middle) and 
$\overline\mu_\Gamma = \overline\lambda_\Gamma = 10$ (right). 
Plots are at times $t=0,\,4,\,8,\,12$. 
}
\label{fig:rho0}
\end{figure}%
The same experiments with surface mass present, 
i.e.\ for $\rho_{\Gamma,0} = 1$, can be seen in Figure~\ref{fig:rho1}.
In general, there are not many differences to the evolutions shown in
Figure~\ref{fig:rho0}. However, for
small surface viscosity constants there is a marked difference in the
evolution. In particular, the bubble appears to be shearing more 
when surface mass is present. 
\begin{figure}
\center
\newcommand\localwidth{0.3\textwidth}
\includegraphics[angle=-90,width=\localwidth]{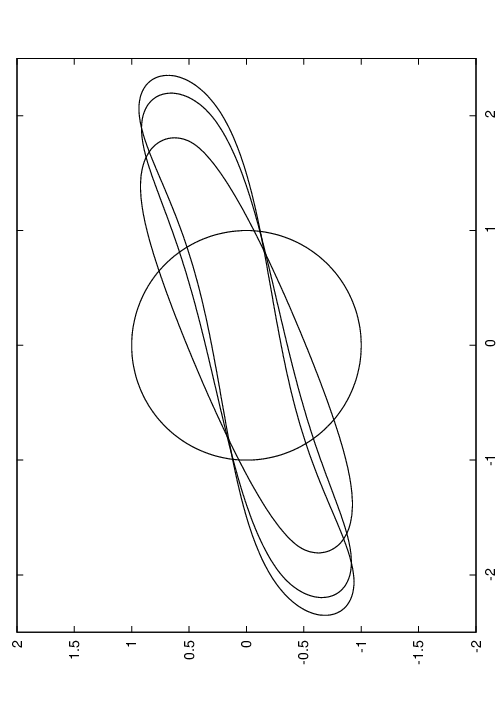}
\includegraphics[angle=-90,width=\localwidth]{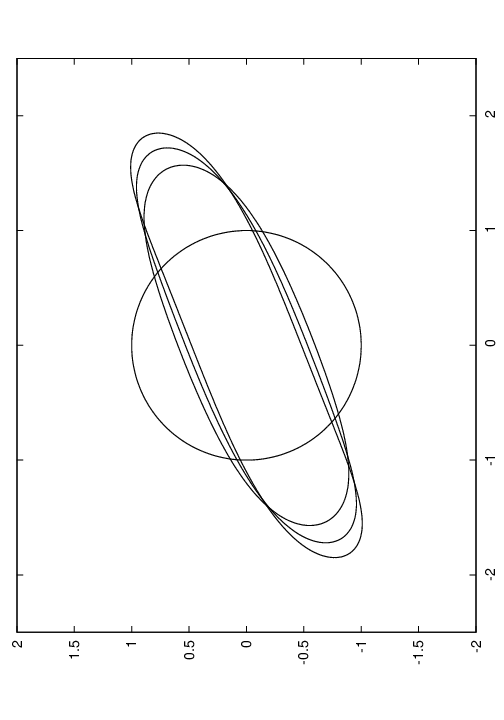} 
\includegraphics[angle=-90,width=\localwidth]{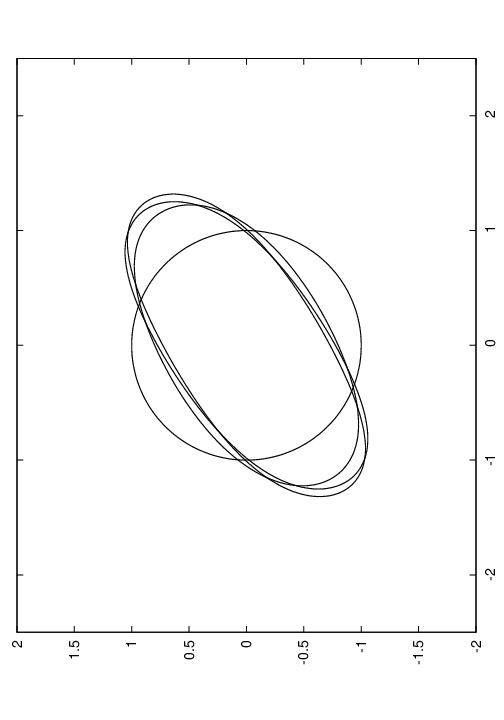}
\caption{(2\,adapt$_{9,4}$)
The time evolution of a drop in shear flow with $\rho_{\Gamma,0}=1$
for (\ref{eq:Fconst}) and (\ref{eq:muconst}) with 
$\overline\mu_\Gamma = \overline\lambda_\Gamma = 0.01$ (left), 
$\overline\mu_\Gamma = \overline\lambda_\Gamma = 1$ (middle) and 
$\overline\mu_\Gamma = \overline\lambda_\Gamma = 10$ (right). 
Plots are at times $t=0,\,4,\,8,\,12$.
}
\label{fig:rho1}
\end{figure}%
Details of the surface mass distribution at the final time $t=12$ can be seen
in Figure~\ref{fig:rho1_1d}, while velocity plots are given in 
Figure~\ref{fig:2dLai_u}.
\begin{figure}
\center
\newcommand\localwidth{0.3\textwidth}
\includegraphics[angle=-90,width=\localwidth]{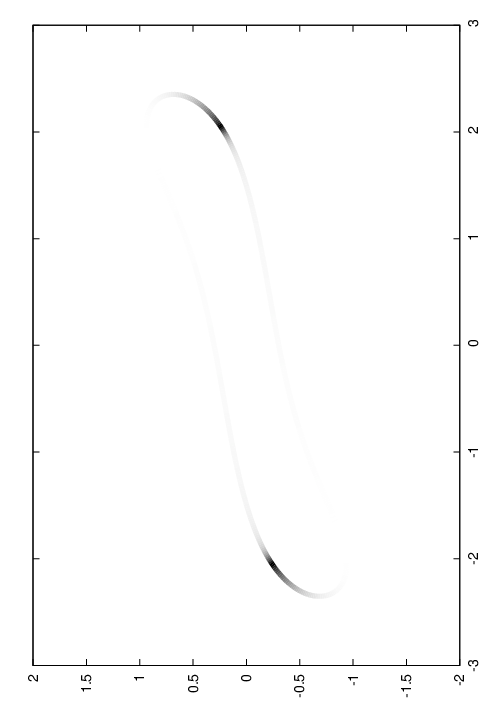}
\includegraphics[angle=-90,width=\localwidth]{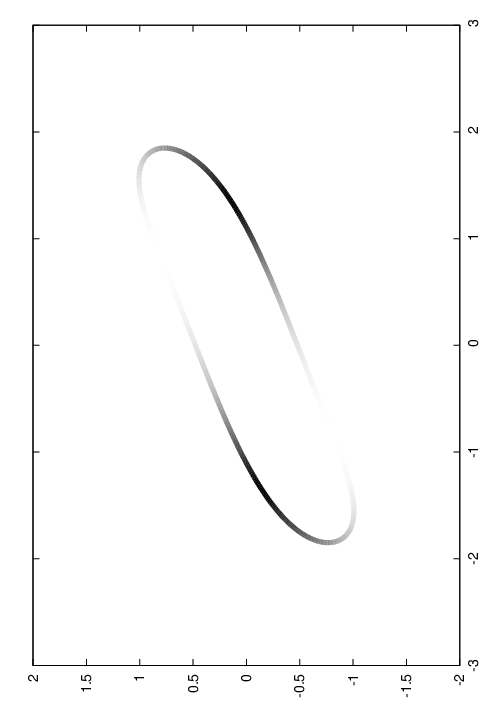}
\includegraphics[angle=-90,width=\localwidth]{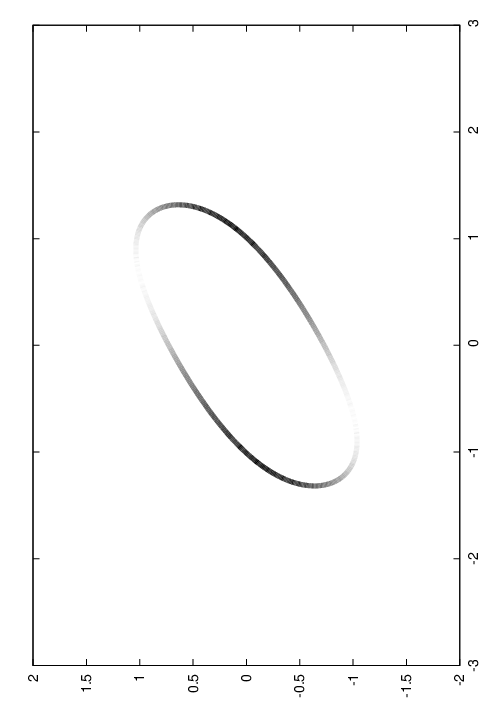}
\includegraphics[angle=-90,width=\localwidth]{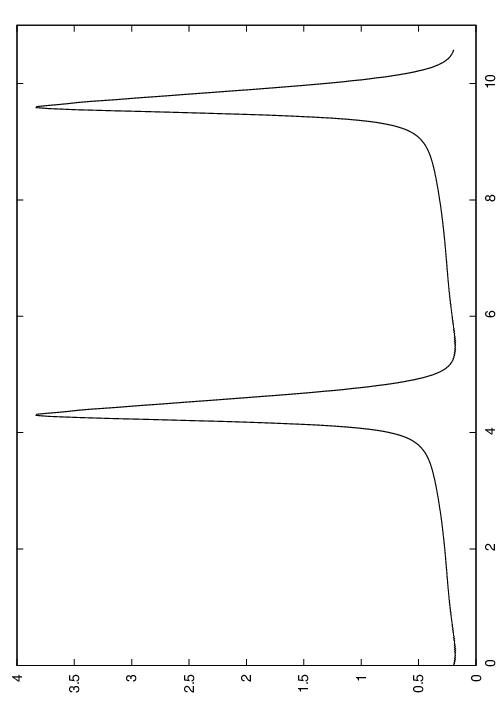} 
\includegraphics[angle=-90,width=\localwidth]{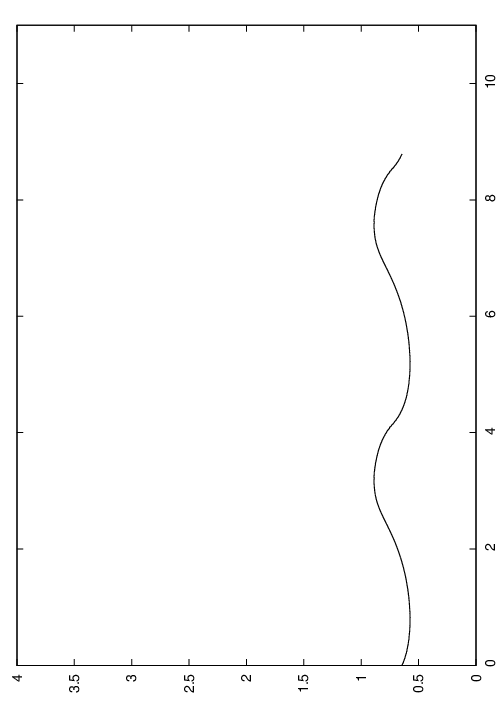}
\includegraphics[angle=-90,width=\localwidth]{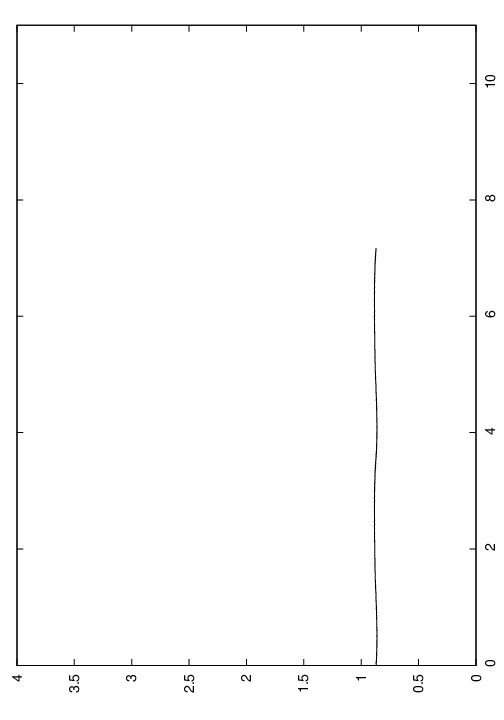}
\caption{(2\,adapt$_{9,4}$)
Plots of the discrete surface mass on $\Gamma^m$ at time $t=12$ for
$\overline\mu_\Gamma = \overline\lambda_\Gamma = 0.01$ (left), 
$\overline\mu_\Gamma = \overline\lambda_\Gamma = 1$ (middle) and 
$\overline\mu_\Gamma = \overline\lambda_\Gamma = 10$ (right). 
Below are plots of the discrete surface mass against arclength.
}
\label{fig:rho1_1d}
\end{figure}%
\begin{figure}
\center
\newcommand\localwidth{0.32\textwidth}
\includegraphics[angle=-0,width=\localwidth]{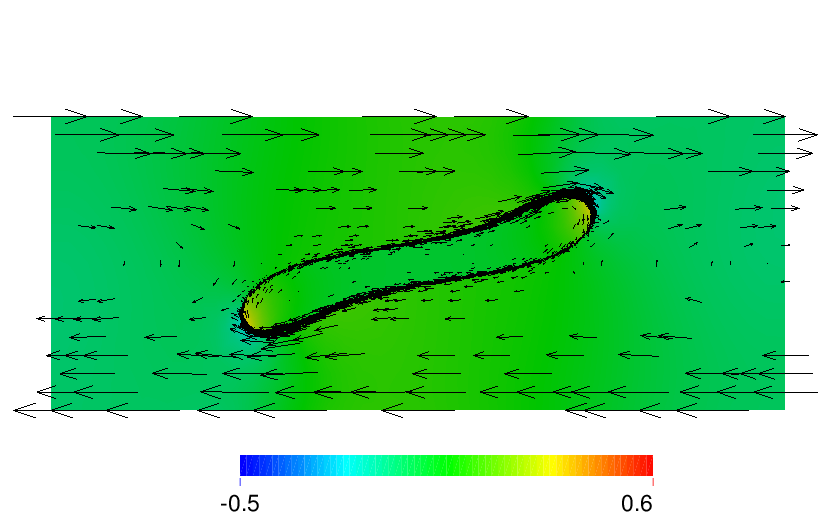}
\includegraphics[angle=-0,width=\localwidth]{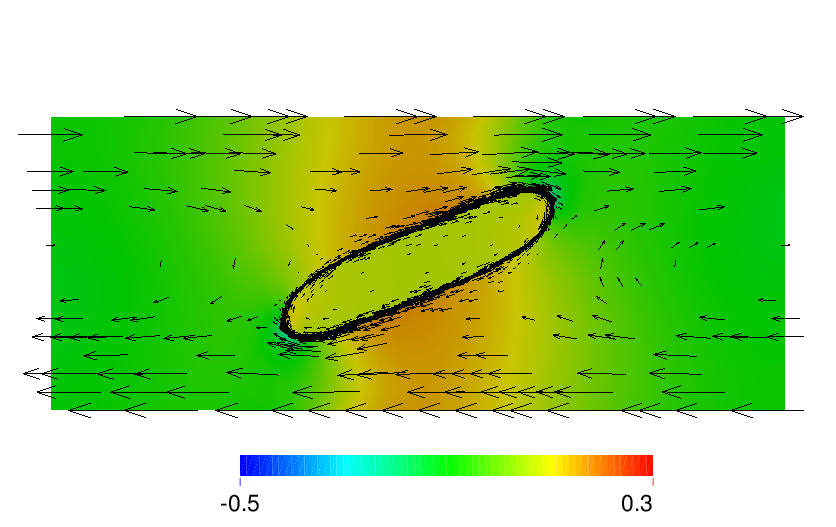}
\includegraphics[angle=-0,width=\localwidth]{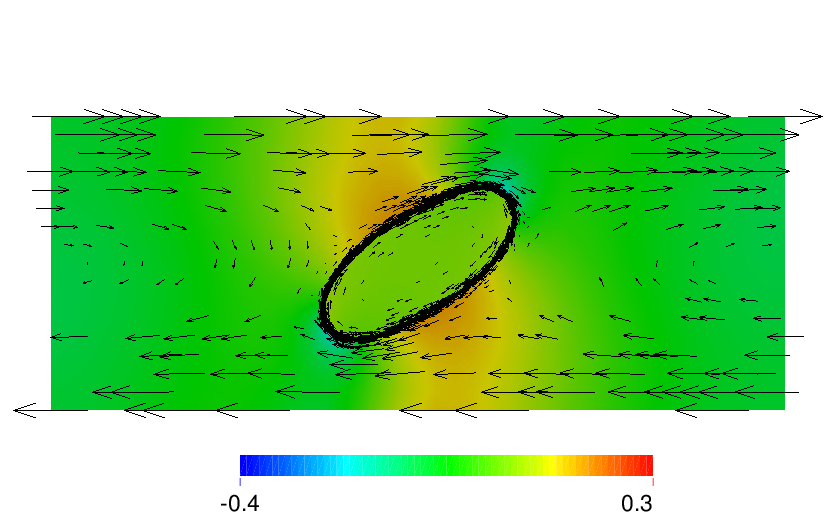}
\caption{(2\,adapt$_{9,4}$)
Velocity fields for the solutions depicted in Figure~\ref{fig:rho1_1d}, 
with the background colouring depending on the pressure values.
}
\label{fig:2dLai_u}
\end{figure}%

For very small values of $\overline\mu_\Gamma = \overline\lambda_\Gamma$ an
interesting effect can be observed. As this value gets smaller, we 
observe a marked concentration of the discrete surface material density 
$\rho_\Gamma^m$ at two points on the interface. This poses a challenge for the
numerical methods, as the peaks in the surface mass density lead to sharp
fronts, which behave almost like a shock. We exhibit the difficulties of the
schemes (\ref{eq:GDa}--f) and (\ref{eq:HGa}--f) with the ``degenerate'' case
$\overline\mu_\Gamma = \overline\lambda_\Gamma=0$ in 
Figure~\ref{fig:degen}.
Clearly, the scheme (\ref{eq:GDa}--f) displays a very nonuniform mesh, 
with some vertices close to coalescence. The latter appears to lead to small
oscillations in $\rho_\Gamma^m$.
The scheme (\ref{eq:HGa}--f), on the other hand, shows very uniform meshes, but
suffers from oscillations in the discrete surface mass density where
$\rho_\Gamma^m$ is close to zero.
On recalling Remark~\ref{rem:oscm}, we note that by adding numerical diffusion
into the scheme, these oscillations can be avoided. This is underlined by the 
numerical results shown in Figure~\ref{fig:degen} for the scheme
(\ref{eq:HGa}--f) with numerical diffusion; $\vartheta(s) = \tfrac s{20}$. 
\begin{figure}
\center
\newcommand\localwidth{0.4\textwidth}
\includegraphics[angle=-90,width=\localwidth]{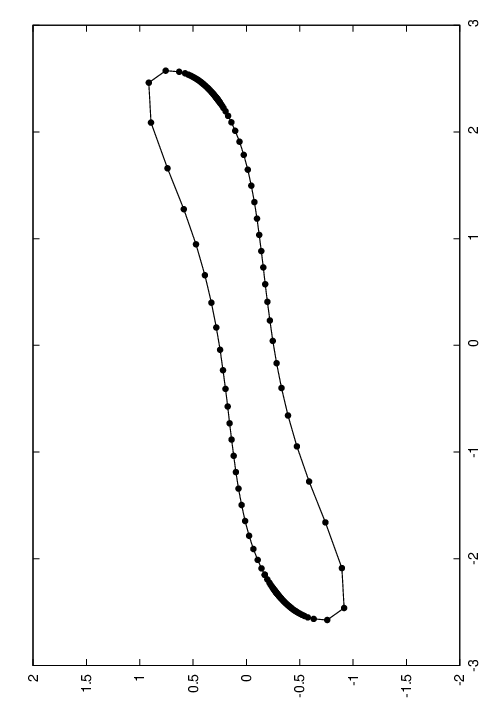}
\includegraphics[angle=-90,width=\localwidth]{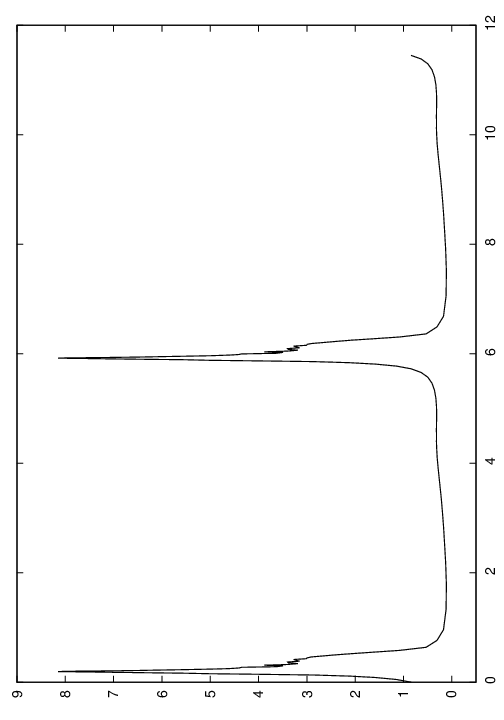}
\includegraphics[angle=-90,width=\localwidth]{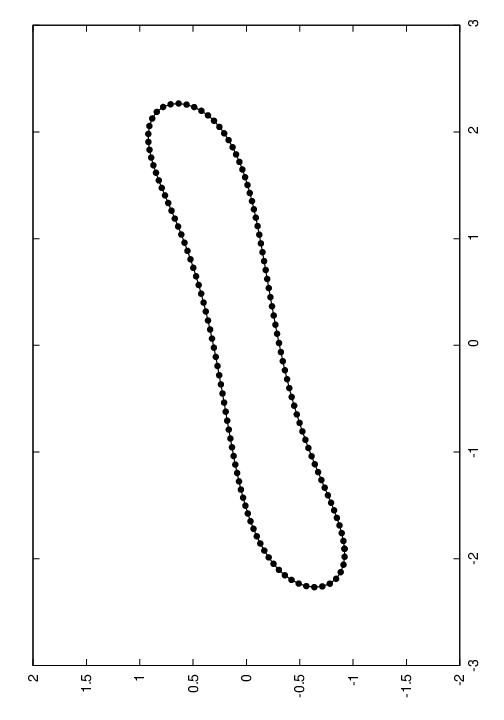}
\includegraphics[angle=-90,width=\localwidth]{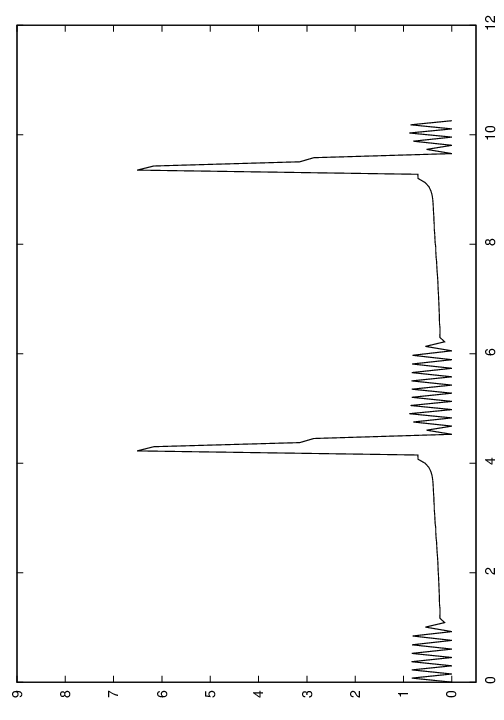}
\includegraphics[angle=-90,width=\localwidth]{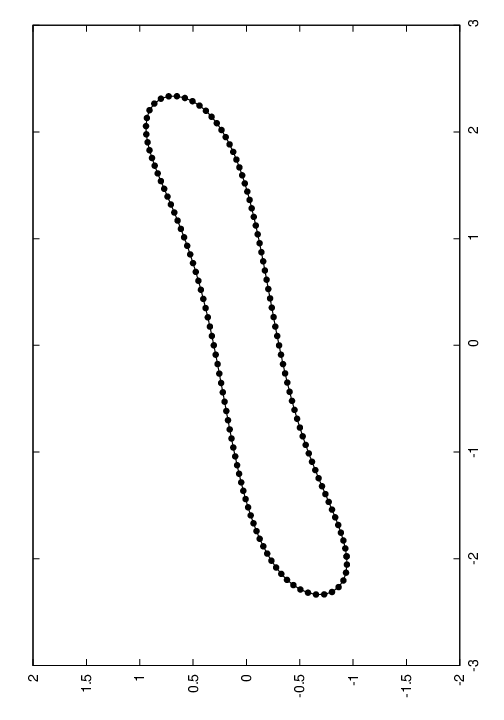}
\includegraphics[angle=-90,width=\localwidth]{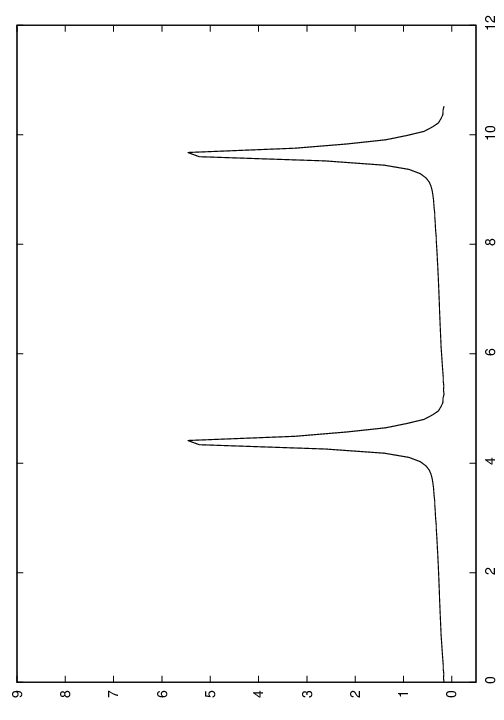}
\caption{(adapt$_{7,3}$)
Plots of $\Gamma^m$ and plots of the discrete surface mass 
against arclength at time $t=12$ for
$\overline\mu_\Gamma = \overline\lambda_\Gamma = 0$ for the schemes
(\ref{eq:GDa}--f), top,
(\ref{eq:HGa}--f), middle, and
(\ref{eq:HGa}--f) with numerical diffusion; $\vartheta(s) =
\tfrac s{20}$, bottom.
}
\label{fig:degen}
\end{figure}%

{From} a physical point of view it is not easy to explain the fact that the
surface mass accumulates at two points on the interface. However, we recall
from Theorem~\ref{thm:stabHG} that such a relocation of mass on the
interface leads to a smaller overall energy, if the discrete velocity 
$\vec U^{m}$ at these points is zero, or nearly zero. In fact, this is what
appears to happen for $\overline\mu_\Gamma = \overline\lambda_\Gamma = 0$, as
can be seen from the velocity plot in Figure~\ref{fig:HGdegenu}.
\begin{figure}
\center
\newcommand\localwidth{0.6\textwidth}
\includegraphics[angle=-0,width=\localwidth]{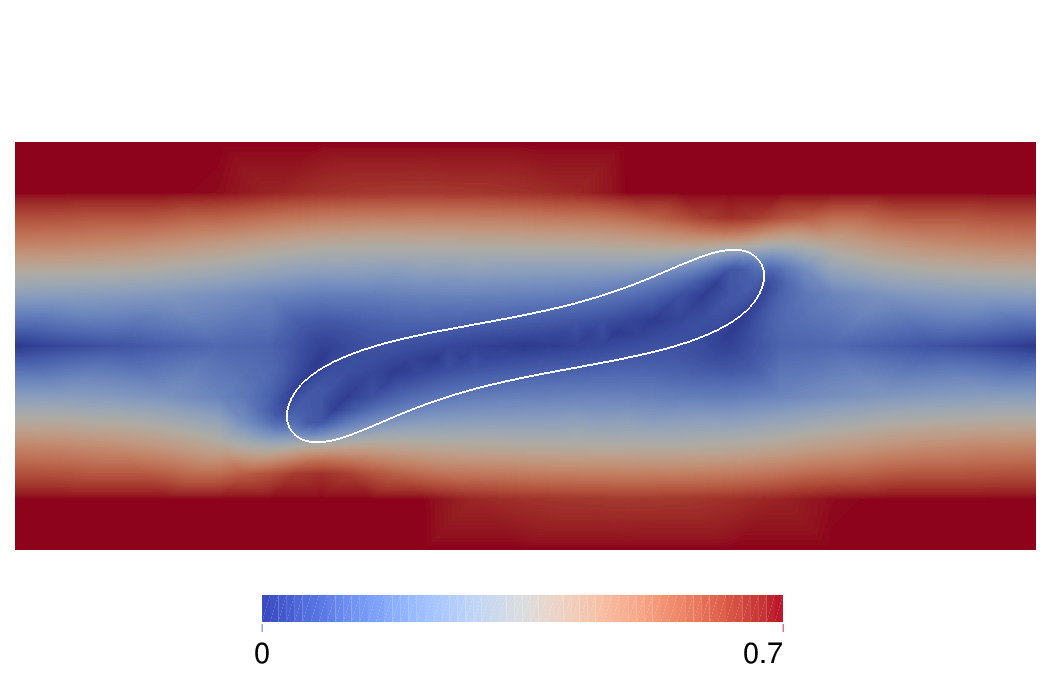}
\caption{(adapt$_{7,3}$)
A plot of $|\vec U^m|$ at time $t=12$ for
$\overline\mu_\Gamma = \overline\lambda_\Gamma = 0$
for the scheme (\ref{eq:HGa}--f) with numerical diffusion; $\vartheta(s) =
\tfrac s{20}$.}
\label{fig:HGdegenu}
\end{figure}%

In the next simulation we consider the presence of surfactant on the
interface. To this end, we choose the linear equation of state 
(\ref{eq:gamma1}) with $\beta = 0.5$ and let
\begin{equation} \label{eq:mllin}
\mu_\Gamma(r) = \overline\mu_\Gamma\,(1 + b_\mu\,[r]_+)
\quad\text{and}\quad
\lambda_\Gamma(r) = \overline\lambda_\Gamma\,(1 + b_\lambda\,[r]_+)
\quad \forall\ r \in \R\,,
\end{equation}
where $\overline\mu_\Gamma = \overline\lambda_\Gamma = 0.1$ and
$b_\mu = b_\lambda = 100$, 
with the remaining parameters as in (\ref{eq:Lai}).
We also let $\rho_{\Gamma,0}=1$, while the initial distribution of
surfactant on $\Gamma(0)$ is chosen as
\begin{equation} \label{eq:psi0}
\psi_0(\vec z) = 10^{-6} + [z_1]_+\,.
\end{equation}
The evolutions of the approximations of $\psi$ and $\rho_\Gamma$ can be seen in
Figure~\ref{fig:2dLai_psi05}. The initially onesided distribution of
surfactant, together with the definitions (\ref{eq:mllin}), leads to the bubble
moving significantly to the right. The higher concentration of surfactant on 
the right leads to surface tension gradients on the interface, which then 
cause tangential shear stresses on the interface. These so called Marangoni 
forces lead to the overall movement of the drop to the right.
Varying the value of $\beta$ between $0$ and $1$
had no significant effect on the overall evolution, and so we omit further
numerical results for this setting.
\begin{figure}
\center
\newcommand\localwidth{0.24\textwidth}
\includegraphics[angle=-90,width=\localwidth]{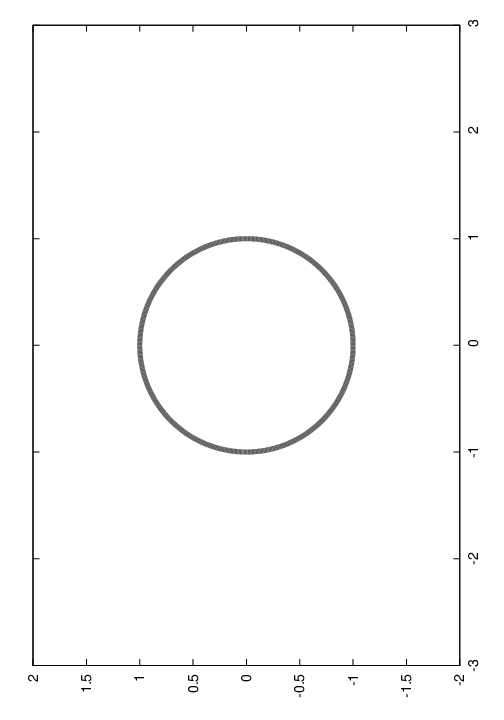}
\includegraphics[angle=-90,width=\localwidth]{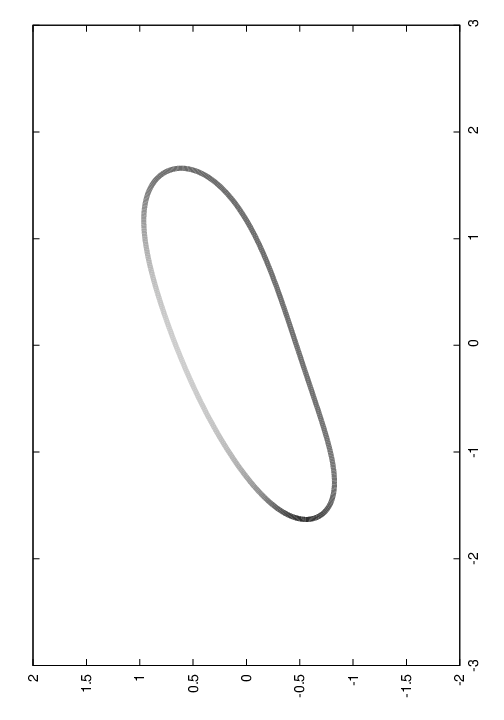}
\includegraphics[angle=-90,width=\localwidth]{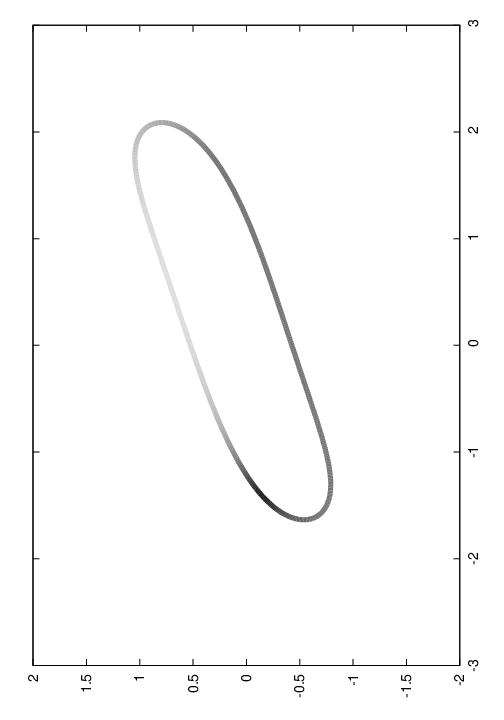}
\includegraphics[angle=-90,width=\localwidth]{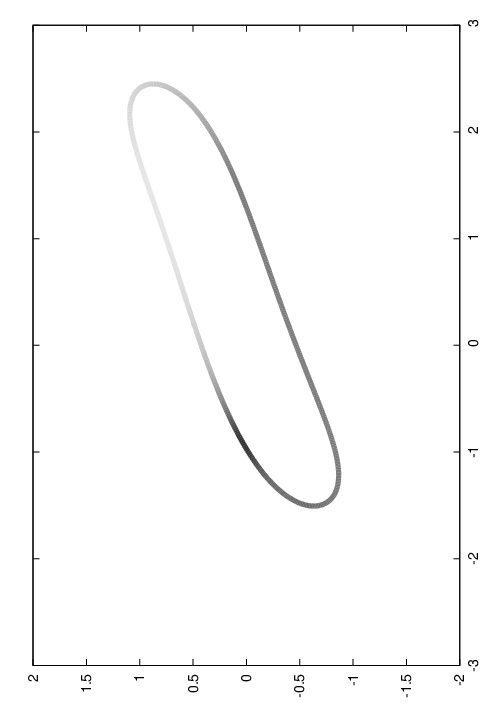}
\includegraphics[angle=-90,width=\localwidth]{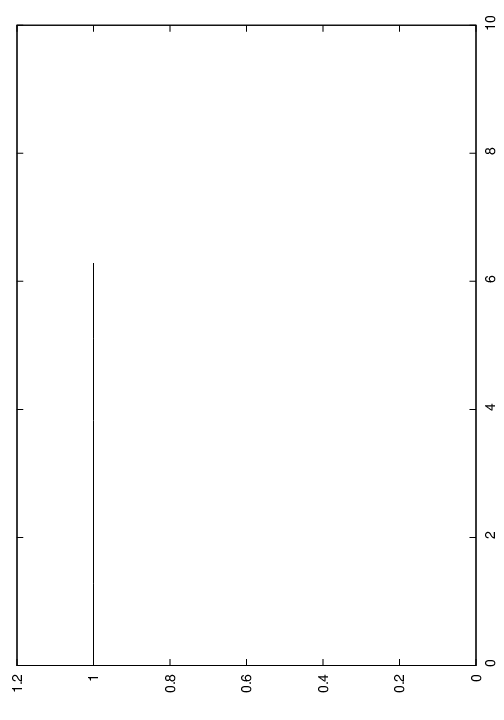}
\includegraphics[angle=-90,width=\localwidth]{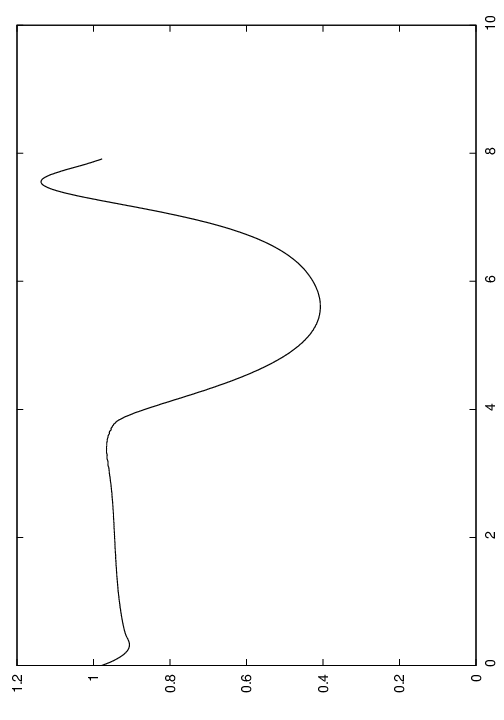}
\includegraphics[angle=-90,width=\localwidth]{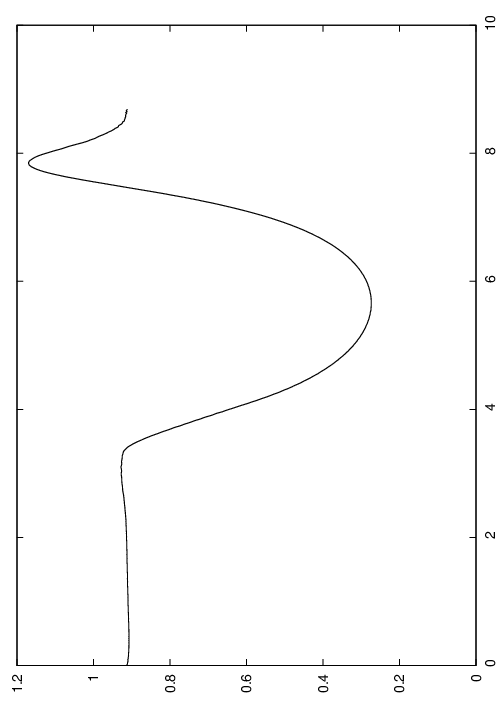}
\includegraphics[angle=-90,width=\localwidth]{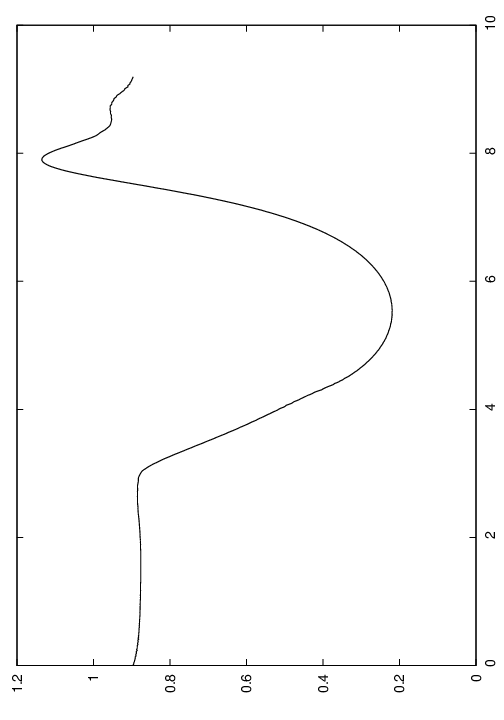}
\includegraphics[angle=-90,width=\localwidth]{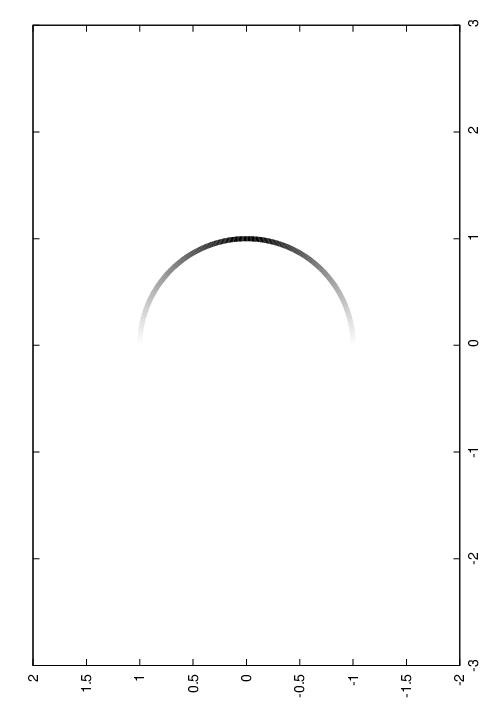}
\includegraphics[angle=-90,width=\localwidth]{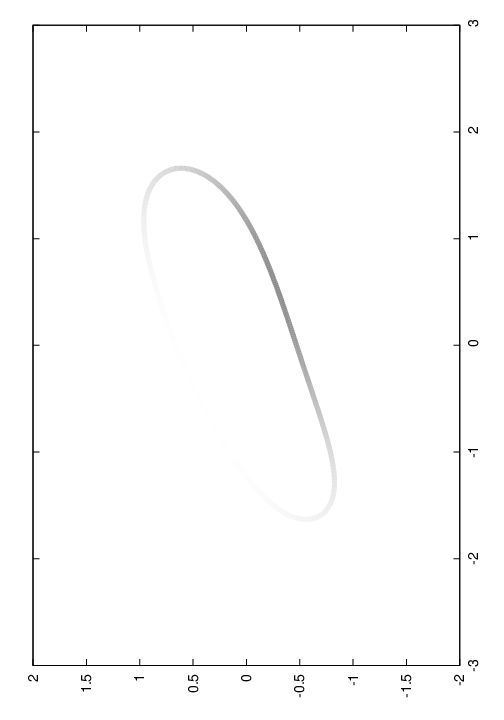}
\includegraphics[angle=-90,width=\localwidth]{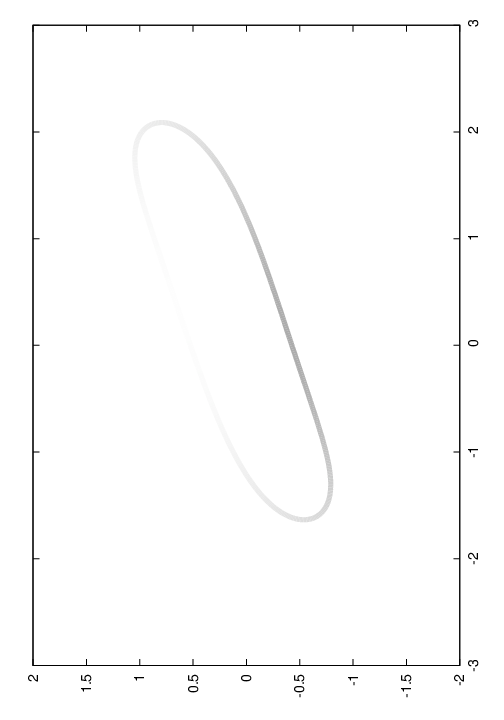}
\includegraphics[angle=-90,width=\localwidth]{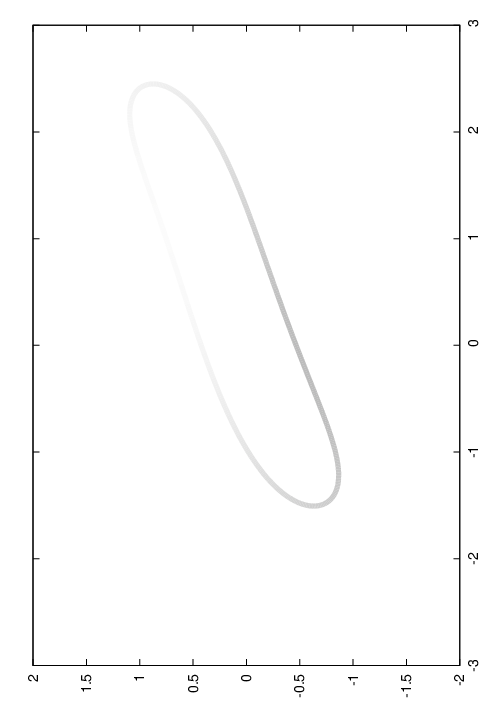}
\includegraphics[angle=-90,width=\localwidth]{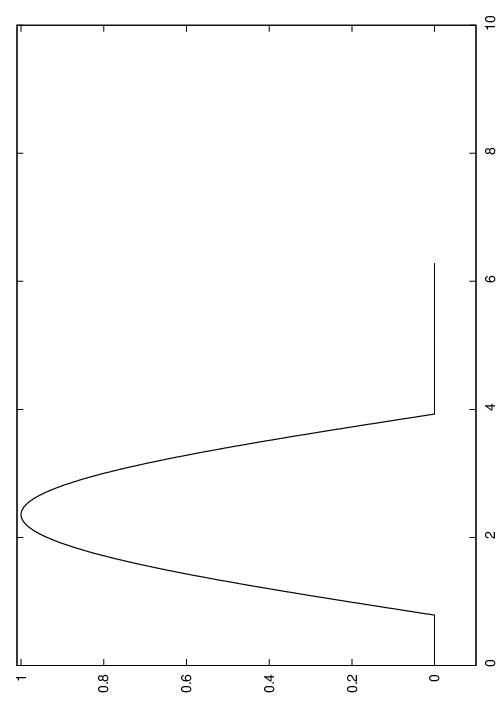}
\includegraphics[angle=-90,width=\localwidth]{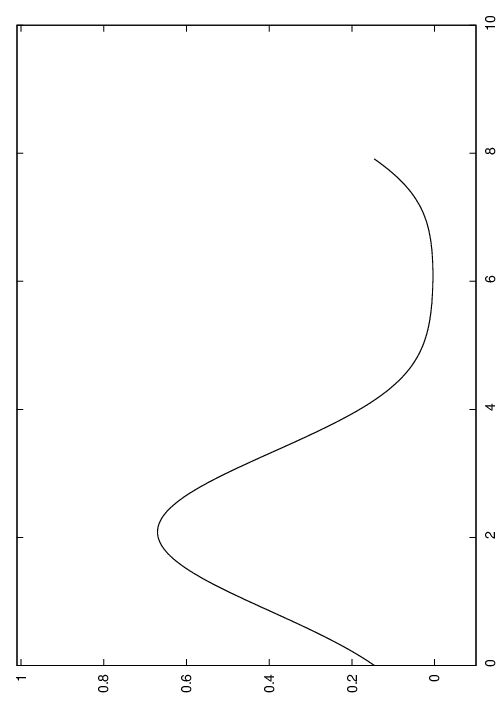}
\includegraphics[angle=-90,width=\localwidth]{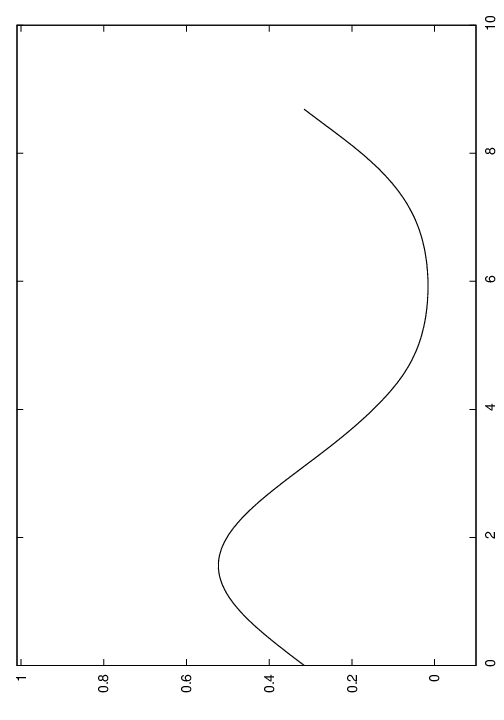}
\includegraphics[angle=-90,width=\localwidth]{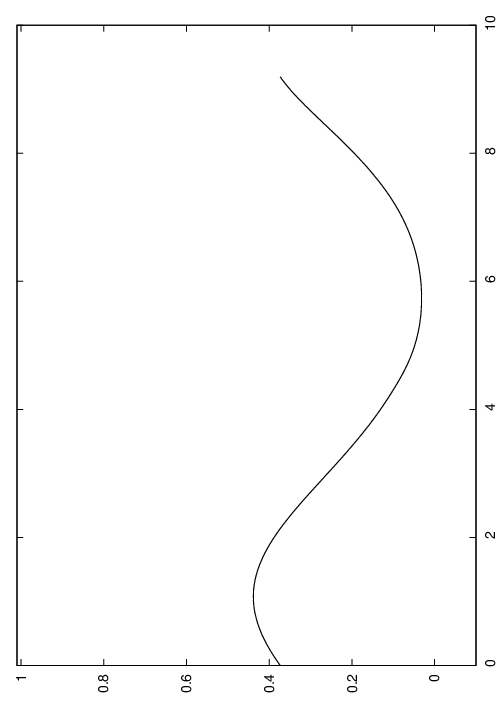}
\caption{(2\,adapt$_{9,4}$)
The time evolution of a drop in shear flow with (\ref{eq:gamma1}) and 
$\beta = 0.5$ for the scheme (\ref{eq:HGa}--f). 
The top two rows show the evolution of the discrete surface
material density, while the lower two rows show the evolution of the discrete
surfactant concentration.
Plots are at times $t=0,\,4,\,8,\,12$. 
In the first row the grey scales linearly with the surface material density 
ranging from 0 (white) to 1.4 (black).
In the third row the grey scales linearly with the surfactant
concentration ranging from 0 (white) to 1 (black).
}
\label{fig:2dLai_psi05}
\end{figure}%

\subsubsection{Rising bubble} \label{sec:612}
In this subsection we compare the schemes (\ref{eq:GDa}--f) and 
(\ref{eq:HGa}--f) for a rising bubble
experiment that is motivated by the benchmark problems in \cite{HysingTKPBGT09}
for two-phase Navier--Stokes flow.
In particular, we use the setup described in \cite{HysingTKPBGT09}, 
see Figure~2 there; i.e.\ $\Omega = (0,1) \times (0,2)$ with 
$\partial_1\Omega = [0,1] \times \{0,2\}$ and 
$\partial_2\Omega = \{0,1\} \times (0,2)$.
Moreover, $\Gamma_0 = \{\vec z \in \R^2 : |\vec z - (\frac12, \frac12)^T| =
\frac14\}$.
The physical parameters from the test case 1 in 
\citet[Table~I]{HysingTKPBGT09} are given by
\begin{equation} \label{eq:Hysing1}
\rho_+ = 1000\,,\quad \rho_- = 100\,,\quad \mu_+ = 10\,,\quad \mu_- = 1\,,\quad
\gamma_0 = 24.5\,,\quad \vec f_1 = -0.98\,\vec\ek_d\,,\quad 
\vec f_2 = \vec 0\,,
\end{equation}
where, here and throughout, 
$\{\vec \ek_j\}_{j=1}^d$ denotes the standard basis in $\R^d$.
For the surfactant problem we choose the parameters $\Ds = 0.1$ and 
(\ref{eq:gamma1}) with $\beta = 0.5$. For the surface material parameters
we choose $\overline\mu_\Gamma = \overline\lambda_\Gamma = 0.1$ and
$\rho_{\Gamma,0}=1$. We refer to our recent papers \cite{fluidfbp,tpfs} for
numerical simulations for this benchmark problem in the absence of a
Boussinesq--Scriven surface fluid.

We start with a simulation for the scheme (\ref{eq:GDa}--f), using 
the discretization parameters adapt$_{7, 3}$.
The results can be seen on the left of Figure~\ref{fig:dziukbgn}.
We see that the vertices of the approximation
$\Gamma^m$ are transported with the fluid flow.
This means that many vertices can
be found at the bottom of the bubble, with hardly any vertices left at the top.
\begin{figure}
\center
\includegraphics[angle=-90,width=0.45\textwidth]{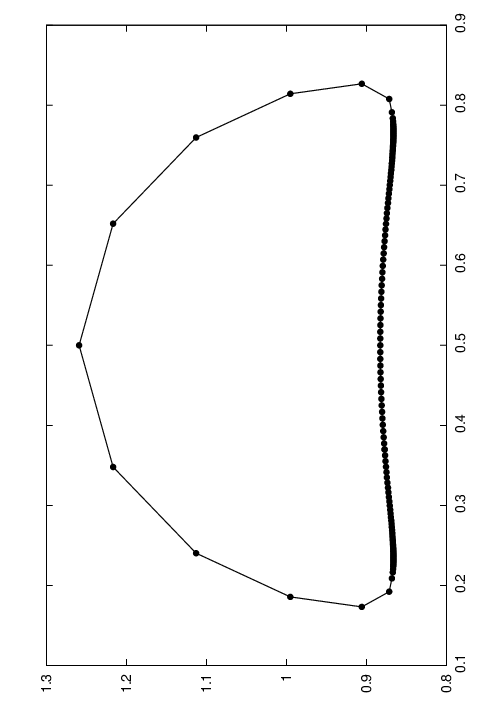}
\includegraphics[angle=-90,width=0.45\textwidth]{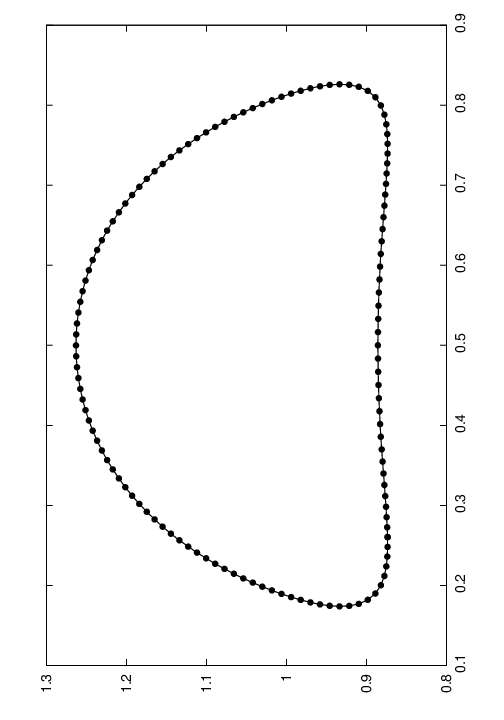}
\mbox{
\includegraphics[angle=-90,width=0.22\textwidth]{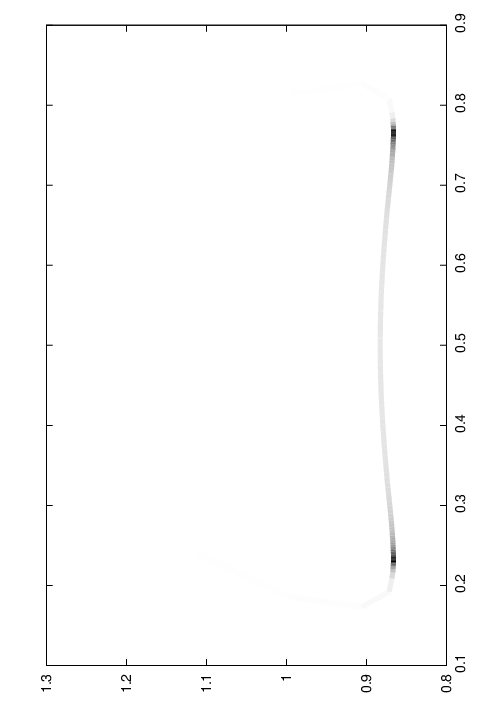}
\includegraphics[angle=-90,width=0.22\textwidth]{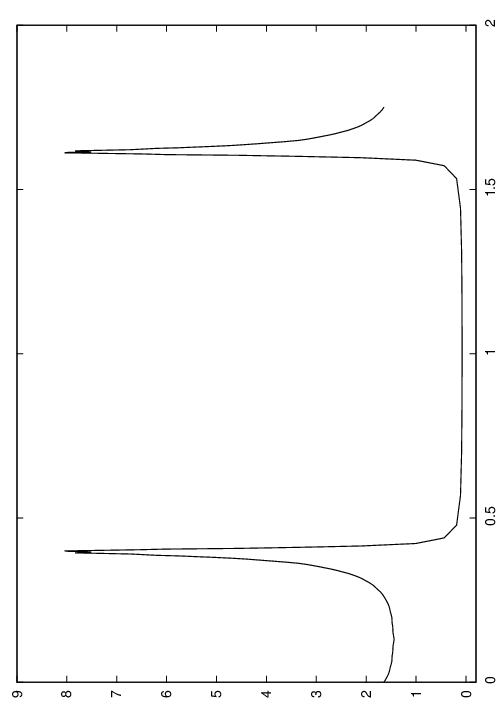}
\includegraphics[angle=-90,width=0.22\textwidth]{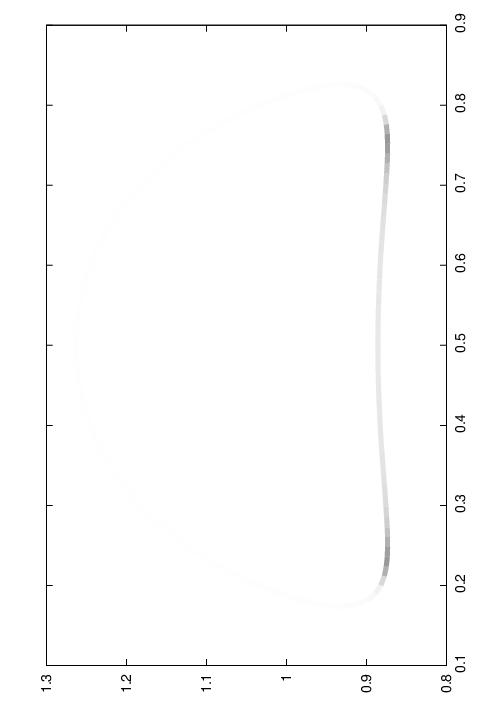}
\includegraphics[angle=-90,width=0.22\textwidth]{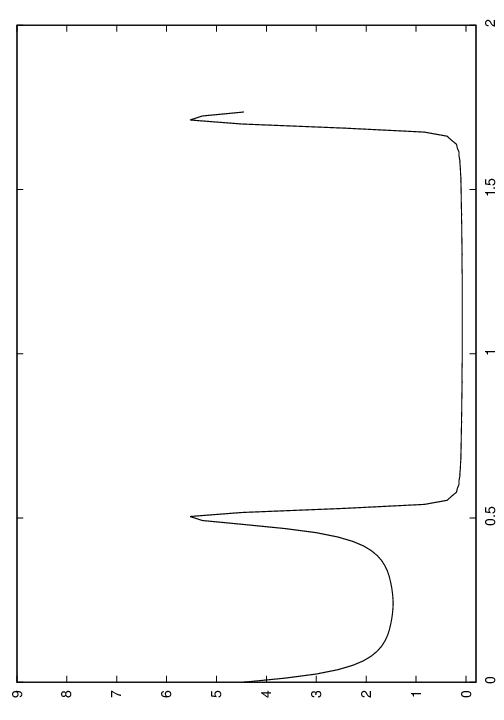}
}
\mbox{
\includegraphics[angle=-90,width=0.22\textwidth]{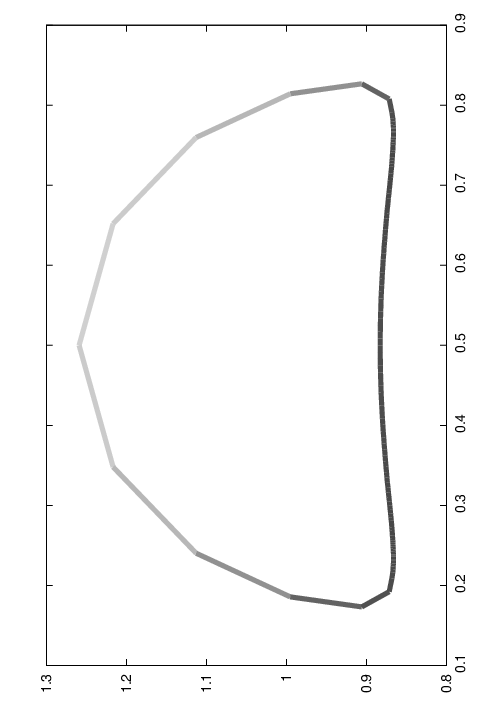}
\includegraphics[angle=-90,width=0.22\textwidth]{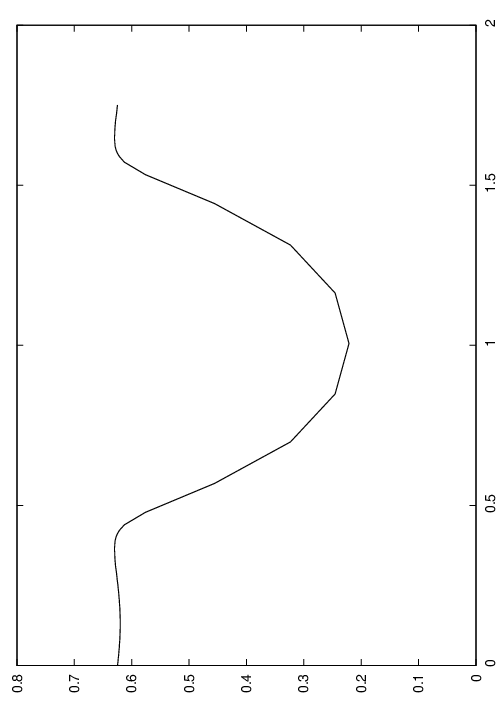}
\includegraphics[angle=-90,width=0.22\textwidth]{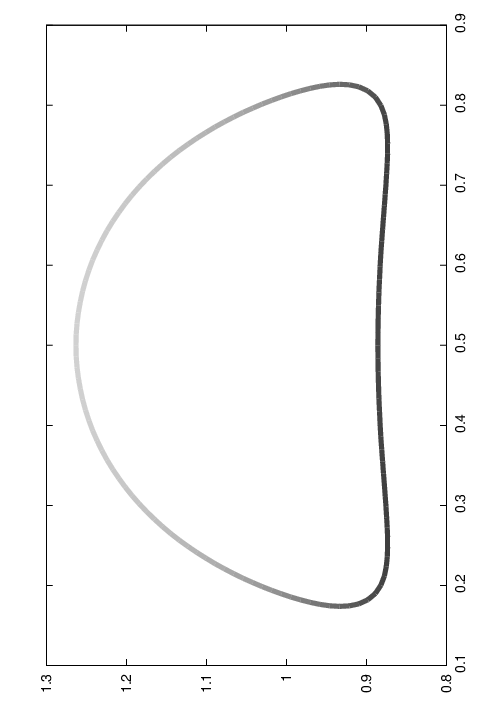}
\includegraphics[angle=-90,width=0.22\textwidth]{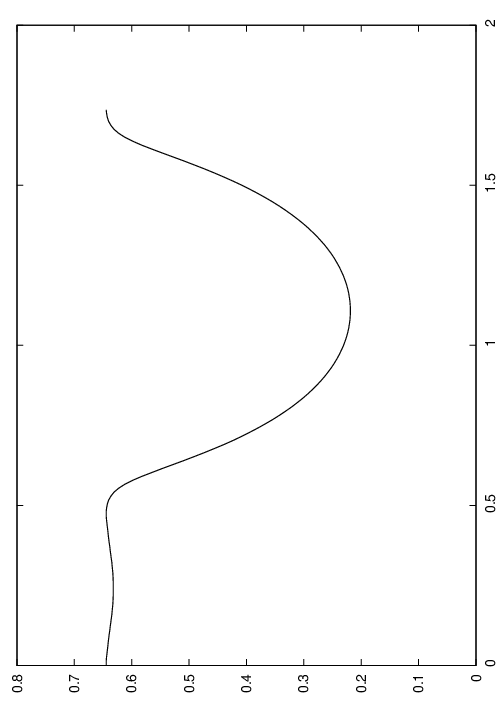}
}
\caption{(adapt$_{7,3}$) 
Vertex distributions for the final bubble at time 
$t=3$ for the schemes (\ref{eq:GDa}--f), left, and (\ref{eq:HGa}--f), right.
The latter scheme uses numerical diffusion with 
$\vartheta(s) = \tfrac s{20}$.
The middle row shows the discrete surface material densities, while the bottom
row shows the discrete surfactant concentrations.
In the former the grey scales linearly with the surface material density 
ranging from 0 (white) to 9 (black), while in the latter
the grey scales linearly with the surfactant concentration ranging from 
0 (white) to 0.7 (black).
}
\label{fig:dziukbgn}
\end{figure}%
We also remark that for this computation the area of the inner phase 
decreases by $1.3\%$, so the volume of the two phases is not preserved.
The same computation for our preferred scheme (\ref{eq:HGa}--f), where 
the tangential movement of vertices yields an almost equidistributed 
approximation of $\Gamma^m$, can be seen on the right of
Figure~\ref{fig:dziukbgn}. In order to avoid oscillations in 
$\rho^m_\Gamma$ close to zero, we use numerical diffusion with 
$\vartheta(s) = \tfrac s{20}$ for this numerical experiment.
We remark that for this computation the areas of the two phases,
as well as the total surfactant amount and the total surface mass on 
$\Gamma^m$, were conserved. 

In view of the superior mesh properties of our preferred scheme 
(\ref{eq:HGa}--f), from now on we only consider numerical experiments for the 
scheme (\ref{eq:HGa}--f). 

\subsection{Numerical experiments in 3d}

In this subsection we present numerical results for $d=3$ for our preferred
scheme \mbox{(\ref{eq:HGa}--f)}. As discretization parameters we always choose 
$\frac1{10}\,{\rm adapt}_{5,2}$.

\subsubsection{Bubble in shear flow}
In this subsection we report on some 3d analogues of the computations in
\S\ref{sec:611}.
In particular, we perform shear flow experiments
on the domain $\Omega = (-5,5)\times (-2,2)^2$ with
$\partial\Omega=\partial_1\Omega$ and $\vec g(\vec z) = (\frac12\,z_3,0,0)^T$.
The physical parameters are as in (\ref{eq:Lai}), and for simplicity 
we take
$\rho_{\Gamma,0}=0$. See Figure~\ref{fig:3drho0} for the final bubble shapes
for a selection of parameters $\overline\mu_\Gamma$ and 
$\overline\lambda_\Gamma$ in (\ref{eq:muconst}). 
\begin{figure}
\center
\newcommand\localwidth{0.45\textwidth}
\includegraphics[angle=-90,width=\localwidth]{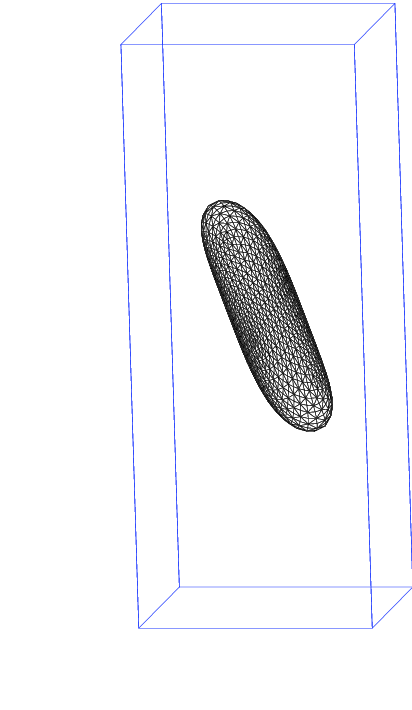}
\includegraphics[angle=-90,width=\localwidth]{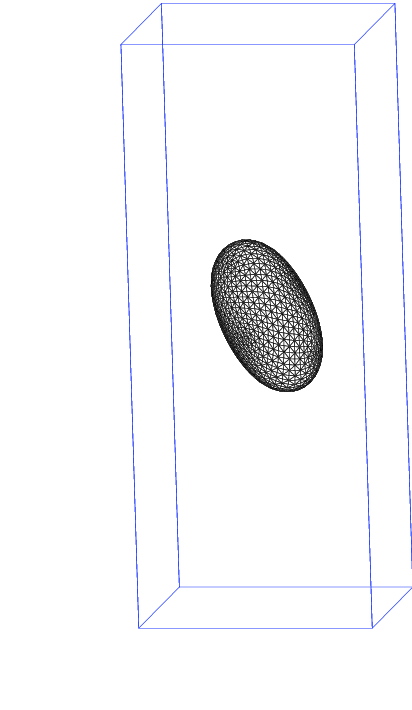}
\includegraphics[angle=-90,width=\localwidth]{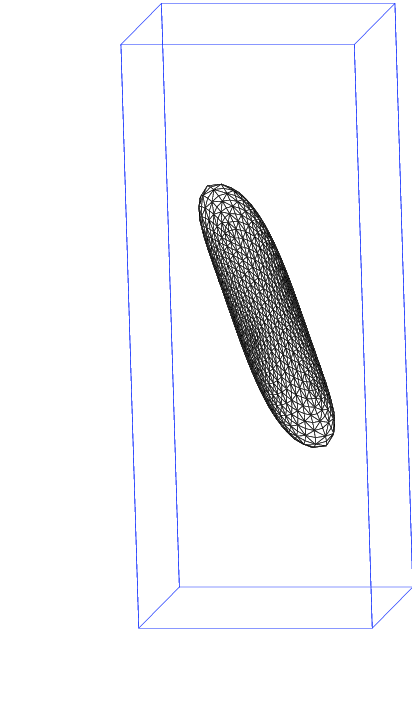}
\includegraphics[angle=-90,width=\localwidth]{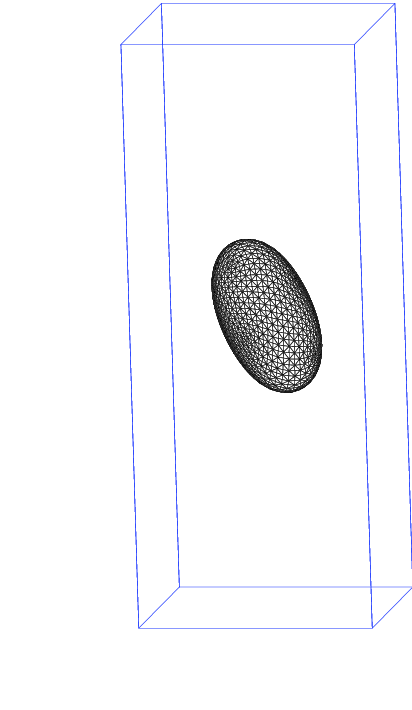}
\caption{
The discrete interface $\Gamma^m$ at time $t=12$ for 
a drop in shear flow with $\rho_{\Gamma,0}=0$
for (\ref{eq:Fconst}) and (\ref{eq:muconst}) with 
$\overline\mu_\Gamma = \overline\lambda_\Gamma = 0$, 
$\overline\mu_\Gamma = 1$, $\overline\lambda_\Gamma = 0$, 
$\overline\mu_\Gamma = \overline\lambda_\Gamma = 1$,
$\overline\mu_\Gamma = 0$, $\overline\lambda_\Gamma = 1$
(clockwise from top left). 
}
\label{fig:3drho0}
\end{figure}%

\subsubsection{Rising bubble} \label{sec:622}
Here we consider the natural 3d analogue of the problem in \S\ref{sec:612}.
To this end, we let $\Omega =
(0,1) \times (0,1) \times (0.2)$ with 
$\partial_1\Omega = [0,1] \times [0,1] \times \{0,2\}$ and 
$\partial_2\Omega = \partial\Omega \setminus \partial_1\Omega$.
Moreover, we set $T=3$, $\Gamma_0 = \{ \vec z \in \R^3 : |\vec z -
(\frac12, \frac12, \frac12)^T| = \frac14\}$, and choose all the remaining
parameters as in \S\ref{sec:612}; recall e.g.\ (\ref{eq:Hysing1}).
As in the 2d equivalent, the bubble rises due to density difference against the
direction of gravity. In the process, the surfactant and the surface mass
accumulate at the bottom of the bubble. We show the concentrations of these two
quantities in Figure~\ref{fig:3dbubble}.
\begin{figure}
\center
\includegraphics[angle=-0,width=0.4\textwidth]{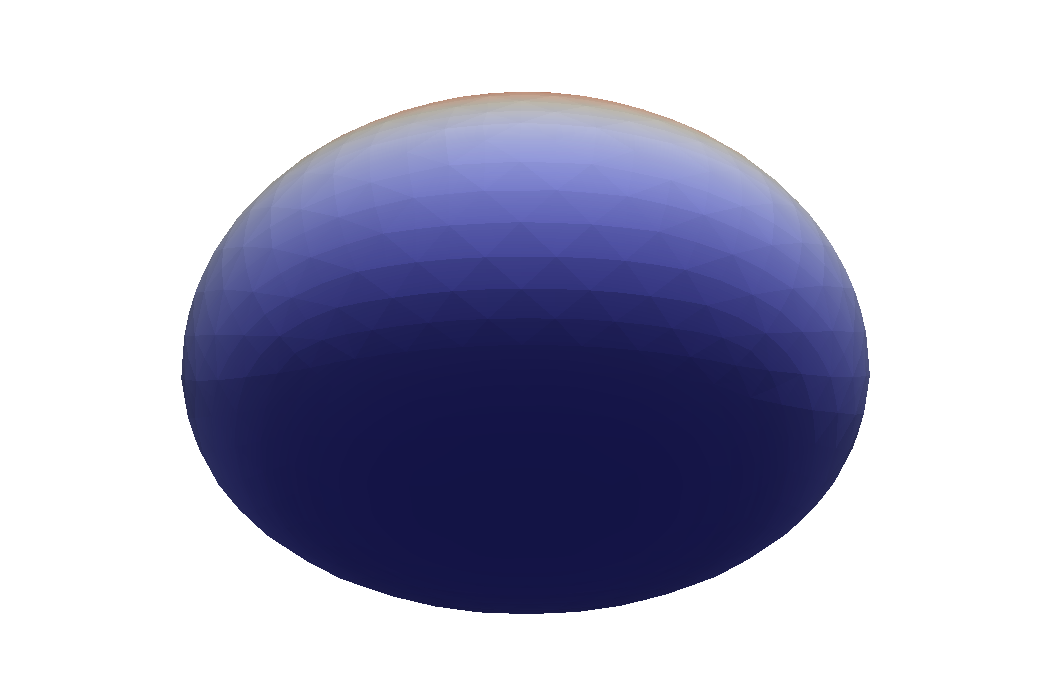}
\qquad\qquad
\includegraphics[angle=-0,width=0.4\textwidth]{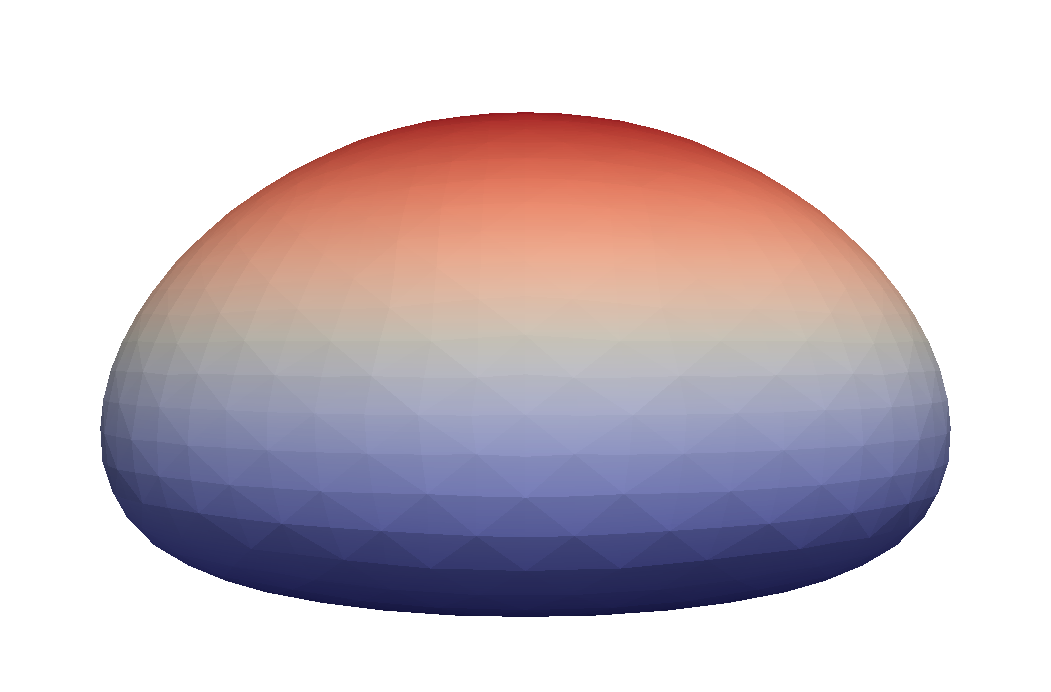}
\includegraphics[angle=-0,width=0.4\textwidth]{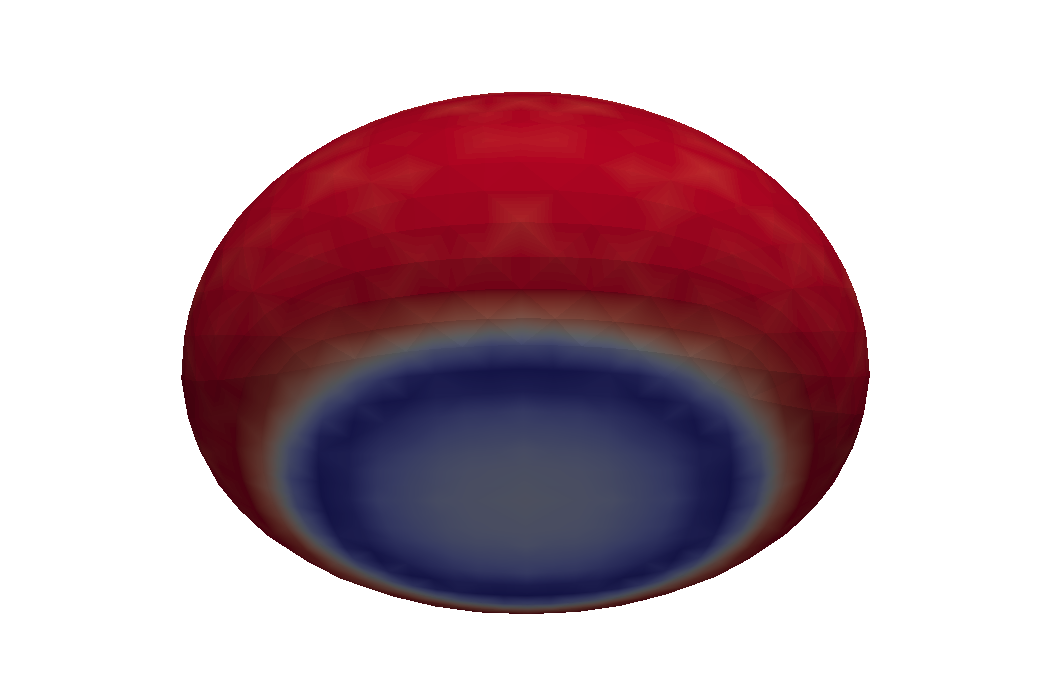}
\qquad\qquad
\includegraphics[angle=-0,width=0.4\textwidth]{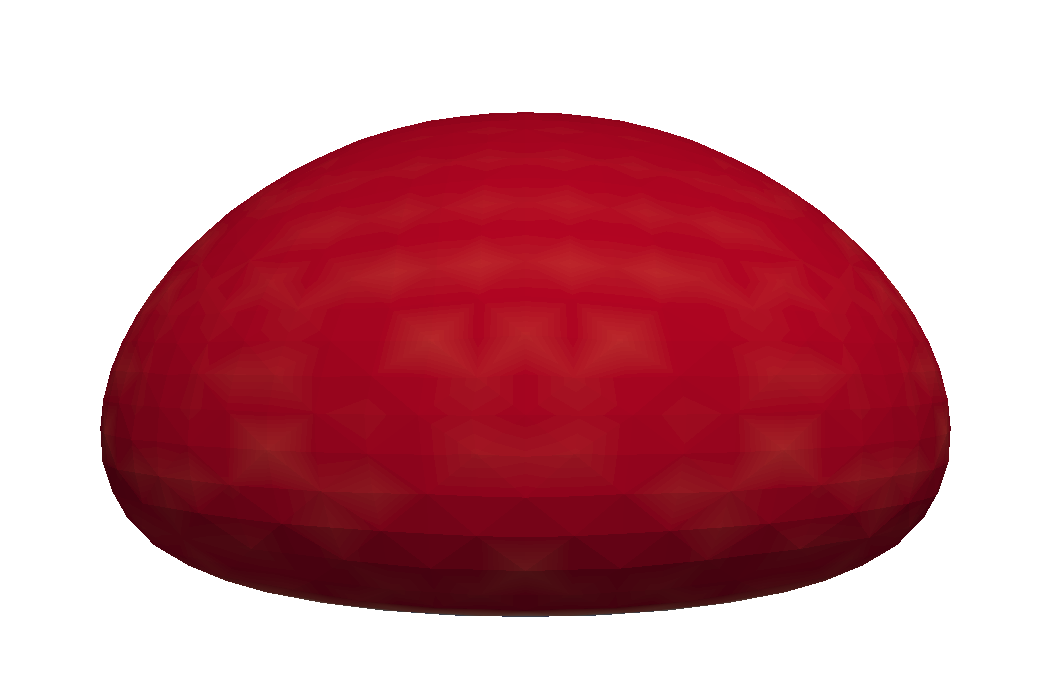}
\caption{
The surfactant concentration $\Psi^m$ and the surface mass $\rho^m_\Gamma$ 
on $\Gamma^m$ at
time $t=3$. The top row shows $\Psi^m$, with the colour ranging from red 
(0.3) to blue (0.6). The bottom row shows 
$\rho^m_\Gamma$, with the colour ranging from red (0) to blue (6.5).
}
\label{fig:3dbubble}
\end{figure}%

\def\soft#1{\leavevmode\setbox0=\hbox{h}\dimen7=\ht0\advance \dimen7
  by-1ex\relax\if t#1\relax\rlap{\raise.6\dimen7
  \hbox{\kern.3ex\char'47}}#1\relax\else\if T#1\relax
  \rlap{\raise.5\dimen7\hbox{\kern1.3ex\char'47}}#1\relax \else\if
  d#1\relax\rlap{\raise.5\dimen7\hbox{\kern.9ex \char'47}}#1\relax\else\if
  D#1\relax\rlap{\raise.5\dimen7 \hbox{\kern1.4ex\char'47}}#1\relax\else\if
  l#1\relax \rlap{\raise.5\dimen7\hbox{\kern.4ex\char'47}}#1\relax \else\if
  L#1\relax\rlap{\raise.5\dimen7\hbox{\kern.7ex
  \char'47}}#1\relax\else\message{accent \string\soft \space #1 not
  defined!}#1\relax\fi\fi\fi\fi\fi\fi}

\end{document}